\documentclass[11pt,twoside]{article}

\usepackage{amsmath,amsthm,amsfonts,amssymb,mathtools,mathdots,scalerel,mathrsfs,stmaryrd,cmll}
\usepackage[utf8]{inputenc}
\usepackage[T1]{fontenc}
\usepackage{mlmodern, tikz-cd, quiver}
\usepackage[many]{tcolorbox}

\usepackage[shortlabels]{enumitem}
\usepackage[hidelinks]{hyperref}
\usepackage{tikz, fancyhdr, xparse, xcolor, tocloft}
\usepackage[british]{babel}
\usepackage[a4paper, top=4.5cm, bottom=4.5cm, left=4cm, right=4cm, asymmetric]{geometry}

\usepackage{crossreftools}
\usepackage[textwidth=3cm, textsize=small, colorinlistoftodos]{todonotes}

\usepackage{cctitle, titlesec}
\usepackage{etoolbox}

\makeatletter
\patchcmd{\ttlh@hang}{\parindent\z@}{\parindent\z@\leavevmode}{}{}
\patchcmd{\ttlh@hang}{\noindent}{}{}{}
\makeatother

\newcommand\eqdef{\coloneqq}
\newcommand\nbd{\nobreakdash-\hspace{0pt}}
\newcommand\idd[1]{\mathrm{id}_{#1}}
\newcommand\invrs[1]{#1^{-1}}
\newcommand\after{\circ}
\newcommand\incl{\hookrightarrow}
\newcommand\surj{\twoheadrightarrow}
\newcommand\restr[2]{{#1}{\raisebox{0pt}{$|_{#2}$}}}
\newcommand\sd{\looparrowright}
\newcommand\set[1]{\left\{ {#1} \right\}}
\newcommand\order[2]{#2^{(#1)}}
\renewcommand{\a}{\alpha}

\newcommand\opp[1]{#1^\mathrm{op}}
\newcommand\cat[1]{\mathbf{#1}}
\newcommand\fun[1]{\mathsf{#1}}
\DeclareMathOperator*{\colim}{colim}
\def\raiseslice#1#2{\raisebox{-2pt}{$#1#2$}}
\newcommand{\slice}[2]{{#1{/}\mathpalette\raiseslice{#2}}}

\newcommand\rdcpx{\cat{RDCpx}}
\newcommand\rdcpxmap{\rdcpx_{\downarrow}}
\newcommand\rdcpxcomap{\rdcpx_{\uparrow}}
\newcommand\rdcpxcart{\rdcpx_{\mathsf{cart}}}
\newcommand{\atom}{{\scalebox{1.3}{\( \odot \)}}}

\newcommand{\Set}{\cat{Set}}
\newcommand{\Cat}{\cat{Cat}}
\newcommand{\Pos}{\cat{Pos}}
\newcommand{\dgmSet}{\atom\Set}
\newcommand{\sSet}{\cat{sSet}}
\newcommand{\pt}{\mathbf{1}}
\newcommand{\bpt}[1]{#1_{**}}

\DeclareMathOperator{\clos}{cl}
\newcommand\clset[1]{\mathrm{cl}\set{#1}}
\newcommand\gr[2]{#2_{#1}}
\newcommand\maxel[1]{\mathscr{M}\!\mathit{ax}\,#1}
\newcommand\bd[2]{\partial_{#1}^{#2}}
\newcommand\faces[2]{\Delta_{#1}^{#2}}
\newcommand\cofaces[2]{\nabla_{#1}^{#2}}
\newcommand\cp[1]{\,{\scriptstyle\#}_{#1}\,}
\newcommand\cpsub[1]{\triangleright_{#1}}
\newcommand\subcp[1]{\prescript{}{#1\!}{\triangleleft}\;}
\newcommand\subs[3]{#1[#2/#3]}
\newcommand\compos[1]{\langle#1\rangle}
\newcommand\celto{\Rightarrow}
\newcommand\submol{\sqsubseteq}
\newcommand\augm[1]{{{#1}_\bot}}
\newcommand\dual[2]{\fun{D}_{#1}{#2}}
\newcommand\gray{\otimes}
\newcommand\sus[1]{\fun{S}{#1}}
\newcommand\pcyl[1]{\ltimes_{#1}}

\DeclareMathOperator{\Pd}{Pd}
\DeclareMathOperator{\Rd}{Rd}
\DeclareMathOperator{\Eqv}{Eqv}
\DeclareMathOperator{\cell}{cell}
\DeclareMathOperator{\eqv}{eqv}
\DeclareMathOperator{\Dgn}{Dgn}
\DeclareMathOperator{\dgn}{dgn}
\DeclareMathOperator{\Nd}{Nd}
\DeclareMathOperator{\nd}{nd}
\newcommand\rev[1]{{#1}^\dag}

\newcommand{\un}{\varepsilon}
\newcommand{\hinv}[1]{\xi_{#1}}
\newcommand{\lcyl}[1]{\mathrm{L}_{#1}}
\newcommand{\rcyl}[1]{\mathrm{R}_{#1}}
\newcommand{\hcyl}[1]{\Xi_{#1}}

\newcommand{\dglobe}[1]{\vec{O}^{#1}}
\newcommand\arr{\dglobe{1}}
\newcommand{\rglobe}[1]{\selfloc{O}^{#1}}

\newcommand{\coll}[2]{{#2} / {#1}}
\newcommand{\zcoll}{\bullet}
\newcommand{\icoll}[1]{(#1)}
\newcommand{\mapcoll}{p}

\newcommand{\mapel}[1]{\iota_{#1}}
\newcommand{\imel}[2]{{#1}_{#2}}

\newcommand{\loc}[2]{#1[#2^{-1}]}
\newcommand{\preloc}[2]{#1\set{{#2}^{-1}}}
\newcommand{\selfloc}[1]{\widetilde{#1}}
\DeclareMathOperator{\Loc}{\fun{Loc}}

\newcommand{\F}{\fun{F}}
\newcommand{\G}{\fun{G}}

\newcommand{\m}[1]{{\mathsf{m}#1}}
\newcommand{\mdgmSet}{\atom\Set^{\m{}}}
\newcommand{\minmark}[1]{{#1}^{\flat}}
\newcommand{\natmark}[1]{{#1}^{\natural}}

\newcommand{\Jn}[1]{J_{#1}}
\newcommand{\Jhorn}{J_{\mathsf{horn}}}
\newcommand{\Jcomp}{J_{\mathsf{comp}}}

\newcommand{\omegaCat}{{\omega\Cat}}
\newcommand{\nCat}[1]{{#1\Cat}}
\newcommand{\somegaCat}{\omegaCat^>}
\newcommand{\snCat}[1]{\nCat{#1}^>}
\newcommand{\Comp}{\omega\cat{CSt}}
\newcommand{\nComp}[1]{{#1\cat{CSt}}}
\newcommand{\comp}[1]{*_{#1}}
\newcommand{\omegaReg}{\omega\cat{Reg}}
\newcommand{\omegaRegSet}{[\opp{\omegaReg}, \Set]^{\mathsf{rw, pst}}}
\newcommand{\rc}{\fun{r}}
\newcommand{\rcs}{\fun{r}^{>}}
\newcommand{\pcell}[1]{\langle#1\rangle}
\let\amalg\relax
\DeclareMathOperator*{\amalg}{\mathrm{amalg}}
\newcommand{\trunc}[1]{\tau_{#1}}
\newcommand{\skel}[1]{\sigma_{#1} }
\newcommand\molec{\mathit{Mol}}
\newcommand\Atom{\mathit{Atom}}
\newcommand\molecin[1]{\slice{\molec}{#1}}
\newcommand\atomin[1]{\slice{\Atom}{#1}}
\newcommand{\globe}[1]{O^{#1}}

\newcommand{\N}[1]{N_{#1}}

\newcommand{\Icof}{\mathcal{I}}
\newcommand{\wE}{\tilde E}
\newcommand{\swE}{\tilde E^>}

\newcommand{\homlax}{\underline{\fun{Hom}}_{\mathrm{lax}}}
\newcommand{\homcolax}{\underline{\fun{Hom}}_{\mathrm{colax}}}

\newcommand{\s}{\mathsf{t}}

\newcommand{\unit}{\eta}
\newcommand{\counit}{\epsilon}

\newcommand\cls[1]{\mathscr{#1}}
\makeatletter
\newcommand{\oset}[3][0ex]{%
  \mathrel{\mathop{#3}\limits^{
    \vbox to#1{\kern-1.5\ex@
    \hbox{$\scriptstyle#2$}\vss}}}}
\makeatother
\newcommand\qeq{\oset{?}{=}}

\newcommand*\cocolon{%
        \nobreak
        \mskip6mu plus1mu
        \mathpunct{}%
        \nonscript
        \mkern-\thinmuskip
        {:}%
        \mskip2mu
        \relax
}

\newtheoremstyle{ittheorem}
  {\topsep}   
  {\topsep}   
  {\itshape}  
  {0pt}       
  {\bfseries} 
  { ---}         
  {5pt plus 1pt minus 1pt} 
  {}          

\newtheoremstyle{itdfn}
  {\topsep}
  {\topsep}
  {}
  {0pt}
  {\bfseries}
  {}
  {5pt plus 1pt minus 1pt}
  {\thmnumber{#2}{\thmnote{\normalfont\ \ %
{\sffamily(#3)}.}}}

\newtheoremstyle{itrmk}
  {0.5\topsep}
  {0.5\topsep}
  {\normalfont}
  {0pt}
  {\sffamily \itshape}
  { --- }
  {5pt plus 1pt minus 1pt}
  {}
  
\theoremstyle{ittheorem}
\newtheorem{thm}{Theorem}[section]
\newtheorem*{thm*}{Theorem}
\newtheorem{prop}[thm]{Proposition}
\newtheorem*{prop*}{Proposition}
\newtheorem{cor}[thm]{Corollary}
\newtheorem{lem}[thm]{Lemma}

\newtheorem*{conj*}{Conjecture}
\theoremstyle{itdfn}
\newtheorem{dfn}[thm]{}
\theoremstyle{itrmk}
\newtheorem{rmk}[thm]{Remark}
\newtheorem{comm}[thm]{Comment}

\setlist{leftmargin=20pt,itemsep=0pt,parsep=0pt,topsep=1ex}

\titleformat{\section}
 {\Large\bfseries}{\thesection}{1em}{}

\titleformat{\subsection}
 {\normalsize\bfseries}{\thesubsection.}{1em}{}

\newcommand\runtitle{homotopy theory of stricter $n$-categories}
\newcommand\runauthor{chanavat}

\title{Homotopy theory of stricter $n$-categories}

\author{Cl\'emence Chanavat}

\institution{Tallinn University of Technology}

\begin{document}

\maketitle
\begin{center}
	\begin{minipage}[t]{.95\textwidth}
		\small\textsc{Abstract.}
		We make strict \( n \)\nbd categories even stricter by requiring they satisfy higher exchange laws governed by Hadzihasanovic's theory of regular directed complexes. 
		We study the first properties of stricter \( n \)\nbd categories, in particular, we define the Gray product, and prove stability under suspension, which is non-trivial. 
		After reviewing and briefly expanding the theory diagrammatic sets and their associated model structures for \( (\infty, n) \)\nbd categories, we construct a folk model structure on stricter \( n \)\nbd categories, show that the walking equivalence coincides with the stricter polygraph generated by the walking equivalence in diagrammatic sets, and finally, that the folk model structure on stricter \( n \)\nbd categories is right transferred from the diagrammatic model structure along a nerve construction.
	\end{minipage}
	
	\vspace{20pt}

	\begin{minipage}[t]{0.95\textwidth}
		\setcounter{tocdepth}{2}
		\tableofcontents
	\end{minipage}
\end{center}

\makeaftertitle

\section{Introduction}

Let \( n \in \mathbb{N} \cup \set{\omega} \).
The definition of strict \( n \)\nbd category is the most natural generalisation of the definition of strict \( 2 \)\nbd categories allowing for cells of dimension higher than two.
A strict \( n \)\nbd category \cite{brown1981groupoids} consists of a set, whose elements we call \emph{globular cells}\footnote{The terminology globular cell is non-standard. We use it to avoid any confusion with the already overloaded, in our context, notion of cell.}.
Each globular cell possesses, for each natural number \( k \), a notion of \( k \)\nbd source and \( k \)\nbd target. 
Given two globular cells such that the \( k \)\nbd target of the first one is equal to the \( k \)\nbd source of the second, one can compose them using the \( k \)\nbd composition operation, and obtain a third one.
Those composition operations are required to satisfy, akin to strict \( 2 \)\nbd categories, axioms of unitality, an axiom of associativity, and of exchange.
The latter makes it so that the \( k \)\nbd composition operation is functorial with respect to the \( (k + \ell) \)\nbd composition operation, for all \( \ell > 0 \).
As a consequence, the category of strict \( (n + 1) \)\nbd categories can canonically be identified with the category of categories enriched in strict \( n \)\nbd categories. 
Repeating this enrichment a transfinite number of times, one obtains the category of strict \( \omega \)\nbd categories, which is enriched over itself.
There should, a priori, be no need for further axioms.   

In fact, it is well known that, in a certain sense, the theory of strict \( n \)\nbd categories is already \emph{too} strict: all of its axioms hold up to equality and well-devised constructions, when \( n \geq 3 \), allow to break the topological intuition that associates to a globular cell of dimension \( k \), a \( k \)\nbd ball, with the hemispheres of its boundary decomposed recursively as the \( (k - 1) \)\nbd source and target, see \cite{simpson1998homotopy}.
Therefore, one is led to think that making strict \( n \)\nbd categories even stricter will only make the matter worse.
We claim, however, that turning them stricter is restoring a part of the topological intuition which is absent in the strict case: there is a way in which the strict \( n \)\nbd categories are not strict enough.
What we strictify is of a different nature than the way in which it is commonly understood that strict \( n \)\nbd categories are too strict, as we now explain.

\subsubsection*{Higher exchanges}

A useful way to define strict \( \omega \)\nbd categories is to generate them from combinatorial data, describing, for instance, classes of pasting diagrams.
In general, the underlying combinatorial data of such a strict \( \omega \)\nbd category consists of a poset with extra structure: Johnson's pasting schemes \cite{johnson1989pasting}, Street's parity complexes \cite{street1991parity}, Steiner's directed complexes \cite{steiner1993algebra}, and Hadzihasanovic's regular directed complexes \cite{hadzihasanovic2024combinatorics}.
Informally, let us call such a strict \( \omega \)\nbd category, a complex. 
A desideratum is that a complex \( P \) is freely generated (in the sense of polygraphs \cite{burroni1993higher}) by each of the elements \( x \) in \( P \), so that a strict functor \( \F \colon P \to C \) amounts to the data of a globular cell \( c_x \in C \) for each \( x \in P \) (satisfying, of course, appropriate boundary conditions).
For this \emph{freeness condition} to hold, most of the traditional literature imposes certain acyclicity conditions on the underlying combinatorial data, hence restrains analytically the collection of possible complexes --- see Forest \cite{forest2022pasting} for a unified treatment.

On the other hand, in his monograph \cite{hadzihasanovic2024combinatorics}, Hadzihasanovic takes a synthetic approach to the theory of higher categorical diagrams, modeled by his regular directed complexes.
A regular directed complex is the face poset of a regular CW complex together with orientation data which partition the closure of the face of each cell into two halves: its input and output boundary, both of them representing a composable arrangement of lower dimensional cells.
The main point of departure with the existing literature is that the class of complexes representing composable arrangements, called the \emph{molecules}\footnote{Named after Steiner's notion of molecules \cite{steiner1993algebra}}, is defined by induction.
One starts from the point, the terminal regular CW complex with its only possible orientation, and closes under two kinds of operations.
\begin{enumerate}
    \item The first operation is \emph{pasting}. Given two molecules \( U, V \) such that \( \bd{k}{+} U \), the output \( k \)\nbd boundary of \( U \) matches (in the sense of a directed and cellular isomorphism) \( \bd{k}{-} V \), the input \( k \)\nbd boundary of \( V \), the \emph{pasting at the \( k \)\nbd boundary of \( U \) and \( V \)} is the molecule \( U \cp{k} V \), defined as the pushout of \( U \) and \( V \) along the common boundary. 
    \item The second operation is \emph{rewrite}, and takes as input two molecules \( U, V \) of the same dimension \( k \), which are \emph{round}, that is, whose associated regular CW complex is homeomorphic to the topological \( k \)\nbd ball.
    Then, if both the input and output \( (k - 1) \)\nbd boundary of \( U \) match the input and output \( (k - 1) \)\nbd boundary of \( V \), the \emph{rewrite of \( U \) and \( V \)} is the molecule \( U \celto V \) defined by adding a greatest \( (k + 1) \)\nbd dimensional element \( \top \) to the pushout \( \bd{}{}(U \celto V) \) of \( U \) and \( V \) along their total boundary, in such a way that \( U \) and \( V \) are respectively the input and output boundary of the closure of \( \top \).
    This makes the regular CW complex associated to \( U \celto V \) a topological \( (k + 1) \)\nbd ball.
\end{enumerate}  
In particular, the boundary operators and pasting operations make the collection of (isomorphism classes of) molecules a strict \( \omega \)\nbd category.
For instance, given molecules \( U, U', V, V \) and \( k < \ell \) such that \( (U \cp{\ell} U') \cp{k} (V \cp{\ell} V') \) is defined, we have a unique isomorphism
\begin{equation*}
    (U \cp{\ell} U') \cp{k} (V \cp{\ell} V') \cong (U \cp{k} V) \cp{\ell} (U' \cp{k} V').
\end{equation*}
One might reasonably hope that the axioms of strict \( n \)\nbd category suffice to generate all the equations satisfied by pasting of molecules.
This would imply that a strict \( \omega \)\nbd category generated by a regular directed complex always satisfies the freeness condition. 
This is not the case.
More precisely, this is the case as long as \( n \le 3 \), but not afterwards, a four-dimensional ``higher exchange'' is discussed in Comment \ref{comm:strict_are_not_stricter}.
In this paper, we look at what happens to the theory of strict \( n \)\nbd categories when one forces the freeness condition to hold, not by restricting the class of allowed complexes, but by requiring instead the strict \( n \)\nbd categories to satisfy more axioms, to be stricter.

We want to emphasise that this is not the study of the localisation of the category of strict \( \omega \)\nbd categories along an arbitrary set of equations, but along the equations prescribed by combinatorial topology, which restores the status of pasting as a universal construction.
What we gain is a stronger pasting theorem than the ones for strict \( n \)\nbd categories.
Indeed, stricter \( n \)\nbd categories have, by definition, a pasting theorem with respect to the class of all regular directed complexes. 
Since the latter contain many shapes used in the higher categorical literature (directed cubes and simplices, positive opetopes, thetas), the pasting theorem holds in particular for many classes of diagrams of practical interest.   
Similarly, regular directed complexes are closed under many categorical operations (join, Gray product, duals, \dots), hence allow for uniform definitions of the said operation on stricter \( n \)\nbd categories.
Conversely, we loose the ability to define stricter \( n \)\nbd categories by iterated enrichment when \( n \geq 4 \), and we have, by design, to work harder to prove that some given data assemble into a stricter \( n \)\nbd category, even when knowing that it is a strict \( n \)\nbd category, as witnessed for instance by the proof of Theorem \ref{thm:suspension_of_stricter}, asserting that stricter \( \omega \)\nbd categories are stable under suspensions, and whose proof relies on quite a number technical preliminary Lemmas.  

\subsection*{Homotopy theory of stricter \( n \)\nbd categories}

Another reason to introduce stricter \( n \)\nbd categories, and our primary motivation, is to compare them to the diagrammatic model of \( (\infty, n) \)\nbd categories, which is a model of higher categories \cite{chanavat2024htpy,chanavat2024model} built in the category \( \dgmSet \) of \emph{diagrammatic sets} --- the presheaves over the categories \( \atom \) of \emph{atoms}, which are the regular directed complexes with a greatest element.  
Our working conjecture is that, as long as \( n \le 3 \), the standard model of \( (\infty, n) \)\nbd categories \cite{barwick2020unicity} and the diagrammatic model of \( (\infty, n) \)\nbd categories are equivalent --- we remain noncommital for \( n > 3 \).
Before any attempt of solving this problem, we thought it would be a good idea to understand the missing corner in the following diagram of \( (\infty, 1) \)\nbd categories,
\begin{center}
    \begin{tikzcd}
        \begin{array}{c} \text{standard} \\ (0, n) \end{array} && \begin{array}{c} \text{standard} \\ (\infty, n) \end{array} \\
        {?} && \begin{array}{c} \text{diagrammatic} \\ (\infty, n). \end{array}
        \arrow[hook, from=1-1, to=1-3]
        \arrow[""{name=0, anchor=center, inner sep=0}, shift right=3, from=1-1, to=2-1]
        \arrow[""{name=1, anchor=center, inner sep=0}, shift right=3, from=1-3, to=2-3]
        \arrow[""{name=2, anchor=center, inner sep=0}, shift right=3, from=2-1, to=1-1]
        \arrow[hook, from=2-1, to=2-3]
        \arrow[""{name=3, anchor=center, inner sep=0}, shift right=3, from=2-3, to=1-3]
        \arrow["\dashv"{anchor=center}, draw=none, from=0, to=2]
        \arrow["\dashv"{anchor=center}, draw=none, from=1, to=3]
    \end{tikzcd}
\end{center}
This is what we start doing with the main results (Theorem \ref{thm:folk_model_structure_on_stricter_n} and Theorem \ref{thm:quillen_folk_dgm_n}) of this article, that we summarise here.
\begin{thm*}
    Let \( n \in \mathbb{N} \cup \set{\omega} \).
    There exists a model structure, called the \emph{folk model structure}, on the category \( \snCat{n} \) of stricter \( n \)\nbd categories.
    This model structure is right transferred along both:
    \begin{enumerate}
        \item the full subcategory inclusion \( \snCat{n} \incl \nCat{n} \) of stricter \( n \)\nbd categories into strict \( n \)\nbd categories with the folk model structure, and
        \item the diagrammatic nerve \( \N{n} \colon \snCat{n} \to \dgmSet \) where \( \dgmSet \) is endowed with the \( (\infty, n) \)\nbd model structure. 
    \end{enumerate}
\end{thm*}
\noindent Here, the diagrammatic nerve \( \N{n} \colon \snCat{n} \to \dgmSet \) is the right adjoint to the functor \( \trunc{n} \after \molecin{-} \), which sends a diagrammatic set \( X \) to the \( n \)\nbd truncation of \( \molecin{X} \), the free \emph{stricter polygraph} generated by the non-degenerate cells of \( X \).
In its non-truncated version, the diagrammatic nerve \( \N{} \colon \somegaCat \to \dgmSet \) restricts to the Street nerve \cite{street1987oriental} \( N^S \colon \somegaCat \to \sSet \) along the full subcategory of \( \atom \) on atoms which are \emph{directed simplices}; this subcategory being isomorphic to the simplex category.
Our initial hope was also to include the fact the diagrammatic nerve is homotopically fully faithful, but our attempts at a proof were unsuccessful.
We can only state conjecturally.
\begin{conj*}
    Let \( n \in \mathbb{N} \cup \set{\omega} \) and \( C \) be a stricter \( n \)\nbd category. 
    Then the (derived) counit
    \begin{equation*}
        \counit_C \colon \trunc{n} \molecin{(\N{}C)} \to C
    \end{equation*}
    is an acyclic fibration in the folk model structure on stricter \( n \)\nbd categories.
\end{conj*}

\noindent Since the homotopy theory of stricter \( n \)\nbd categories will not be equivalent to that of diagrammatic \( (\infty, n) \)\nbd categories --- the former are still too strict in the traditional sense --- the derived unit will not be a weak equivalence.
However, following \cite{gagna2023nerve, maehara2023oriental}, we are quite confident that the derived unit is a weak equivalence on regular directed complexes, but have not attempted a proof.
\begin{conj*}
    Let \( P \) be a regular directed complex.
    Then the (derived) unit 
    \begin{equation*}
        \unit_P \colon P \to \N{}\molecin{P}
    \end{equation*}
    is an acyclic cofibration in the \( (\infty, \omega) \)\nbd model structure on diagrammatic sets. 
\end{conj*}

\noindent Thus, the role played by stricter \( n \)\nbd categories in the diagrammatic model is analogous to the role played by strict \( n \)\nbd categories in the standard model.
In particular, we believe that the \( (\infty, 1) \)\nbd categories of diagrammatic \( (\infty, n) \)\nbd categories should be presented by a model structure on certain presheaves over the full subcategory of \( \snCat{n} \) on stricter \( n \)\nbd categories of the form \( \molecin{P} \), for \( P \) a (finite) regular directed complex.
This presentation would be akin to the presentation of the \( (\infty, 1) \)\nbd category of standard \( (\infty, n) \)\nbd categories given by \( n \)\nbd quasicategories \cite{ara2014higher} or \( \Theta_n \)\nbd spaces \cite{rezk2010cartesian}.     
Since by Theorem \ref{thm:strict_le_3_are_stricter}, the categories \( \snCat{n} \) and \( \nCat{n} \) coincide when \( n \le 3 \), such a presentation of the \( (\infty, 1) \)\nbd categories of diagrammatic \( (\infty, n) \)\nbd categories should lead to a proof that the diagrammatic and the standard model coincide for \( n \le 3 \).
We are actively working on implementing this strategy.

\subsection*{Background on regular directed complexes}

We now set up some notations on the combinatorics of regular directed complexes.
All the details and proofs are in \cite{hadzihasanovic2024combinatorics}.
The reader can also read the introductions of \cite{chanavat2024htpy, chanavat2024equivalences, chanavat2024model}.
The basic combinatorial structure is that of \emph{oriented graded poset}: posets \( P \) graded by a function \( \dim \colon P \to \mathbb{N} \), together with orientation data specified, for each \( x \in P \), by a partition of the set \( \faces{}{} x \), the \emph{faces of \( x \)}, into \( \faces{}{} x = \faces{}{-} x + \faces{}{+} x \), interpreted as the elements that \( x \) covers with orientation \( - \) and the elements that \( x \) covers with orientation \( + \). 
This is equivalent to giving a partition \( \cofaces{}{} x = \cofaces{}{-} x + \cofaces{}{+} x \) of the \emph{cofaces of \( x \)}, the elements of \( P \) that cover \( x \).
We write \( \maxel{P} \) for the set of maximal elements of \( P \), that is, the elements \( x \in P \) such that \( \cofaces{}{} x = \varnothing \).
Given a subset \( U \subseteq P \), we write \( \clos U \) for the \emph{closure of \( U \)}, defined to be the lower set of \( U \). 
We say that \( U \) is \emph{closed} if it is equal to its closure.
Given an oriented graded poset \( P \) and \( k \in \mathbb{N} \), we write \( \gr{k}{P} \) for the set of elements \( x \in P \) such that \( \dim x = k \), and \( \gr{\le k}{P} \) for the closed set on elements \( x \in P \) such that \( \dim x \le k \). 
Each closed subset \( U \subseteq P \) of an oriented graded poset \( P \) possesses, for \( k \in \mathbb{N} \) and \( \a \in \set{-, +} \), a notion of \emph{input (\( \a = - \)) and output (\( \a = + \)) \( k \)\nbd boundary}, written \( \bd{k}{\a} U \), which is a closed subset of \( U \). 
We let the \emph{\( k \)\nbd boundary of \( U \)} be \( \bd{k}{} U \eqdef \bd{k}{-} U \cup \bd{k}{+} U \).
If \( U \) has a greatest element \( x \), we also write \( \bd{k}{\a} x \) for \( \bd{k}{\a} U \).
By convention, those subsets are empty if \( k < 0 \), and we may omit \( k \) if it is equal to \( \dim U - 1 \).
We say that a finite oriented graded poset \( P \) is round if for all \( k < \dim P  \), \( \bd{k}{-} P \cap \bd{k}{+} P = \bd{k - 1}{} P \).

We define the collection of \emph{molecules} to be the collection of oriented graded posets defined by the inductive procedure described previously, and call an \emph{atom} a molecule with a greatest element.
In particular, \emph{the point} is the atom \( \pt \) whose underlying set is a singleton, endowed with its unique possible oriented graded structure. 
A \emph{regular directed complex} is an oriented graded poset \( P \) with the property that for all \( x \in P \), the closure \( \clset{x} \) of \( x \) is an atom.
Isomorphisms of molecules are unique when they exist and, as stated earlier in the introduction, boundaries and pastings of molecules satisfy all the axioms of strict \( \omega \)\nbd categories.

A \emph{comap of regular directed complexes} is given by an order-preserving function \( c \colon Q \to P \) such that, for all \( x \in P \), \( k \in \mathbb{N} \) and \( \a \in \set{-, +} \), 
\begin{enumerate}
    \item \( \invrs{c}(\clset{x}) \) is a molecule, 
    \item \( \invrs{c}(\bd{k}{\a} x) = \bd{k}{\a} \invrs{c}\clset{x} \).
\end{enumerate}
Comaps compose, and we write \( \rdcpxcomap \) for the category of regular directed complexes and comaps.
A \emph{subdivision \( s \colon P \sd Q \) of regular directed complexes} is a comap \( c \colon Q \to P \).
We say that \( s \) is the formal dual of \( c \) and reciprocally, that \( c \) is the formal dual of \( s \).
If \( U \subseteq P \) is a closed subset of \( P \), we write \( s(U) \) for the closed subset of \( Q \) defined by \( \invrs{c}(U) \).
This is a molecule if \( U \) is a molecule and in that case, we again have \( s(\bd{k}{\a} U) = \bd{k}{\a} s(U) \), for all \( k \in \mathbb{N} \) and \( \a \in \set{-, +} \).

A \emph{map of regular directed complexes} is given by an order-preserving function \( f \colon P \to Q \) such that, for all \( x \in P \), \( k \in \mathbb{N} \) and \( \a \in \set{-, +} \), \( f(\bd{k}{\a}x) = \bd{k}{\a} f(x) \), and the restriction \( \restr{f}{\bd{k}{\a} x} \colon \bd{k}{\a} x \to f(\bd{k}{\a} x) \) is a final functor of posets seen as categories. 
We write \( \rdcpxmap \) for the category of regular directed complexes and maps.
Among maps of regular directed complexes are the \emph{cartesian map of regular directed complexes}, which are maps \( f \colon P \to Q \) of regular directed complexes with the extra property of being Grothendieck fibration of their underlying posets seen as categories.
We write \( \rdcpxcart \) for the wide subcategory of \( \rdcpxmap \) on cartesian maps, and \( \atom \) for (a skeleton of) the full subcategory of \( \rdcpxcart \) on atoms.

An \emph{inclusion} is a (necessarily cartesian) map of regular directed complexes which is injective, and a \emph{local embedding} is a map of regular directed complexes which is a discrete fibration of the underlying posets seen as categories.
Alternatively, \( f \colon P \to Q \) is a local embedding if, for all \( x \in P \), \( \restr{f}{\clset{x}} \) is an inclusion.
Given a regular directed complex \( P \) and \( x \in P \), we write \( \mapel{x} \colon \imel{P}{x} \incl P \) for the unique inclusion with image \( \clset{x} \) in \( P \).
We will sometimes conflate an element \( x \in \gr{0}{P} \) with its associated local embedding \( x \colon \pt \to P \).

The class of \emph{submolecule inclusions} is the smallest class of inclusions of molecules closed under isomorphisms, compositions, and containing the inclusions
\begin{equation*}
    U \incl U \cp{k} V \quad\text{ and }\quad V \incl U \cp{k} V,
\end{equation*}
whenever a pasting \( U \cp{k} V \) is defined. 
We also write \( \iota \colon U \submol V \) for a submolecule inclusion \( \iota \colon U \incl V \).
If \( \iota \colon \bd{k}{+} U \submol \bd{k}{-} V \) for some \( k \in \mathbb{N} \), then the pushout \( V \cup_\iota U \) is a molecule, written \( U \cpsub{\iota} V \), and called the \emph{pasting of \( U \) at the submolecule \( \iota \)}.
Dually if \( \iota \colon \bd{k}{-} U \submol \bd{k}{+} V \), we also define \( V \subcp{\iota} U \).

Given a map of regular directed complexes \( f \colon U \to P \) such that \( U \) is a molecule, \( k \in \mathbb{N} \), and \( \a \in \set{-, +} \), we write \( \bd{k}{\a} f \) for the restriction of \( f \) along the inclusion \( \bd{k}{\a} U \incl U \).
By the universal property of the pushout, if  \( g \colon V \to P \) is another map of regular directed complexes such that \( V \) is a molecule and \( \bd{k}{+} f = \bd{k}{-}g \), then we have a map of regular directed complexes \( f \cp{k} g \colon U \cp{k} V \to P \), which is a local embedding if \( f \) and \( g \) are.

Let \( P, Q \) be two regular directed complexes.
The \emph{Gray product of \( P \) and \( Q \)} is the oriented graded poset \( P \gray Q \) whose underlying graded poset is given by \( P \times Q \) and orientation is defined, for all \( (x, y) \in P \times Q \) and \( \a \in \set{-, +} \), by
\begin{equation*}
    \faces{}{\a} (x, y) = \faces{}{\a} x \times \set{y} + \set{x} \times \faces{}{(-)^{\dim x} \a} y.
\end{equation*}  
The Gray product of two regular directed complexes is (non-trivially) a regular directed complex, and determines monoidal structure on the categories \( \rdcpxcomap \), \( \rdcpxmap \), \( \rdcpxcart \), and \( \atom \).

The \emph{suspension of \( P \)} is the oriented graded poset with underlying set given by
\begin{equation*}
    \sus{P} \eqdef \set{\bot^-, \bot^+} + \set{\sus{x} \mid x \in P},
\end{equation*}
and oriented graded structure defined, for all \( x \in \sus{P} \) and \( \a \in \set{-, +} \), by
\begin{equation*}
    \cofaces{}{\a} x \eqdef 
    \begin{cases}
        \set{\sus{y} \mid y \in \cofaces{}{\a} x'} & \text{if } x = \sus{x'}, \\
        \set{\sus{y} \mid y \in \gr{0}{P}} & \text{if } x = \bot^\a, \\
        \varnothing & \text{if } x = \bot^{- \a}.
    \end{cases}
\end{equation*}
The suspension of a regular directed complex is again a regular directed complex, and determines functors
\begin{equation*}
    \sus{-} \colon \rdcpxcomap \to \rdcpxcomap \quad\text{ and }\quad \sus{-} \colon \rdcpxmap \to \rdcpxmap,
\end{equation*}
the latter restricting to \( \rdcpxcart \) and \( \atom \).

Given \( J \subseteq \mathbb{N} \), the \emph{\( J \)\nbd dual} of \( P \) is the oriented graded poset \( \dual{J}{P} \) whose underlying set is \( \set{\dual{J}{x} \mid x \in P} \), and oriented graded structure is defined, for all \( x \in P \) and \( \a \in \set{-, +} \), by
\begin{equation*}
    \faces{}{\a} \dual{J}{x} = 
    \begin{cases}
        \set{\dual{J}{y} \mid y \in \faces{}{-\a} x} & \text{if } \dim x \in J,\\
        \set{\dual{J}{y} \mid y \in \faces{}{\a} x}  & \text{if } \dim x \notin J.
    \end{cases}
\end{equation*}
The \( J \)\nbd dual of a regular directed complex is again a regular directed complex, and determines functors 
\begin{equation*}
    \dual{J}{-} \colon \rdcpxcomap \to \rdcpxcomap \quad\text{ and }\quad \dual{J}{-} \colon \rdcpxmap \to \rdcpxmap,
\end{equation*}
the latter restricting to \( \rdcpxcart \) and \( \atom \).
We write \( \dual{k}{P} \) when \( J \) is a singleton \( \set{k} \).

Let \( k \geq 0 \).
The \emph{\( k \)\nbd globe} is the atom \( \dglobe{k} \) defined by letting \( \dglobe{0} \) be the point \( \pt \), and inductively on \( k > 0 \), \( \dglobe{k} \eqdef \sus{\dglobe{k - 1}} \).
For each round molecule \( U \) of dimension \( k \in \mathbb{N} \), there exists a unique subdivision \( \dglobe{k} \sd U \).
We write \( \set{0^- < 1 > 0^+} \) for the underlying poset of \( \dglobe{1} \), with orientation such that \( \faces{}{\a} 1 = 0^\a \), for all \( a \in \set{-, +} \).
For \( k \geq 0 \), we write \( k\arr \) for the molecule \( \underbrace{\arr \cp{0} \cdots \cp{0} \arr}_{k \text{ times}} \).
If \( k = 0 \) then \( k\arr \) is to be interpreted as the point \( \pt \).

The \emph{augmentation} of a graded poset \( P \) is the graded poset \( \augm{P} \) whose underlying set is \( \set{\bot} + P \), and graded structure is given, for all \( x \in \augm{P} \), by
\begin{equation*}
    \cofaces{}{} x \eqdef
    \begin{cases}
        \cofaces{P}{} x &\text{if } x \in P,\\
        \gr{0}{P}         &\text{if } x = \bot.
    \end{cases}
\end{equation*}
Let \( P \) be a graded poset with a least element.
We say that \( P \) is \emph{thin} if, for all \( x, y \in P \) such that \( x \le y \) and \( \dim y - \dim x = 2 \), the interval \( [x, y] \) is a \emph{diamond}, that is, it is of the form 
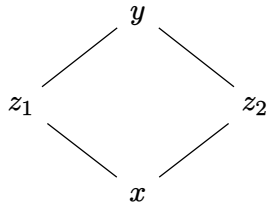
\begin{center}
    \begin{tikzcd}
        & y \\
        {z_1} && {z_2} \\
        & x
        \arrow[no head, from=1-2, to=2-1]
        \arrow[no head, from=1-2, to=2-3]
        \arrow[no head, from=2-1, to=3-2]
        \arrow[no head, from=3-2, to=2-3]
    \end{tikzcd}
\end{center}
for exactly two elements \( x < z_1, z_2 < y \).
Then, if \( P \) is a regular directed complex, the graded poset \( \augm{P} \) is thin\footnote{In fact, the oriented graded poset \( \augm{P} \) is \emph{oriented thin}, but we will not use this stronger property in this article, see \cite[2.3.10]{hadzihasanovic2024combinatorics}.}.

\subsection*{Structure of the article}

In Section \ref{sec:stricter}, we define and study the first properties of stricter \( \omega \)\nbd categories.
We found it convenient to work at the level of composition structures, which are reflexive \( \omega \)\nbd graphs\footnote{We point out that we chose to use the single set definition of globular graph.} together with composition operations satisfying no axioms at all.
After defining the functor \( \molecin{-} \), sending a regular directed complex to its canonical composition structure, and setting up some terminology, we are ready to introduce in Definition \ref{dfn:stricter_omega_cat} the notion of stricter \( \omega \)\nbd category, and give several alternative characterisations in Lemma \ref{lem:stricter_iff_local_wrt_pasting}.
We show that \( \molecin{-} \) always sends regular directed complexes to stricter \( \omega \)\nbd categories (Proposition \ref{prop:regular_directed_complex_stricter}) and deduce, in Corollary \ref{cor:regular_directed_complex_colimit_of_itself}, the pasting theorem for stricter \( \omega \)\nbd categories.
We then define stricter \( n \)\nbd categories, for \( n \in \mathbb{N} \), and show in Lemma \ref{lem:truncation_stricter_are_stricter} that the \( n \)\nbd skeleton and the \( n \)\nbd truncation of a stricter \( \omega \)\nbd category is a stricter \( n \)\nbd category.
This allows us to define the notion of stricter polygraph (Definition \ref{dfn:stricter_polygraph}), of which \( \molecin{P} \) is an instance (Lemma \ref{lem:stricter_regular_complex_are_stricter_polygraph}).
Then, we conclude the first part by showing that stricter \( \omega \)\nbd categories are local presheaves over the full subcategory of the category of composition structures on objects of the form \( \molecin{P} \), for \( P \) a finite regular directed complex (Proposition \ref{prop:stricter_cat_are_local_presheaves}); analogous to the relationship that strict \( \omega \)\nbd categories entertain with thetas.
From there, we move on to the definition of the Gray product (Definition \ref{dfn:gray_product_stricter_categories}), following closely the strict \( \omega \)\nbd categorical literature.
The next part is concerned with comparing strict and stricter \( \omega \)\nbd categories, the latter being indeed particular case of the former (Proposition \ref{prop:stricter_are_strict}).
This exhibits the category stricter \( \omega \)\nbd categories as a reflective subcategory of the category of strict \( \omega \)\nbd categories.
We show that the reflector sends polygraphs to stricter polygraphs (Lemma \ref{lem:reflection_of_polygraphs_are_stricter_polygraphs}) and is monoidal with respect to the Gray product (Proposition \ref{prop:reflection_to_stricter_monoidal}).
We then show that, as long as \( n \le 3 \), a stricter \( n \)\nbd category is a strict \( n \)\nbd category (Theorem \ref{thm:strict_le_3_are_stricter}) and briefly describe a stricter \( 4 \)\nbd category which is not a strict \( 4 \)\nbd category (Comment \ref{comm:strict_are_not_stricter}).
We conclude this section with the study of suspension of stricter \( \omega \)\nbd categories.
First, we show that it is the case that a stricter \( \omega \)\nbd category is a category enriched in stricter \( \omega \)\nbd categories (Lemma \ref{lem:hom_of_stricter_is_stricter}).
Even though the converse cannot hold, we nonetheless show that it holds for the particular case of the suspension (Theorem \ref{thm:suspension_of_stricter}).
We prove this result by showing that the quotient of a molecule along particular collapsible subsets (Definition \ref{dfn:collapsible}) is again a molecule which is the suspension of another molecule, culminating with Proposition \ref{prop:collapsible_collapse_to_molecules}, which follows a number of technical preliminary results that the reader can safely skip during their first read.

In Section \ref{sec:diagrammatic}, we review the homotopy theory of diagrammatic sets, with some small improvements. 
After setting up the usual terminology, we show in Lemma \ref{lem:Pd_is_stricter} that the collection of pasting diagrams of a diagrammatic set assembles into a stricter \( \omega \)\nbd category, giving a further collection of canonical examples.
We then introduce degenerate diagrams and equivalences, and show (Lemma \ref{lem:subdivision_of_unitors} and Lemma \ref{lem:subdivision_of_invertors}) that a certain class of surjective cartesian maps respect subdivisions.
Next, we recall the definition of (diagrammatic) \( (\infty, n) \)\nbd category (Definition \ref{dfn:infty_n_cat}) and recall the construction of the coinductive \( (\infty, n) \)\nbd model structure on marked diagrammatic sets (Proposition \ref{prop:model_structre_on_marked_dgm_set}).
We then move on to the model structure on plain diagrammatic sets. 
We wish to give a smaller pseudo-generating set of acyclic cofibrations, thus start by some preliminary results (Lemma \ref{lem:isofib_left_right_lift}, Lemma \ref{lem:isofib_rlp_marked_horn}) which will be enough in Theorem \ref{thm:n_model_structure_on_dgm_set} to show that the acyclic cofibrations in the \( (\infty, n) \)\nbd model structure on diagrammatic sets are pseudo-generated by the walking weak composites and the walking \( k \)\nbd equivalences, for \( k > n \). 

Finally, in Section \ref{sec:model}, we study the folk model structure on stricter \( n \)\nbd category. 
Since Gray products of strict \( n \)\nbd categories are reflected to Gray products of stricter \( \omega \)\nbd categories, the existence (Theorem \ref{thm:folk_model_structure_on_stricter_n}) of the model structure is a direct application of a Theorem of \cite{ara2025polygraphs}.
We then extend the functor \( \molecin{-} \) to the whole category of diagrammatic sets, and show (Corollary \ref{cor:molecin_polygraph_with_basis}) that its image are stricter polygraphs. 
We then reflect the walking equivalence of strict \( \omega \)\nbd categories constructed in \cite{hadzihasanovic2024model} and show that it coincides with the stricter polygraph generated by the walking equivalence of diagrammatic sets (Lemma \ref{lem:swE_is_iso_to_molecin_loc_globe}).
Using suspension, we show in Proposition \ref{prop:walking_eq_of_dim_n} that this is again the case of the walking equivalence of dimension \( n \).
The last part is concerned with showing that the functor \( \molecin{-} \) is left Quillen and that its right adjoint transfers the diagrammatic model structure onto the folk model structure.
The main technical bit is to show that the localisation is compatible with the subdivision of atoms (Lemma \ref{lem:pushout_with_localisation}); after that we can quickly deduce our main result, first in the case \( n = \omega \) (Proposition \ref{prop:quillen_folk_dgm_infty}), then truncating it for all \( n \in \mathbb{N} \cup \set{\omega} \) in Theorem \ref{thm:quillen_folk_dgm_n}.  
We conclude the article with two parallel proofs that the Gray product is monoidal with respect to the folk model structure on stricter \( n \)\nbd categories (Proposition \ref{prop:Gray_monoidal}).

\subsection*{Related work}

The existence of stricter \( n \)\nbd categories and the folk model structure was conjectured by the author and Hadzihasanovic in \cite[Conjecture 6.3]{chanavat2024model}, which is now Theorem \ref{thm:quillen_folk_dgm_n}.
Concerning right transferring model structures from weak to strict, the author developed her approach following \cite{ozornova2021nerves}.

\subsection*{Acknowledgment}
\noindent We are grateful to Amar Hadzihasanovic for feedback, and thank Fosco Loregian and Viktoriya Ozornova for helpful discussions.

\section{Stricter \texorpdfstring{$\omega$}{ω}-categories} \label{sec:stricter}

\subsection{Definitions and properties}

\begin{dfn} [Reflexive \( \omega \)\nbd graph]
    A \emph{reflexive \( \omega \)\nbd graph} is a set \( C \), whose element are called the \emph{globular cells}, together with, for each \( k \in \mathbb{N} \), operators
    \begin{equation*}
        \bd{k}{-}, \bd{k}{+} \colon C \to C,
    \end{equation*}
    called the \emph{input and output \( k \)\nbd boundary operators}, respectively, satisfying the following axioms:
    \begin{enumerate}
        \item for all \( c \in C \), there exists \( k \in \mathbb{N} \) such that \( \bd{k}{-} c = c = \bd{k}{+} c \); the \emph{dimension of \( c \)}, written \( \dim c \), is the minimum of all such values of \( k \);
        \item for all \( c \in C \), all \( k, n \geq 0 \) and all \( \a, \beta \in \set{-, +} \),
        \begin{equation*}
            \bd{k}{\a}(\bd{n}{\beta} c) = 
            \begin{cases}
                \bd{k}{\a} c & \text{if }k < n, \\
                \bd{n}{\beta} c& \text{else.}
            \end{cases}
        \end{equation*}
    \end{enumerate}
    A \emph{morphism of reflexive \( \omega \)\nbd graphs} is a function of the underlying set commuting with the boundary operators.

    If \( C \) is a reflexive \( \omega \)\nbd graph, the set of \emph{\( k \)\nbd composable pairs of globular cells} is the set 
    \begin{equation*}
        C \times_k C \eqdef \set{(c, d) \in C \times C \mid \bd{k}{+} c = \bd{k}{-} d}.
    \end{equation*}
    Given a globular cell \( c \) and \( \a \in \set{-, +} \), we write \( \bd{}{\a} c \) in place of \( \bd{\dim c - 1}{\a} c \).
    We say that a globular cell \( c \) is an \emph{object} if \( \dim c = 0 \). 
\end{dfn}

\begin{dfn} [Composition structure]
    A composition structure is a reflexive \( \omega \)\nbd graph \( C \) together with, for all \( k \in \mathbb{N} \), an operation
    \begin{equation*}
        - \comp{k} - \colon C \times_k C \to C,
    \end{equation*}
    call the \emph{\( k \)\nbd composition operation}.
    If \( C, D \) are composition structures, a \emph{strict functor} \( \F \colon C \to D \) is a morphism of the underlying reflexive  \( \omega \)\nbd graphs respecting the \( k \)\nbd composition operations, for all \( k \in \mathbb{N} \).
    We denote \( \Comp \) the category of composition structures and strict functors.
\end{dfn}

\begin{rmk}
    The category \( \Comp \) is equivalent to a category of models of a limit sketch; using for instance a simpler version of \cite[Proposition 14.2.4]{ara2025polygraphs}.
    In particular, it is locally presentable, complete, and cocomplete \cite{adamek1994locally}.
\end{rmk}

\begin{dfn} [Basis for composition structure]
    Let \( C \) be a composition structure, and \( \cls{S} \) be a subset of the globular cells of \( C \).
    We say that \( \cls{S} \) is a \emph{generating set for \( C \)} if the closure of \( \cls{S} \) under the operations \( \set{- \comp{k} -}_{k \in \mathbb{N}} \) is equal to \( C \).
    We say that a generating set is a \emph{basis for \( C \)} if for any other generating set \( \cls{T} \) such that \( \cls{T} \subseteq \cls{S} \), then \( \cls{T} = \cls{S} \).
\end{dfn}

\begin{lem}\label{lem:strict_functor_determined_by_basis}
    Let \( f, g \colon C \to D \) be strict functors of composition structures and let \( \cls{S} \) be a generating set for \( C \) such that for all \( c \in \cls{C} \), \( f(c) = g(c) \).
    Then \( f = g \).
\end{lem}
\begin{proof}
    See \cite[Lemma 5.1.23]{hadzihasanovic2024combinatorics}.
\end{proof}

\begin{dfn} [Molecules in a regular directed complexes]
    Let \( P \) be a regular directed complex.
    We let \( \molecin{P} \) to be the composition structure whose
    \begin{itemize}
        \item globular cells are the local embedding \( u \colon U \to P \) with \( U \) a molecule;
        \item boundary operators are defined by \( u \mapsto \bd{k}{ \a} u \), for all \( k \in \mathbb{N} \) and \( \a \in \set{-, +} \);
        \item the \( k \)\nbd composition operation is given by the pasting \( (u, v) \mapsto u \cp{k} v \), for all \( k \in \mathbb{N} \). 
    \end{itemize}    
\end{dfn}

\begin{rmk}
    Since isomorphisms of molecules are unique when they exist, we always identify two isomorphic objects \( u \colon U \to P \) and \( u' \colon U' \to P \) in \( \molecin{P} \).
\end{rmk}

\noindent Recall from \cite[Proposition 5.2.7, Theorem 6.2.35, Theorem 6.3.17]{hadzihasanovic2024combinatorics} that the subset \( \atomin{P} \) of \( \molecin{P} \) on local embeddings \( u \colon U \to P \) whose domain is an atom, is a basis for \( \molecin{P} \), and the assignment \( P \mapsto \molecin{P} \) extends to functors
\begin{align*}
    \molecin{-} &\colon \rdcpxmap \to \Comp \\
    \molecin{-} &\colon \opp{\rdcpxcomap} \to \Comp.   
\end{align*} 

\begin{dfn} [Globe]
    Let \( n \geq 0 \).
    The \emph{\( n \)\nbd globe} is the composition structure defined by \( \globe{n} \eqdef \molecin{\dglobe{n}} \), and its \emph{boundary} is the composition structure given by \( \bd{}{}\globe{n} \eqdef \molecin{(\bd{}{}\dglobe{n})} \).
\end{dfn}

\begin{dfn} [Pasting diagram in a composition structure]
    Let \( C \) be a composition structure and \( U \) be a molecule.
    A \emph{pasting diagram of shape \( U \) in \( C \)} is a strict functor \( u \colon \molecin{U} \to C \).
    If \( U \) is an atom, we say that \( u \) is a \emph{cell}.
    We let \( \dim u \eqdef \dim U \), and for \( k \geq 0 \) and \( \a \in \set{-, +} \), we write \( \bd{k}{\a} u \) for the restriction of \( u \) along the strict functor \( \molecin{\bd{k}{\a} U} \to \molecin{U} \), and may omit \( k \) when it is equal to \( \dim u - 1 \).
    Finally, we write \( u \colon u^- \celto u^+ \) to mean that \( \bd{}{\a} u = u^\a \), for all \( \a \in \set{-, +} \), and call \( u^- \celto u^+ \) the \emph{type} of \( u \).
    Two pasting diagrams are \emph{parallel} if they have the same type.
    A \emph{subdiagram of a pasting diagram \( u \colon \molecin{U} \to C \)} is the data of a pasting diagram \( v \colon \molecin{V} \to C \) together with a submolecule inclusion \( \iota \colon V \submol U \) such that \( u \after \molecin{\iota} = v \).
    We write \( \iota \colon v \submol u \) for the data of a subdiagram.
\end{dfn}

\begin{dfn} [Principal cell]
    Let \( \F \colon \molecin{U} \to C \) be a pasting diagram, and \( c \) be a globular cell of \( C \).
    We say that \emph{\( \F \) classifies \( c \)} if \( \F(\idd{U}) = c \).
    In that case, \( c \) is called the \emph{principal cell of \( \F \)} and is written \( \pcell{\F} \).
\end{dfn}

\begin{comm}
    We warn the reader on the difference between cells and globular cells. 
    A globular cell \( c \in C \) of dimension \( n \geq 0 \) is classified uniquely by the cell \( u \colon \globe{n} \to C \) such that \( \pcell{u} = c \).
    Conversely, a cell \( u \colon \globe{n} \to C \) factors uniquely as \( v \after \molecin{\tau} \), where \( \tau \colon \dglobe{n} \surj \dglobe{\dim \pcell{u}} \) is the unique surjection of atoms, and \( v \colon \globe{\dim \pcell{u}} \to C \) classifies \( \pcell{u} \).
\end{comm}

\begin{dfn}
    For each regular directed complex \( P \), there is a canonical strict functor
    \begin{equation*}
        s_P \colon \colim_{x \in P} \molecin{\imel{P}{x}} \to \molecin{P},
    \end{equation*}
    where the colimit is computed in \( \Comp \).
    We let \( S \) be the set
    \begin{equation*}
        S \eqdef \set{s_P \colon \colim_{x \in P} \molecin{\imel{P}{x}} \to \molecin{P} \mid P \text{ finite regular directed complex}}.
    \end{equation*}
\end{dfn}

\begin{dfn}[Stricter \( \omega \)\nbd category] \label{dfn:stricter_omega_cat}
    A \emph{stricter \( \omega \)\nbd category} is a composition structure which is local with respect to \( S \).
    We let \( \somegaCat \) be the full subcategory of \( \Comp \) on stricter \( \omega \)\nbd categories.
\end{dfn}

\noindent Since \( \Comp \) is locally presentable and \( S \) is a small set, the full subcategory inclusion \( \iota \colon \somegaCat \incl \Comp \) is reflective \cite{freyd1972continuous}, and we denote by 
\begin{equation*}
    \rc \colon \Comp \to \somegaCat
\end{equation*}
the left adjoint of \( \iota \).

\begin{dfn} [Matching family and amalgamation]
    Let \( P \) be a regular directed complex and \( C \) be a composition structure.
    A \emph{\( P \)\nbd matching family in \( C \)} is a cone 
    \begin{equation*}
        \set{\F_x \colon \molecin{\imel{P}{x}} \to C}_{x \in P}
    \end{equation*}
    under the \( P \)\nbd shaped diagram \( x \mapsto \molecin{\imel{P}{x}} \).    
    An \emph{amalgamation} of this matching family is a strict functor 
    \begin{equation*}
        \amalg_{x\in P} \F_x \colon \molecin{P} \to C
    \end{equation*}
    such that, for all \( y \in P \), \( (\amalg_x \F_x) \after \molecin{\mapel{y}} = \F_y \).
\end{dfn}

\begin{rmk}
    Thus, a composition structure \( C \) is a stricter \( \omega \)\nbd category if for all finite regular directed complex \( P \), all \( P \)\nbd matching families in \( C \) have a unique amalgamation.
\end{rmk}

\begin{rmk}\label{rmk:data_matching family}
    The data of a \( P \)\nbd matching family 
    \begin{equation*}
        \set{\F_x \colon \molecin{\imel{P}{x}} \to C}_{x \in P}
    \end{equation*}
    in \( C \) is given by a family of globular cells \( \set{c_x \in C}_{x \in P} \).
    Indeed, given a matching family \( \set{\F_x}_{x \in P} \), define \( c_x \eqdef \pcell{\F_x} \).
    By functoriality, if \( x \le y \), then \( c_x = \F_y(\imel{P}{x} \incl \imel{P}{y}) \).
    Since \( \atomin{\imel{P}{y}} \) is a basis for \( \molecin{\imel{P}{y}} \), Lemma \ref{lem:strict_functor_determined_by_basis} implies that the data of \( \set{c_x}_{x \in \imel{P}{y}} \) entirely determines \( \F_y \colon \molecin{\imel{P}{y}} \to C \).
    Of course, not all data of this type give rise to a matching family. 
\end{rmk}

\begin{lem}\label{lem:at_most_one_lift}
    Let \( C \) be a composition structure, \( P \) be a regular directed complex, and \( \set{\F_x}_{x \in P} \) be a matching family. 
    Then \( \set{\F_x}_{x \in P} \) has at most one amalgamation.
\end{lem}
\begin{proof}
    Immediate by Lemma \ref{lem:strict_functor_determined_by_basis}.
\end{proof}

\begin{lem} \label{lem:well_define_from_regular_iff_well_defined_from_molecules}
    Let \( C \) be a composition structure, \( P \) be a regular directed complex, \( \set{\F_x \colon \molecin{\imel{P}{x}} \to C}_{x \in P} \) be a \( P \)\nbd matching family in \( C \). 
    Suppose that for all \( w \colon W \to P \) in \( \molecin{P} \), the \( W \)\nbd matching family \( \set{\F_{w(y)} \colon \imel{P}{w(y)} \to X}_{y \in W} \) in \( C \) has a well-defined amalgamation.
    Then \( \set{\F_x}_{x \in P} \) has a well-defined amalgamation.
\end{lem}
\begin{proof}
    For \( w \colon W \to P \) in \( \molecin{P} \), we let \( \F_w \) be the strict functor 
    \begin{equation*}
        \amalg_{y \in W} \F_{w(y)} \colon \molecin{W} \to C
    \end{equation*}
    We claim that \( \F \colon \molecin{P} \to C \) defined by \( w \mapsto \pcell{\F_w} \) is a strict functor.
    By construction, \( \F \) is a morphism of the underlying reflexive \( \omega \)\nbd graphs.
    Let \( w \colon W \to P \) in \( \molecin{P} \), and suppose that \( w = w^- \cp{k} w^+ \) for some \( k \geq 0 \) and \( w^\a \colon W^\a \to P \) in \( \molecin{P} \), for all \( \a \in \set{-, +} \).
    Then
    \begin{equation*}
        \pcell{\F_w} = \F_w(W^- \incl W) \comp{k} \F_w(W^+ \incl W) = \pcell{\F_{w^-}} \comp{k} \pcell{\F_{w^+}}.
    \end{equation*}
    Thus \( \F(w) = \F(w^-) \comp{k} \F(w^+) \).
    This concludes the proof.
\end{proof}

\begin{comm} \label{comm:well_defined_amalgamation}
    Given a \( P \)\nbd matching family \( \set{\F_x} \) in \( C \), we thus have a \emph{candidate amalgamation} \( \F \colon \molecin{P} \to C \) defined on the basis \( \atomin{P} \) by \( \mapel{x} \mapsto \pcell{\F_x} \).
    Then, \( \F \) is a well-defined strict functor if for all local embeddings \( u \colon U \to P \) with \( U \) a molecule, if \( u = u_1 \cp{k} u_2 \), then \( \F(u) = \F(u_1) \comp{k} \F(u_2) \).
    To show this, we may use \emph{induction on submolecules} (see \cite[Comment 4.1.7]{hadzihasanovic2024combinatorics}) as follows.
    Take an arbitrary element \( w \colon W \to P \) in \( \molecin{P} \), and prove that \( \F \after \molecin{w} \) is a well-defined strict functor under the hypothesis that for all proper subdiagrams \( \iota \colon W' \submol W \) of \( W \), \( \F \after \molecin{w \after \iota} \) is a well-defined strict functor.
    The base case on submolecules \( W' \) of \( W \) of dimension \( 0 \) is always true in that case. 
\end{comm}

\begin{lem} \label{lem:stricter_iff_local_wrt_pasting}
    Let \( C \) be composition structure.
    The following are equivalent.
    \begin{enumerate}
        \item \( C \) is a stricter \( \omega \)\nbd category;
        \item for all molecules \( U \), \( C \) is local with respect to \( s_U \);
        \item for all regular directed complexes \( P \), \( C \) is local with respect to \( s_P \);
        \item for all pairs of molecules \( U, V \), and \( k \in \mathbb{N} \) such that \( U \cp{k} V \) is defined, each lifting problem
        \begin{center}
            \begin{tikzcd}
                {\molecin{U} \cup \molecin{V}} & C \\
                {\molecin{(U \cp{k} V)}}
                \arrow[from=1-1, to=1-2]
                \arrow[from=1-1, to=2-1]
            \end{tikzcd}
        \end{center}
        has a (necessarily unique) solution.
    \end{enumerate}
\end{lem}
\begin{proof}
    The equivalences between the first three points directly follows from Lemma \ref{lem:at_most_one_lift} and Lemma \ref{lem:well_define_from_regular_iff_well_defined_from_molecules}.
    Next, the last condition is clearly necessary, since any functor \( \molecin{U} \cup \molecin{V} \to C \) defines in particular a \( (U \cp{k} V) \)\nbd matching family in \( C \).
    Conversely, we show it is sufficient.
    Let \( P \) be a regular directed complex.
    We show that \( C \) is local with respect to \( s_P \).
    Let \( \set{\F_x}_{x \in P} \) be a \( P \)\nbd matching family in \( C \).
    We show that the candidate amalgamation \( \F \colon \molecin{P} \to C \) is a strict functor as per Comment \ref{comm:well_defined_amalgamation}.
    Let \( w \colon W \to P \) in \( \molecin{P} \), and suppose that \( \F \after \molecin{w \after \iota} \) is well-defined for all proper submolecules \( \iota \colon W' \submol W \) of \( W \).
    Then either \( w \) is in \( \atomin{P} \), in which cases we are done since \( \F \after \molecin{w} = \F_x \) for the unique \( x \in P \) such that \( w = \mapel{x} \), or \( w = w_1 \cp{k} w_2 \) for some local embeddings \( w_1 \colon W_1 \to P \) and \( w_2 \colon W_2 \to P \).
    Then, by inductive hypothesis, we have a strict functor \( (\F \after \molecin{w_1}, \F \after \molecin{w_2}) \colon \molecin{W_1} \cup \molecin{W_2} \to C \).
    By hypothesis, this extends to a strict functor \( \F' \colon \molecin{(W_1 \cp{k} W_2)} \to C \), which is equal to \( \F \after \molecin{w} \) by Lemma \ref{lem:strict_functor_determined_by_basis}.
    This shows that \( \F \after \molecin{w} \) is well-defined and concludes the proof.
\end{proof}

\begin{prop} \label{prop:regular_directed_complex_stricter}
    Let \( P \) be a regular directed complex.
    Then \( \molecin{P} \) is a stricter \( \omega \)\nbd category.
\end{prop}
\begin{proof}
    Let \( Q \) be a regular directed complex, and consider a \( Q \)\nbd matching family \( \set{\F_x \colon \molecin{\imel{Q}{x}} \to \molecin{P}}_{x \in Q} \) in \( \molecin{P} \).
    We want to show that the candidate amalgamation \( \F \colon \molecin{Q} \to \molecin{P} \) is well-defined. 
    For each \( w \colon W \to Q \) in  \( \molecin{Q} \), \( \F(w) \) is given by the canonical local embedding \( \colim_{y \in W} \pcell{\F_{w(y)}} \to P \), which is independent of the chosen decomposition of \( w \).
    This concludes the proof.
\end{proof}

\noindent Therefore, the functor \( \molecin{-} \colon \rdcpxmap \to \Comp \) factors through the subcategory \( \somegaCat \).

\begin{cor} \label{cor:regular_directed_complex_colimit_of_itself}
    Let \( P \) be a regular directed complex.
    Then in \( \somegaCat \),
    \begin{equation*}
        s_P \colon \colim_{x \in P} \molecin{\imel{P}{x}} \cong \molecin{P}.
    \end{equation*}
\end{cor}
\begin{proof}
    By Lemma \ref{lem:stricter_iff_local_wrt_pasting}, \( \rc(s_P) \) is an isomorphism in \( \somegaCat \).
    Since \( \rc \) is left adjoint, we conclude by Proposition \ref{prop:regular_directed_complex_stricter}.
\end{proof}

\begin{cor} \label{cor:molecin_preserves_pushout_inclusions}
    The functor \( \molecin{-} \colon \rdcpxmap \to \somegaCat \) preserves all pushouts of inclusions. 
\end{cor}

\begin{dfn} [Stricter \( n \)\nbd category]
    Let \( n \in \mathbb{N} \).
    An \emph{\( n \)\nbd composition structure} is a composition structure \( C \) such that for all globular cells \( c \in C \), we have \( \dim c \le n \).
    If \( C \) was a stricter \( \omega \)\nbd category, we speak of \emph{stricter \( n \)\nbd category}.
    We denote by \( \nComp{n} \) and \( \snCat{n} \) the full subcategories of \( \Comp \) and \( \somegaCat \) on \( n \)\nbd composition structures and stricter \( n \)\nbd categories, respectively. 
\end{dfn}

\begin{dfn} 
    Let \( n \in \mathbb{N} \).
    The inclusion \( \iota_n \colon \nComp{n} \incl \Comp \) has a right adjoint \( \skel{n} \) defined by
    \begin{equation*}
        \skel{n}(C) \eqdef \set{c \in C \mid \dim c \le n},
    \end{equation*}
    and a left adjoint \( \trunc{n} \) defined by
    \begin{equation*}
        \trunc{n}(C) \eqdef \skel{n - 1}(C) \cup \set{[c] \mid c \in C, \dim c = n},
    \end{equation*}
    where \( [-] \) denotes the equivalence class on the globular cells of \( C \) of dimension \( n \) generated by \( \bd{}{-} d \sim \bd{}{+} d \), for all globular cells \( d \) of dimension \( n + 1 \). 
    By convention, \( \skel{-1}(C) = \varnothing \).
    It is also convenient to let \( \trunc{\omega} \) be the identity.
\end{dfn}

\begin{rmk}
    By \cite[Proposition 5.2.14]{hadzihasanovic2024combinatorics}, if \( P \) is a regular directed complex, \( \skel{n} \molecin{P} \) is naturally isomorphic to \( \molecin{(\gr{\le n}P)} \).
\end{rmk}

\begin{lem} \label{lem:stricter_n_iff_local_with_dim_le_n}
    Let \( C \) be an \( n \)\nbd composition structure.
    The following are equivalent.
    \begin{enumerate}
        \item \( C \) is a stricter \( n \)\nbd category;
        \item for all molecules \( U \) with \( \dim U \le n \), \( C \) is local with respect to \( s_U \);
        \item for all regular directed complexes with \( \dim P \le n \), \( C \) is local with respect to \( s_P \);
        \item for all pairs of molecules \( U, V \) with \( \dim U, \dim V \le n \) and \( k \in \mathbb{N} \) such that \( U \cp{k} V \) is defined, each lifting problem
            \begin{center}
                \begin{tikzcd}
                    {\molecin{U} \cup \molecin{V}} & C \\
                    {\molecin{(U \cp{k} V)}}
                    \arrow[from=1-1, to=1-2]
                    \arrow[from=1-1, to=2-1]
                \end{tikzcd}
            \end{center}
            has a (necessarily unique) solution.
    \end{enumerate}
\end{lem}
\begin{proof}
    The second and third condition are equivalent by Lemma \ref{lem:well_define_from_regular_iff_well_defined_from_molecules}.
    Suppose \( C \) is a stricter \( n \)\nbd category.
    In particular, \( C \) is a stricter \( \omega \)\nbd category, so by Lemma \ref{lem:stricter_iff_local_wrt_pasting}, the second and third conditions hold.
    Now suppose the third condition holds, consider a regular directed complex \( P \) and a \( P \)\nbd matching family \( \set{\F_x \colon \molecin{\imel{P}{x}} \to C} \) with candidate amalgamation \( \F \).
    Restricting this matching family to \( \set{\F_x}_{x \in \gr{\le n}{P}} \), and using the assumption, we have an amalgamation 
    \begin{equation*}
        \gr{\le n}{\F} \eqdef \amalg_{x \in \gr{\le n}{P}} \F_x \colon \gr{\le n}{P} \to C.
    \end{equation*}
    Let \( x \in P \) with \( \dim x > n \).
    Then \( \dim \pcell{\F_x} \le n \), hence for any \( \a \in \set{-, +} \), \( \pcell{\bd{n}{\a} \F_x} = \pcell{\F(\bd{n}{\a} x \to P)} \).
    Using this fact, an induction on the submolecules of any \( w \colon W \to P \) shows that \( \F(w) = \gr{\le n}{\F}(\bd{n}{\a} w) \), proving that \( \F \) is well-defined.
    This shows that \( C \) is stricter.
    The last condition is shown equivalent to the third one using a similar argument as in the proof of Lemma \ref{lem:stricter_iff_local_wrt_pasting}.
    This concludes the proof.
\end{proof}

\begin{lem} \label{lem:truncation_stricter_are_stricter}
    Let \( n \geq 0 \), and \( C \) be a stricter \( \omega \)\nbd category.
    Then \( \skel{n}(C) \) and \( \trunc{n}(C) \) are stricter \( n \)\nbd categories.
\end{lem}
\begin{proof}
    We use the third point of Lemma \ref{lem:stricter_n_iff_local_with_dim_le_n}.
    Let \( P \) be a regular directed complex with \( \dim P \le n \), and \( \set{\F_x \colon \molecin{\imel{P}{x}} \to \skel{n}(C)}_{x \in P} \) be a \( P \)\nbd matching family in \( \skel{n}(C) \).
    Then post-composing each \( \F_x \) by the counit \( \counit_C \colon \skel{n}(C) \to C \) defines a \( P \)\nbd matching family \( \set{\counit \after \F_x}_{x \in P} \) in \( C \).
    Since \( C \) is stricter, it has an amalgamation \( \F' \colon \molecin{P} \to C \).
    Then, since \( \skel{n}(P) = P \), \( \skel{n}(\F') \) is the desired amalgamation of the matching family \( \set{\F_x}_x \).

    Now consider a \( P \)\nbd matching family \( \set{\G_x \colon \molecin{\imel{P}{x}} \to \trunc{n}(C)}_{x \in P} \) in \( \trunc{n}C \).
    By definition of \( \trunc{n}(C) \), and since \( \skel{n - 1}(C) \subseteq \trunc{n}(C) \) is stricter, we have by the first part of the proof the amalgamation
    \begin{equation*}
        \gr{< n }{\G} \eqdef \amalg_{x \in \gr{< n}{P}} \G_x \colon \molecin{\gr{< n}{P}} \to \trunc{n}(C).
    \end{equation*}
    For each \( x \in P \), choose a representative \( c_x \in C \) of the cell \( \pcell{\G_x} \) in \( \trunc{n}(C) \) (if \( \dim \pcell{\G_x} \le n \), then we mean that we let \( c_x \eqdef \pcell{\G_x} \)).
    We claim that \( \set{c_x}_{x \in P} \) give rise to a matching family \( \set{\G'_x \colon \molecin{\imel{P}{x}} \to C}_{x \in P} \), as per Remark \ref{rmk:data_matching family}.
    This is clear for all \( x \in \gr{< n}{P} \), since this is the data associated to the matching family \( \set{\G_x}_{x \in \gr{< n}{P}} \). 
    If \( \dim x = n \), then we claim that \( \bd{}{}\G'_x \colon \molecin{\bd{}{}\imel{P}{x}} \to C \) extends to a strict functor \( \G'_x \colon \molecin{\imel{P}{x}} \to C \) by letting \( \idd{x} \mapsto c_x \).
    Indeed, for all \( k < n \) and \( \a \in \set{-, +} \), we have \( \bd{k}{\a} c_x = \gr{< n}{G}(\bd{k}{\a} \mapel{x} ) \), since the former is independent of the chosen representative \( c_x \) of \( \pcell{\G_x} \).
    Thus \( \set{\G'_x}_{x \in P} \) is a \( P \)\nbd matching family in \( C \).
    Since \( C \) is stricter, it admits an amalgamation \( \G' \colon \molecin{P} \to C \). 
    Post-composing \( \G' \) with the unit \( \unit_C \colon C \to \trunc{n}(C) \), we obtain the strict functor \( \unit_C \after \G' \), which is the candidate amalgamation \( \G \) by Lemma \ref{lem:strict_functor_determined_by_basis}.
    This shows that \( \trunc{n}(C) \) is stricter, and concludes the proof.
\end{proof}

\begin{dfn}[\( n \)\nbd skeleton and \( n \)\nbd truncation]
    Let \( n \geq 0 \) and \( C \) be a stricter \( n \)\nbd category.
    The \emph{\( n \)\nbd skeleton} of \( C \) is the stricter \( n \)\nbd category \( \skel{n}(C) \).
    The \emph{\( n \)\nbd truncation} of \( C \) is the stricter \( n \)\nbd category \( \trunc{n}(C) \).
\end{dfn}

\noindent The adjoint triple \( \trunc{n} \dashv \iota_n \dashv \skel{n} \) restricts the adjoint triple
\begin{center}
    \begin{tikzcd}
        {\snCat{n}} && \somegaCat.
        \arrow["{\iota_n}"{description}, from=1-1, to=1-3]
        \arrow["{\skel{n}}", shift left=2, curve={height=-12pt}, from=1-3, to=1-1]
        \arrow["{\trunc{n}}"', shift right=2, curve={height=12pt}, from=1-3, to=1-1]
    \end{tikzcd}
\end{center}
Notice that given a stricter \( \omega \)\nbd category \( C \), we have a chain of inclusions 
\begin{equation*}
    \skel{-1} C \incl \skel{0} C \incl \skel{1} C \incl \ldots \incl \skel{n} C \incl \ldots
\end{equation*}
whose colimit \( \somegaCat \) is \( C \).

The following definition is adapted from \cite[8.2.1]{hadzihasanovic2024combinatorics}.
\begin{dfn} [Cellular extension] \label{dfn:cellular_extension}
    Let \( C \) be a stricter \( \omega \)\nbd category.
    A \emph{cellular extension of \( C \)} is a stricter \( \omega \)\nbd category \( C_{\cls{S}} \) together with a pushout diagram 
    \begin{center}
        \begin{tikzcd}[column sep=large]
            {\coprod_{e \in \cls{S}} \bd{}{}U_e} & {\coprod_{e \in \cls{S}} \molecin{U_e}} \\
            C & {C_{\cls{S}}}
            \arrow[""{name=0, anchor=center, inner sep=0}, "{\molecin{\bd{e}{}}}", from=1-1, to=1-2]
            \arrow["{(\bd{}{}e)_{e \in \cls{S}}}"', from=1-1, to=2-1]
            \arrow["{(e)_{e \in \cls{S}}}", from=1-2, to=2-2]
            \arrow[from=2-1, to=2-2]
            \arrow["\lrcorner"{anchor=center, pos=0.125, rotate=180}, draw=none, from=2-2, to=0]
        \end{tikzcd}
    \end{center}
    in \( \somegaCat \), where, for each \( e \in \cls{S} \), \( U_e \) is an atom.
\end{dfn}

\begin{comm}  \label{comm:def_cellular_extension}
    This is a non-standard definition of cellular extension, which is a priori more general than the usual one, that would require each of the atoms \( U_e \) to be globes.
    However, by Corollary \ref{cor:pushout_principal_cell} applied to the unique subdivision \( s \colon \dglobe{{\dim U}} \sd U \), we may turn every cellular extension in our sense into one in the restricted sense. 
    See \cite[Comment 8.2.2]{hadzihasanovic2024combinatorics}. 
\end{comm}

\begin{dfn} [Generalised substitution]
    Let \( U, V, P \) be regular directed complexes, \( s \colon U \sd V \) be a subdivision with formal dual \( c \), and \( \iota \colon U \incl P \) be an inclusion.
    The \emph{generalised substitution of \( U \) for \( V \) in \( P \)} is the oriented graded poset \( \subs{P}{V}{U}_s \) whose underlying set is \( (P \setminus \iota(U)) \coprod V \), with partial order defined for each \( z \in \subs{P}{V}{U}_s  \) and \( \a \in \set{-, +} \), by
    \begin{equation*}
        \cofaces{}{\a} z \eqdef 
        \begin{cases}
            \cofaces{P}{\a} z, & \text{if } z \in P \setminus \iota(U), \\
            \cofaces{V}{\a} z + \cup \set{\cofaces{P}{\a} y \mid y = \iota(c(z)), \dim y = \dim z} & \text{if } z \in V.
        \end{cases}
    \end{equation*}
    This comes equipped with the following commutative square in \( \Pos \)
    \begin{center}
        \begin{tikzcd}
            U & P \\
            V & {{\subs P V U}_s.}
            \arrow["\iota", hook, from=1-1, to=1-2]
            \arrow["c", from=2-1, to=1-1]
            \arrow["{\iota'}"', hook, from=2-1, to=2-2]
            \arrow["{c'}"', from=2-2, to=1-2]
        \end{tikzcd}
    \end{center}
\end{dfn}

\begin{lem} \label{lem:generalised_substitution}
    Let \( U, V, P \) be regular directed complexes, \( s \colon U \sd V \) be a subdivision with formal dual \( c \), and \( \iota \colon U \incl P \) be an inclusion.
    Then \( \subs{P}{V}{U}_s \) is a regular directed complex, the formal dual of \( c' \colon \subs{P}{V}{U}_s \to P \) is a subdivision \( s' \), and the square of strict functors
    \begin{center}
        \begin{tikzcd}
            {\molecin{U}} & {\molecin{P}} \\
            {\molecin{V}} & {\molecin{({\subs P V U}_s)}}
            \arrow[""{name=0, anchor=center, inner sep=0}, "{{{\molecin{\iota}}}}", from=1-1, to=1-2]
            \arrow["{{{\molecin{s}}}}"', from=1-1, to=2-1]
            \arrow["{{{\molecin{s'}}}}", from=1-2, to=2-2]
            \arrow["{\molecin{\iota'}}"', from=2-1, to=2-2]
            \arrow["\lrcorner"{anchor=center, pos=0.125, rotate=180}, draw=none, from=2-2, to=0]
        \end{tikzcd}
    \end{center}
    is a pushout in \( \somegaCat \).
\end{lem}
\begin{proof}
    That \( \subs{P}{V}{U}_s \) is a regular directed complex and \( s' \colon P \sd \subs{P}{V}{U}_s \) is a subdivision follows from \cite[Proposition 1.46, Comment 1.48]{chanavat2025semistrict}.
    We prove that the square is a pushout.
    Consider a stricter \( \omega \)\nbd category \( C \), together with strict functors \( \F \colon \molecin{P} \to C \) and \( \G \colon \molecin{V} \to C \) such that \( \F \after \molecin{\iota} = \G \after \molecin{s} \).
    Let \( \set{\F_x}_{x \in P} \) and \( \set{\G_x}_{x \in V} \) be respectively the matching families that \( \F \) and \( \G \) are the amalgamation of.
    For each \( x \in {\subs P V U}_s \), we let 
    \begin{equation} \label{eq:H_matching_of}
        \fun{H}_x \eqdef 
        \begin{cases}
            \F_x & \text{if } x \in V \\
            \G_x & \text{else}.
        \end{cases}
    \end{equation}
    By construction, \( \set{\fun{H}_x}_{x \in {\subs P V U}_s} \) is a matching family in \( C \).
    Since \( C \) is stricter, it has an amalgamation \( \fun{H} \colon \molecin{({\subs P V U}_s)} \to C \).
    By Lemma \ref{lem:strict_functor_determined_by_basis}, \( \fun{H} \after \molecin{\iota'} = \F \) and \( \fun{H} \after \molecin{s'} = \G \).
    Any other strict functor \( \fun{H}' \) with this property is necessarily the amalgamation of the same matching family.
    By Lemma \ref{lem:at_most_one_lift}, this proves uniqueness and concludes the proof. 
\end{proof}

\begin{cor} \label{cor:pushout_principal_cell}
    Let \( s \colon U \sd V \) be a subdivision between atoms.
    Then the square
    \begin{center}
        \begin{tikzcd}
            {\molecin{(\bd{}{}U)}} & {\molecin{U}} \\
            {\molecin{(\bd{}{}V)}} & {\molecin{V}}
            \arrow[from=1-1, to=1-2]
            \arrow["{\molecin{(\restr{s}{\bd{}{}U})}}"', from=1-1, to=2-1]
            \arrow["{\molecin{s}}", from=1-2, to=2-2]
            \arrow[from=2-1, to=2-2]
        \end{tikzcd}
    \end{center}
    is a pushout square in \( \somegaCat \).
\end{cor}
\begin{proof}
    By Lemma \ref{lem:generalised_substitution}.
\end{proof}

\begin{dfn} [Stricter polygraph] \label{dfn:stricter_polygraph}
    A \emph{stricter polygraph} is a stricter \( \omega \)\nbd category \( C \), together with, for each \( n \geq 0 \), a pushout diagram
    \begin{center}
        \begin{tikzcd}[column sep=large]
            {\coprod_{e \in \cls{S}_n} \bd{}{}U_e} & {\coprod_{e \in \cls{S}_n} \molecin{U_e}} \\
            \skel{n - 1}C & {\skel{n} C}
            \arrow[""{name=0, anchor=center, inner sep=0}, "{\molecin{\bd{e}{}}}", from=1-1, to=1-2]
            \arrow["{(\bd{}{}e)_{e \in \cls{S}_n}}"', from=1-1, to=2-1]
            \arrow["{(e)_{e \in \cls{S}_n}}", from=1-2, to=2-2]
            \arrow[from=2-1, to=2-2]
            \arrow["\lrcorner"{anchor=center, pos=0.125, rotate=180}, draw=none, from=2-2, to=0]
        \end{tikzcd}
    \end{center}
    in \( \somegaCat \), exhibiting \( \skel{n}C \) as a cellular extension of \( \skel{n - 1}C \), and such that for each \( e \in \cls{S}_n \), \( \dim e = n \).
    The set \( \cls{S} = \coprod_{k \geq 0} \cls{S}_k \) is called the \emph{generating set of globular cells} of \( C \).
    More generally, we say that a strict functor \( \F \colon A \to C \) of composition structure is a \emph{relative stricter polygraph} if \( \F \) can be obtained as the transfinite composition of cellular extensions of \( A \).
\end{dfn}

\begin{lem} \label{lem:stricter_polygraph_basis}
    Let \( (C, \cls{S}) \) be a stricter polygraph.
    Then the set \( \cls{S} \) is a basis for \( C \).
\end{lem}
\begin{proof}
    The proof is similar to the case of polygraphs, see \cite[Proposition 15.1.8, Lemma 16.6.2]{ara2025polygraphs}.
\end{proof}

\begin{lem} \label{lem:stricter_regular_complex_are_stricter_polygraph}
    Let \( P \) be a regular direct complex.
    Then \( \molecin{P} \) is a stricter polygraph.
\end{lem}
\begin{proof}
    By Corollary \ref{cor:molecin_preserves_pushout_inclusions}, the square
    \begin{center}
        \begin{tikzcd}
            {\coprod_{x \in \gr{n}P} \molecin{\bd{}{}\imel P x}} & {\coprod_{x \in \gr{n}P} \molecin{\imel P x}} \\
            {\skel{n - 1}P} & {\skel{n} P}
            \arrow["{{\molecin{\bd{e}{}}}}", from=1-1, to=1-2]
            \arrow["{{(\bd{}{}\molecin{\mapel x})_{x \in \gr n P}}}"', from=1-1, to=2-1]
            \arrow["{{(\molecin{\mapel x})_{x \in \gr n P}}}", from=1-2, to=2-2]
            \arrow[from=2-1, to=2-2]
        \end{tikzcd}
    \end{center}
    is a pushout in \( \somegaCat \).
    This concludes the proof.
\end{proof}

\noindent We conclude this section by describing \( \somegaCat \) as a subcategory of a category of presheaves.

\begin{dfn}
    We let \( \omegaReg \) be the full subcategory of \( \somegaCat \) on stricter \( \omega \)\nbd categories of the form \( \molecin{P} \), for \( P \) a \emph{finite} regular directed complex.
\end{dfn}

\begin{comm}
    The finiteness condition is only here to ensure that \( \omegaReg \) is a small category. 
\end{comm}

\begin{dfn} [Rewrite and pasting condition] \label{dfn:rewrite_condition}
    Let \( X \colon \opp{\omegaReg} \to \Set \) be a presheaf. 
    We say that \( X \) satisfies \emph{the rewrite condition} if for all atoms \( U \) of dimension \( n \in \mathbb{N} \), letting \( s \colon \dglobe{n} \sd U \) be the unique subdivision, \( X \) sends the square
    \begin{center}
        \begin{tikzcd}
            {\molecin{(\bd{}{}\dglobe{n})}} & {\molecin{\dglobe{n}}} \\
            {\molecin{(\bd{}{}U)}} & {\molecin{U}}
            \arrow[""{name=0, anchor=center, inner sep=0}, from=1-1, to=1-2]
            \arrow["{{\molecin{(\restr{s}{\bd{}{}\dglobe{n}})}}}"', from=1-1, to=2-1]
            \arrow["{{\molecin{s}}}", from=1-2, to=2-2]
            \arrow[from=2-1, to=2-2]
        \end{tikzcd}
    \end{center}
    to a pullback square.
    We say that \( X \) satisfies the \emph{pasting condition} if for all finite regular directed complex \( P \), \( X \) sends the cone 
    \begin{equation*}
        \set{\molecin{\imel{P}{x}} \to \molecin{P}}_{x \in P}
    \end{equation*}
    over \( \molecin{P} \) to a limit cone.
    We let \( \omegaRegSet \) be the full subcategory of \( [\opp{\omegaReg}, \Set] \) on presheaves satisfying the rewrite and pasting conditions.
\end{dfn}

\begin{rmk}
    Let \( X \colon \opp{\omegaReg} \to \Set \) be a presheaf, and \( P \) a finite regular directed complex. 
    We may still call a cone \( \set{t_x \colon \molecin{\imel{P}{x}} \to X}_{x \in P} \) a \( P \)\nbd matching family in \( X \), and an extension along \( s_P \) an amalgamation.
    Then \( X \) satisfies the pasting condition if for all finite regular directed complexes \( P \), all \( P \)\nbd matching families in \( X \) have a unique amalgamation.
\end{rmk}

\begin{dfn}
    Let \( X \colon \opp{\omegaReg} \to \Set \) be a presheaf, \( U \) be a molecule of dimension \( n \geq 0 \), and \( t \colon \molecin{U} \to X \) a natural transformation.
    Then there exists a unique strict functor \( \F \colon \molecin{\dglobe{n}} \to \molecin{U} \) classifying \( \idd{U} \).
    The \emph{principal cell of \( t \)} is the natural transformation \( \pcell{t} \eqdef t \after \F \colon \molecin{\dglobe{n}} \to X \).
\end{dfn}

\begin{dfn}
    Let \( X \colon \opp{\omegaReg} \to \Set \) be a presheaf satisfying the rewrite and pasting conditions.
    The \emph{composition structure associated to \( X \)} is the composition structure \( \phi(X) \eqdef \coprod_{n \geq 0} X(\dglobe{n}) \), whose boundary operators are induced, for all \( k \in \mathbb{N} \) and \( \a \in \set{- ,+} \), by the strict functors
    \begin{equation*}
        \bd{k}{\a} \colon \molecin{\dglobe{k}} \to \molecin{\dglobe{n}}, 
    \end{equation*}
    and \( k \)\nbd composition operation is induced by the strict functor
    \begin{equation*}
        \molecin{\dglobe{n}} \to \molecin{(\dglobe{n} \cp{k} \dglobe{n})},
    \end{equation*}
    classifying \( \idd{(\dglobe{n} \cp{k} \dglobe{n})} \).
    This construction extends to a functor 
    \begin{equation*}
        \phi \colon \omegaRegSet \to \Comp.
    \end{equation*}
\end{dfn}

\begin{lem} \label{lem:stricter_hom_set_iso_presheaf}
    Let \( X \colon \opp{\omegaReg} \to \Set \) be a presheaf satisfying the rewrite and pasting conditions, and \( P \) be a finite regular directed complex.
    Then there is a bijection
    \begin{equation*}
        X(\molecin{P}) \cong \Comp(\molecin{P}, \phi(X)),
    \end{equation*}
    natural in \( X \) and \( U \).
\end{lem}
\begin{proof}
    Let \( t \in X(\molecin{P}) \), which by the Yoneda Lemma, is a natural transformation \( t \colon \molecin{P} \to X \).
    We define \( \F \colon \molecin{P} \to \phi(X) \) by sending \( w \) in \( \molecin{P} \) to \( \pcell{t \after \molecin{w}} \), which is, by definition of \( \phi(X) \), a strict functor of composition structures.
    Conversely, let \( \F \colon \molecin{P} \to \phi(X) \) be a strict functor, and let \( \set{\F_x}_{x \in P} \) be the matching family of which it is the amalgamation.
    We construct by induction on \( n \), a matching family \( \set{t_x \colon \molecin{\imel{P}{x}} \to X}_{x \in \gr{\le n}{P}} \) such that \( \pcell{t_x} = \pcell{\F_x} \).    
    For \( x \in \gr{0}{P} \), we let \( t_x \colon \molecin{\pt} \to X \) classifying \( \pcell{\F_x} \in X(\globe{0}) \).
    Let \( n > 0 \) and suppose inductively that the matching family is constructed for all \( x \in \gr{< n}{P} \).
    Let \( x \in \gr{n}{P} \) and let \( s \colon \dglobe{n} \sd \imel{P}{x} \) be the unique subdivision.
    Then we have a commutative diagram
    \begin{center}
        \begin{tikzcd}
            {\molecin{(\bd{}{}\dglobe{n})}} & {\molecin{\dglobe{n}}} \\
            {\molecin{(\bd{}{}\imel{P}{x})}} & X
            \arrow[from=1-1, to=1-2]
            \arrow["{{\molecin{(\restr{s}{\bd{}{}\dglobe{n}})}}}"', from=1-1, to=2-1]
            \arrow["{\pcell{\F_x}}", from=1-2, to=2-2]
            \arrow["{t_{\bd{}{} \imel{P}{x}}}"', from=2-1, to=2-2]
        \end{tikzcd}
    \end{center}
    where \( t_{\bd{}{} \imel{P}{x}} \) is the amalgamation of \( \set{t_y \colon \molecin{\imel{P}{y}} \to X}_{y \in \bd{}{} \imel{P}{x}} \).
    Since \( X \) satisfies the rewrite condition, this data defines a unique natural transformation \( t_x \colon \molecin{\imel{P}{x}} \to X \) such that \( t_x \after \molecin{s} = \pcell{\F_x} \).
    This concludes the construction of the matching family.
    Since \( X \) satisfies the pasting condition, the matching family admits an amalgamation \( t \colon \molecin{P} \to X \), defining the desired element of \( X(\molecin{P}) \).
    It is straightforward to check naturality and that those two constructions are inverse of each other.
\end{proof}

\begin{prop} \label{prop:stricter_cat_are_local_presheaves}
    The categories \( \omegaRegSet \) and \( \somegaCat \) are equivalent, realised by the functors \( X \mapsto \phi(X) \) and \( C \mapsto \somegaCat(\molecin{-}, C) \). 
\end{prop}
\begin{proof}
    For all presheaves \( X \) in \( \omegaRegSet \), we claim that \( \phi(X) \) is a stricter \( \omega \)\nbd category.
    Indeed, a \( P \)\nbd matching family \( \set{\F_x \colon \molecin{\imel{P}{x}} \to \phi(X)}_{x \in P} \) defines, by Lemma \ref{lem:stricter_hom_set_iso_presheaf} a matching family \( \set{t_x \colon \molecin{\imel{P}{x}} \to X} \) which has an amalgamation \( t \colon \molecin{P} \to X \), which, by Lemma \ref{lem:stricter_hom_set_iso_presheaf} again, defines a strict functor \( \F \colon \molecin{P} \to \phi(X) \), which is the amalgamation of \( \set{\F_x}_{x \in P} \). 
    Thus \( X \mapsto \phi(X) \) defines a functor \( \omegaRegSet \to \somegaCat \).
    We construct an inverse up to natural isomorphism to \( \phi \).
    Let \( C \) be a stricter \( \omega \)\nbd category.
    Then the presheaf defined by \(\molecin{P} \mapsto \somegaCat(\molecin{P}, C) \) is \( S \)\nbd local by definition, and satisfy the rewrite condition by Corollary \ref{cor:pushout_principal_cell}.
    This defines a functor \( \psi \colon \somegaCat \to \omegaRegSet \).
    One sees directly that \( \phi \after \psi \) is isomorphic to the identity on \( \somegaCat \). 
    By Lemma \ref{lem:stricter_hom_set_iso_presheaf}, it is also the case of \( \psi \after \phi \).
    This concludes the proof.
\end{proof}

\subsection{Gray product of stricter \texorpdfstring{$\omega$}{ω}-categories}

Recall that if \( P, Q \) are regular directed complexes, the basis \( \atomin{(P \gray Q)} \) of \( \molecin{(P \gray Q)} \) is given exactly by the local embeddings \( u \gray v \) for \( u \in \atomin{P} \) and \( v \in \atomin{Q} \).  

\begin{dfn} [Gray product] \label{dfn:gray_product_of_stricter_regular_complexes}
    Let \( P, Q \) be regular directed complexes.
    The \emph{Gray product} of \( \molecin{P} \) and \( \molecin{Q} \) is the stricter \( \omega \)\nbd category
    \begin{equation*}
        \molecin{P} \gray \molecin{Q} \eqdef \molecin{(P \gray Q)}.
    \end{equation*}
    If \( \F \colon \molecin{P} \to \molecin{P'} \) and \( \G \colon \molecin{Q} \to \molecin{P'} \) are two strict functors, we let 
    \begin{equation*}
        \F \gray \G \colon \atomin{P} \times \atomin{Q} \to  \molecin{P'} \gray \molecin{Q'}
    \end{equation*}
    defined by sending a pair \( (u, v) \) to \( \F(u) \gray \G(v) \).
\end{dfn}

\begin{lem} \label{lem:gray_of_strict_functors}
    Let \( \F \colon \molecin{P} \to \molecin{P'} \) and \( \G \colon \molecin{Q} \to \molecin{P'} \) be strict functors. 
    Then \( \F \gray \G \) extends uniquely to a strict functor 
    \begin{equation*}
        \F \gray \G \colon \molecin{P} \gray \molecin{Q} \to \molecin{P'} \gray \molecin{Q'}.
    \end{equation*}
\end{lem}
\begin{proof}
    Suppose that \( \F \) and \( \G \) are given by the matching families \( \set{\F_x}_{x \in P} \) and \( \set{\G_y}_{y \in Q} \) respectively.
    We show by induction on \( n \) that the data of \( \set{\pcell{\F_x} \gray \pcell{\G_y}}_{\dim (x, y) \le n} \) determines, as per Remark \ref{rmk:data_matching family}, a \( (\gr{\le n}{(P \gray Q)}) \)\nbd matching family in \( \molecin{(P' \gray Q')} \).
    When \( n = 0 \), the statement is clear. 
    Inductively, let \( n > 0 \) and let \( (x, y) \in P \gray Q \) of dimension \( n \). 
    Then by inductive hypothesis, \( \set{\F_{x'} \gray G_{y'}}_{(x', y') < (x, y)} \) forms a \( \bd{}{} (x, y) \)\nbd matching family in \( \molecin{(P' \gray Q')} \).
    By Proposition \ref{prop:regular_directed_complex_stricter}, this matching family has an amalgamation \( \bd{}{}\fun{H} \colon \molecin{(\bd{}{} (x, y))} \to \molecin{(P' \gray Q')} \).
    We show that \( \bd{}{} \fun{H} \) extends to a strict functor \( \fun{H} \colon \molecin{\clset{(x, y)}} \to \molecin{(P' \gray Q')} \) by sending the principal cell to \( \pcell{\F_x} \gray \pcell{\G_y} \).
    For \( k \in \mathbb{N} \) and \( \a \in \set{-, +} \), we have
    \begin{align*}
        \bd{k}{\a} (\pcell{\F_x} \gray \pcell{\G_y}) &= \bigcup_{i = 1}^k \bd{i}{\a} \pcell{\F_x} \gray \bd{k - i}{(-)^i\a} \pcell{\G_y} \\
                                       &= \bigcup_{i = 1}^k \pcell{\bd{i}{\a} \F_x} \gray \pcell{\bd{k - i}{(-)^i\a} \G_y} \\
                                       &= \fun{H}(\bd{k}{\a} \idd{(x, y)}).
    \end{align*}
    Thus, \( \fun{H} \) is a morphism of the underlying globular graph, and since \( \idd{(x, y)} \) is the only globular cell of dimension \( \dim (x, y) \) in \( \molecin{(\clset{(x, y)})} \), this is enough to conclude that \( \fun{H} \) is a strict functor.
    This concludes the induction and shows that \( \set{\pcell{\F_x} \gray \pcell{\G_y}}_{(x, y) \in P \gray Q} \) determines a matching family in \( \molecin{(P' \gray Q')} \).
    By Proposition \ref{prop:regular_directed_complex_stricter}, this matching family has an amalgamation \( \F \gray \G \).
    Uniqueness is given by Lemma \ref{lem:strict_functor_determined_by_basis}.
    This concludes the proof.
\end{proof}

\begin{cor} \label{cor:gray_prodcut_omegareg_monoidal}
    The Gray product determines a monoidal structure on the category \( \omegaReg \), whose monoidal unit is the terminal stricter \( \omega \)\nbd category \( \globe{0} \).
\end{cor}
\begin{proof}
    Lemma \ref{lem:gray_of_strict_functors} shows that \( - \gray - \) is well-defined on strict functors. 
    Since for all regular directed complexes \( P \), \( \pt \gray P = P = P \gray \pt \), we deduce that the monoidal unit is \( \molecin{\pt} = \globe{0} \).
    Functoriality is straightforward.
    This concludes the proof.
\end{proof}

\begin{rmk}
    If \( \F \colon \molecin{P} \to \molecin{P'} \) and \( \G \colon \molecin{Q} \to \molecin{P'} \) are strict functors, \( u \in \molecin{P} \), and \( v \in \molecin{Q} \), then 
    \begin{equation*}
       (\F \gray \G)(u \gray v) = \F(u) \gray \G(v). 
    \end{equation*}
\end{rmk}

\begin{lem} \label{lem:rewrite_condition_for_gray}
    Let \( P \) be a regular directed complex, and \( s \colon U \sd V \) be a subdivision between atoms.
    Then the square
    \begin{center}
        \begin{tikzcd}
            {\molecin{(P \gray \bd{}{}U)}} & {\molecin{(P \gray U)}} \\
            {\molecin{(P \gray \bd{}{}V)}} & {\molecin{(P \gray V)}}
            \arrow[from=1-1, to=1-2]
            \arrow["{\molecin{(\restr{P \gray s}{\bd{}{}U})}}"', from=1-1, to=2-1]
            \arrow["{{\molecin{P \gray s}}}", from=1-2, to=2-2]
            \arrow[from=2-1, to=2-2]
        \end{tikzcd}
    \end{center}
    is a pushout square in \( \somegaCat \).
\end{lem}
\begin{proof}
    By construction, \( P \gray V \) is the generalised substitution of \( P \gray (\bd{}{} U) \) for \( P \gray (\bd{}{} V) \) in \( P \gray U \).
    We conclude by Lemma \ref{lem:generalised_substitution}.
\end{proof}

\begin{dfn} 
    Let \( P \) be a regular directed complex and \( C \) be a stricter \( \omega \)\nbd category.
    By \cite[Lemma 7.2.8]{hadzihasanovic2024combinatorics} and Lemma \ref{lem:rewrite_condition_for_gray}, the presheaf on \( \omegaReg \) given by
    \begin{equation*}
        \molecin{U} \mapsto \somegaCat(\molecin{(P \gray U)}, C) 
    \end{equation*}
    satisfies the rewrite and pasting conditions.
    By Proposition \ref{prop:stricter_cat_are_local_presheaves}, this defines the stricter \( \omega \)\nbd category 
    \begin{equation*}
        \homlax(\molecin{P}, C),
    \end{equation*}
    whose globular \( n \)\nbd cells are strict functors \( \F \colon \molecin{(P \gray \dglobe{n})} \to C \).
    Dually, we can also define the stricter \( \omega \)\nbd category
    \begin{equation*}
        \homcolax(\molecin{P}, C),
    \end{equation*}
    whose globular \( n \)\nbd cells are strict functors \( \F \colon \molecin{(\dglobe{n} \gray P)} \to C \).
\end{dfn}

\begin{lem}
    Let \( P, Q \) be finite regular directed complexes and \( C \) be a stricter \( \omega \)\nbd category.
    Then there are bijections
    \begin{align*}
        \somegaCat(\molecin{(P \gray Q)}, C) &\cong \somegaCat(\molecin{P}, \homcolax(\molecin{Q}, C)) \\
                                             &\cong \somegaCat(\molecin{Q}, \homlax(\molecin{P}, C)),
    \end{align*}
    natural in \( \molecin{P}, \molecin{Q} \) and \( C \).
\end{lem}
\begin{proof}
    Follows directly by Proposition \ref{prop:stricter_cat_are_local_presheaves}.
\end{proof}

\noindent Applying \cite[Th\'eor\`eme 5.3]{ara2020joint} together with the previous result, we get the following definition.
\begin{dfn} [Gray product of stricter \( \omega \)\nbd categories] \label{dfn:gray_product_stricter_categories} 
    Let \( C, D \) be stricter \( \omega \)\nbd categories.
    The Gray product of \( C \) and \( D \) is the stricter \( \omega \)\nbd category
    \begin{equation*}
        \colim_{\substack{u \in \molecin{P} \to C \\ v \in\molecin{Q} \to D}} \molecin{(P \gray Q)},
    \end{equation*}
    where \( P \) and \( Q \) are finite regular directed complexes.
    This determines a biclosed monoidal structure on \( \somegaCat \) whose monoidal unit is the terminal stricter \( \omega \)\nbd category \( \globe{0} \) and such that the inclusion \( \omegaReg \incl \somegaCat \) is strong monoidal.
\end{dfn}

\begin{rmk}
    Let \( P, Q \) be regular directed complexes.
    Since \( - \gray - \) is biclosed and using Corollary \ref{cor:regular_directed_complex_colimit_of_itself}, we get
    \begin{align*}
        \molecin{P} \gray \molecin{Q} &\cong \colim_{x \in P, y \in Q} \molecin{\imel{P}{x}} \gray \molecin{\imel{Q}{y}} \\
                                      &\cong \colim_{(x, y) \in P \gray Q} \molecin{(\imel{P}{x} \gray \imel{Q}{y})} \\
                                      &= \molecin{(P \gray Q)}.
    \end{align*}
    Thus Definition \ref{dfn:gray_product_of_stricter_regular_complexes} and Definition \ref{dfn:gray_product_stricter_categories} agree for not necessarily finite regular directed complexes.
\end{rmk}

\subsection{Strict and stricter categories}

The goal of this section is to clarify the relationship between strict and stricter categories.

\begin{dfn} [Theta]
    Let \( U \) be a molecule.
    We say that \( U \) is a \emph{theta} if for all \( x \in U \), \( \clset{x} \) is a globe.
    We write \( \Theta \) for the full subcategory of \( \omegaReg \) on \( \molecin{U} \), for \( U \) a theta.
\end{dfn}

\begin{rmk} \label{rmk:theta_is_theta}
    By \cite[Corollary 9.1.29]{hadzihasanovic2024combinatorics}, a molecule is a theta if and only if it is a pasting of globes.
    Therefore, by \cite[Theorem 8.2.14, Lemma 9.1.16]{hadzihasanovic2024combinatorics}, the category \( \Theta \) is isomorphic to Joyal's \( \Theta \)\nbd category, the full subcategory of the category of strict \( \omega \)\nbd categories on pastings of globes.  
\end{rmk}

\begin{dfn} [Strict \( \omega \)\nbd category]
    We define \( S^\Theta \) to be the subset of \( S \) defined by 
    \begin{equation*}
        S^\Theta \eqdef \set{s_U \colon \colim_{x \in U} \molecin{\imel{U}{x}} \to \molecin{U} \mid U \text{ a theta}}.
    \end{equation*}  
    A \emph{strict \( \omega \)\nbd category} is a composition structure \( C \) which is local with respect to \( S^\Theta \).
    We write \( \omegaCat \) for the category of strict \( \omega \)\nbd categories and strict functors.
\end{dfn}

\begin{rmk}
    By \cite[Theorem 1.12]{berger2002cellular}, this definition is equivalent to the standard definition.
\end{rmk}

\begin{prop} \label{prop:stricter_are_strict}
    Let \( C \) be a stricter \( \omega \)\nbd category.
    Then \( C \) is a strict \( \omega \)\nbd category.
\end{prop}
\begin{proof}
    Since \( C \) is local with respect to \( S \), it is in particular local with respect to \( S^\Theta \subseteq S \).
\end{proof} 

\noindent Akin to stricter \( \omega \)\nbd categories, strict \( \omega \)\nbd categories are a reflective subcategory of \( \Comp \).
By Proposition \ref{prop:stricter_are_strict}, we have a sequence of full subcategory inclusions
\begin{equation*}
     \somegaCat \incl \omegaCat \incl \Comp.
\end{equation*}
We define
\begin{equation*}
    \rcs \colon \omegaCat \to \somegaCat
\end{equation*}
to be the functor applying the reflector \( \rc \colon \Comp \to \somegaCat  \) to the underlying composition structure of a strict \( \omega \)\nbd category.
The functor \( \rcs \) is left adjoint and exhibits \( \somegaCat \) as a reflective subcategory of \( \omegaCat \).

\begin{dfn} [Strict \( n \)\nbd categories]
    Let \( n \in \mathbb{N} \).
    A \emph{strict \( n \)\nbd category} is a strict \( \omega \)\nbd category which is also an \( n \)\nbd composition structure.
    We write \( \nCat{n} \) for the full subcategory of \( \omegaCat \) on strict \( n \)\nbd categories.
\end{dfn}

\begin{rmk}
    If \( C \) is a strict \( n \)\nbd category, then the stricter \( \omega \)\nbd category \( \rcs C \) is a stricter \( n \)\nbd category.
\end{rmk}

\noindent As in the case of stricter \( \omega \)\nbd categories, we have \( n \)\nbd skeleton and \( n \)\nbd truncation functors \( \skel{n}, \trunc{n} \colon \omegaCat \to \nCat{n} \) fitting in a commutative diagram
\begin{center}
    \begin{tikzcd}
        {\nCat n} & \omegaCat \\
        {\snCat n} & \somegaCat.
        \arrow[""{name=0, anchor=center, inner sep=0}, hook, from=1-1, to=1-2]
        \arrow[""{name=1, anchor=center, inner sep=0}, "\rcs"', shift right=2, from=1-1, to=2-1]
        \arrow[""{name=2, anchor=center, inner sep=0}, "{\trunc n}"', shift right=3, from=1-2, to=1-1]
        \arrow[""{name=3, anchor=center, inner sep=0}, "{\skel n}", shift left=3, from=1-2, to=1-1]
        \arrow[""{name=4, anchor=center, inner sep=0}, "\rcs"', shift right=2, from=1-2, to=2-2]
        \arrow[""{name=5, anchor=center, inner sep=0}, shift right=2, hook, from=2-1, to=1-1]
        \arrow[""{name=6, anchor=center, inner sep=0}, hook, from=2-1, to=2-2]
        \arrow[""{name=7, anchor=center, inner sep=0}, shift right=2, hook, from=2-2, to=1-2]
        \arrow[""{name=8, anchor=center, inner sep=0}, "{\trunc n}"', shift right=3, from=2-2, to=2-1]
        \arrow[""{name=9, anchor=center, inner sep=0}, "{\skel n}", shift left=3, from=2-2, to=2-1]
        \arrow["\dashv"{anchor=center}, draw=none, from=1, to=5]
        \arrow["\dashv"{anchor=center, rotate=-90}, draw=none, from=0, to=3]
        \arrow["\dashv"{anchor=center}, draw=none, from=4, to=7]
        \arrow["\dashv"{anchor=center, rotate=-90}, draw=none, from=2, to=0]
        \arrow["\dashv"{anchor=center, rotate=-90}, draw=none, from=6, to=9]
        \arrow["\dashv"{anchor=center, rotate=-90}, draw=none, from=8, to=6]
    \end{tikzcd}
\end{center}

\begin{dfn} [Polygraph]
    A \emph{polygraph} is a strict \( \omega \)\nbd category \( C \), together with, for each \( n \geq 0 \), a pushout diagram
    \begin{center}
        \begin{tikzcd}[column sep=large]
            {\coprod_{e \in \cls{S}_n} \bd{}{}U_e} & {\coprod_{e \in \cls{S}_n} \molecin{U_e}} \\
            \skel{n - 1}C & {\skel{n} C}
            \arrow[""{name=0, anchor=center, inner sep=0}, "{\molecin{\bd{e}{}}}", from=1-1, to=1-2]
            \arrow["{(\bd{}{}e)_{e \in \cls{S}_n}}"', from=1-1, to=2-1]
            \arrow["{(e)_{e \in \cls{S}_n}}", from=1-2, to=2-2]
            \arrow[from=2-1, to=2-2]
            \arrow["\lrcorner"{anchor=center, pos=0.125, rotate=180}, draw=none, from=2-2, to=0]
        \end{tikzcd}
    \end{center}
    in \( \omegaCat \), exhibiting \( \skel{n}C \) as a cellular extension of \( \skel{n - 1}C \), and such that for each \( e \in \cls{S}_n \), \( \dim e = n \).
    The set \( \cls{S} = \coprod_{k \geq 0} \cls{S}_k \) is called the \emph{generating set} of \( C \).
    More generally, we say that a strict functor \( f \colon A \to C \) of composition structure is a \emph{relative polygraph} if \( f \) can be obtained as the transfinite composition of cellular extensions of \( A \).
\end{dfn}

\begin{rmk}
    This definition is equivalent to the usual definition of polygraphs, see again \cite[Comment 8.2.2]{hadzihasanovic2024combinatorics}.
\end{rmk}

\begin{lem} \label{lem:reflection_of_polygraphs_are_stricter_polygraphs}
    Let \( f \colon C \to D \) be a relative polygraph.
    Then \( \rcs f \) is a relative stricter polygraph.
\end{lem}
\begin{proof}
    Clear by Proposition \ref{prop:regular_directed_complex_stricter} and since \( \rcs \) is left adjoint.
\end{proof}

\noindent Next, recall from \cite[Appendice A]{ara2020joint} that strict \( \omega \)\nbd categories also support a Gray product --- our definition of Gray product of stricter \( \omega \)\nbd categories being adapted from it.

\begin{prop} \label{prop:reflection_to_stricter_monoidal}
    The functor \( \rcs \colon \omegaCat \to \somegaCat \) is strong monoidal with respect to the Gray product on stricter and strict \( \omega \)\nbd categories.
\end{prop}
\begin{proof}
    Since \( \rcs \) is left adjoint, and since \( \Theta \) are dense in \( \omegaCat \) by \cite[Proposition 4.6]{ara2020joint}, it is enough to show that for any two thetas \( U \) and \( V \), \( \molecin{(U \gray V)} \) and \( \rcs (\molecin{U} \gray \molecin{V}) \) are naturally isomorphic. 
    By \cite[Lemma 9.1.16]{hadzihasanovic2024combinatorics}, any theta is acyclic.
    This implies by \cite[Proposition 11.2.36]{hadzihasanovic2020diagrammatic} that \( \molecin{(U \gray V)} \cong \molecin{U} \gray \molecin{V} \) in \( \omegaCat \).
    Since \( \molecin{(U \gray V)} \) is stricter by Proposition \ref{prop:regular_directed_complex_stricter}, we conclude.
\end{proof}

\noindent We conclude this section by showing a partial converse to Proposition \ref{prop:stricter_are_strict}, and exhibiting a point of divergence between strict and stricter \( \omega \)\nbd categories.

\begin{thm}\label{thm:strict_le_3_are_stricter}
    Let \( n \le 3 \), and \( C \) be a strict \( n \)\nbd category.
    Then \( C \) is a stricter \( n \)\nbd category.
\end{thm}
\begin{proof}
    Let \( P \) be a regular directed complex with \( \dim P \le 3 \).
    By \cite[Corollary 8.4.12]{hadzihasanovic2024combinatorics}, \( \molecin{P} \) is a polygraph.
    Thus, the map
    \begin{equation*}
        s_P \colon \colim_{x \in P} \imel{P}{x} \to P 
    \end{equation*}
    is an isomorphism when computed in \( \omegaCat \).
    In particular, any strict \( n \)\nbd category \( C \) is local with respect to \( s_P \).
    By Lemma \ref{lem:stricter_n_iff_local_with_dim_le_n}, this concludes the proof.
\end{proof}

\begin{comm} \label{comm:strict_are_not_stricter}
    Thus, the category \( \nCat{n} \) and \( \snCat{n} \) are equivalent for \( n \le 3 \).
    However, they differ starting from \( n = 4 \).
    For instance, consider the \( 4 \)\nbd dimensional regular directed complex \( P \) of \cite[Example 8.2.20]{hadzihasanovic2024combinatorics}.
    Then \( \molecin{P} \) is \emph{not} a polygraph.
    Now, let \( Q \eqdef \colim_{x \in P} \molecin{\imel{P}{x}} \), where the colimit is computed in \( \omegaCat \).
    By construction and Theorem \ref{thm:strict_le_3_are_stricter}, \( Q \) is a polygraph, but it cannot be a stricter \( \omega \)\nbd category.
    Indeed, if it were the case, we would have \( \molecin{P} \cong \rcs Q \cong Q \), contradicting the fact that \( \molecin{P} \) is not a polygraph.
\end{comm}

\subsection{Suspension of stricter \texorpdfstring{$\omega$}{ω}-categories} \label{subsec:suspension}

\begin{dfn} 
    Let \( C \) be a composition structure with two objects \( a \) and \( b \).
    We define the composition structure
    \begin{equation*}
        C(a, b) \eqdef \set{u \in C \mid \bd{0}{-} u = a, \bd{0}{+} u = b},    
    \end{equation*}
    whose boundary operators and \( k \)\nbd composition are induced by the one of \( C \) shifted by \( 1 \).
\end{dfn}

\begin{lem} \label{lem:hom_of_stricter_is_stricter}
    Let \( C \) be a stricter \( \omega \)\nbd category with two objects \( a \) and \( b \).
    Then \( C(a, b) \) is a stricter \( \omega \)\nbd category.
\end{lem}
\begin{proof}
    Let \( U \) be a molecule.
    Any \( U \)\nbd matching family 
    \begin{equation*}
        \set{\F_x \colon \molecin{\imel{U}{x}} \to C(a, b)}_{x \in U}
    \end{equation*}
    determines a \( \sus{U} \)\nbd matching family in \( C \) by letting \( \F_{\sus{x}} \eqdef \F_x \) for \( x \in U \), and \( \F_{\bot^-}, \F_{\bot^+} \) be the strict functors from \( \molecin{\pt} \) classifying \( a \) and \( b \) respectively. 
    This matching family has an amalgamation \( \sus{\F} \colon \molecin{\sus{U}} \to C \), showing that the amalgamation of \( \set{\F_x}_{x \in U} \) is well-defined.
\end{proof}

\begin{comm}
    Thus, any stricter \( \omega \)\nbd category can canonically be seen as a category enriched in stricter \( \omega \)\nbd categories.
    The converse is not true. 
    Indeed, if it were the case, \( \snCat{4} \) would be equivalent to the categories of categories enriched in \( \snCat{3} \).
    Since \( \snCat{3} \) is equivalent to \( \nCat{3} \) by Theorem \ref{thm:strict_le_3_are_stricter}, this would mean that \( \snCat{4} \) is equivalent to \( \nCat{4} \).
    We know from Comment \ref{comm:strict_are_not_stricter} that this is not the case.
    We do not know any general condition guaranteeing that a category enriched in stricter \( \omega \)\nbd category is stricter.
\end{comm}

\begin{dfn} [Suspension]
    Let \( C \) be a composition structure.
    The \emph{suspension of \( C \)} is the composition structure \( \sus{C} \) with two objects \( \bot^+, \bot^- \) and such that
    \begin{equation*}
        \sus{C}(\bot^\a, \bot^\beta) \eqdef 
        \begin{cases}
            C       & \text{if } (\a, \beta) = (-, +),\\
            \set{*} & \text{if } \a = \beta,\\
            \varnothing & \text{else.}
        \end{cases}
    \end{equation*}
    The boundary operators and composition operations are induced by the one of \( C \). 
\end{dfn}

\noindent If \( C \) is a stricter \( \omega \)\nbd category, then \( \sus{C} \) is a strict \( \omega \)\nbd category, since it is the suspension of its underlying strict \( \omega \)\nbd category.
We conclude this section by showing that it is in fact a stricter \( \omega \)\nbd category. 

The following two definitions are adapted from \cite{chandler2019thin}.
\begin{dfn} [Diamond action]
    Let \( P \) be a finite thin graded poset, let \( x, y \in P \) such that \( x \le y \), \( \gamma \colon x \to \ldots \to y \) be a path from \( x \) to \( y \) in the covering diagram of \( P \), and \( D \eqdef \set{a < c_1, c_2 < b} \) be a diamond in \( P \), for two elements \( a, b \in P \) such that \( a \le b \) and \( \dim b - \dim a = 2 \).
    We define \( D\cdot \gamma \) to be the path
    \begin{equation*}
        D \cdot \gamma \eqdef 
        \begin{cases}
            (\gamma \setminus \set{c_1}) \cup \set{c_2} & \text{if } a \to c_1 \to b \text{ is a subpath of } \gamma, \\
            (\gamma \setminus \set{c_2}) \cup \set{c_1} & \text{if } a \to c_2 \to b \text{ is a subpath of } \gamma, \\
            \gamma & \text{else.}
        \end{cases}
    \end{equation*} 
    We let \( \mathcal{D} \) be the set of diamonds in \( P \), and \( P_{x, y} \) be the set of paths from \( x \) to \( y \).
    The \emph{diamond action} is the function \( - \cdot - \colon \mathcal{D} \times P_{x, y} \to P_{x, y} \) sending a pair \( (D, \gamma) \) to \( D \cdot \gamma \).
\end{dfn}

\begin{dfn} [Diamond transitive oriented graded poset]
    Let \( P \) be a graded poset.
    We say that \( P \) is \emph{diamond transitive} if for all \( x, y \in P \) such that \( x \le y \), the diamond action on the graded poset \( P_{x, y} \) is transitive, that is, for all paths \( \gamma, \gamma' \) from \( x \) to \( y \) in \( P \), there exist \( k \in \mathbb{N} \) and a sequence \( (D_i)_{i = 1}^k \) of diamonds in \( P \) such that \( \gamma' = D_1 \cdot \ldots \cdot D_k \cdot \gamma \).
\end{dfn}

\noindent We learnt the about the following Lemma from Hadzihasanovic.
\begin{lem} \label{lem:diamond_transitive}
    Let \( P \) be a regular directed complex.
    Then the graded poset of \( \augm{P} \) is diamond transitive.
\end{lem}
\begin{proof}
    The underlying graded poset of \( \augm{P} \) is a CW poset by \cite[Corollary 10.3.3]{hadzihasanovic2024combinatorics}.
    We conclude by \cite[Theorem 5.14]{chandler2019thin}.
\end{proof}

\begin{dfn} [Collapsible closed subset] \label{dfn:collapsible}
    Let \( U \) be a molecule, \( K \subseteq U \) be a closed subset, and \( \beta \in \set{-, +} \).
    We say that \( K \) is \emph{\( \beta \)\nbd collapsible} if \( K \) is non-empty and for all \( x \in U \), if \( \bd{0}{-\beta} x \in K \) then \( x \in K \).
    In that case, we let \( \coll K U \) be the oriented graded poset whose underlying set is
    \begin{equation*}
        \set{\zcoll} \coprod \set{\icoll u \mid u \in U \setminus K},
    \end{equation*}
    and oriented graded structure is given by
    \begin{equation*}
        \cofaces{}{\a} x \eqdef
        \begin{cases}
            \set{ \icoll v \mid v \in \cofaces{}{\a} u} & \text{if } x = \icoll u,\\
            \set{ \icoll u \mid \dim u = 1, \faces{}{\a} u \in K} &\text{if } x = \bullet, \a = \beta,\\
            \varnothing &\text{else.}
        \end{cases}
    \end{equation*}
    The underlying poset of \( \coll{K}{U} \) fits into the commutative square
    \begin{equation} \label{tik:square_collapse}
        \begin{tikzcd} 
            K & \pt \\
            U & {\coll K U}
            \arrow[""{name=0, anchor=center, inner sep=0}, two heads, from=1-1, to=1-2]
            \arrow[hook, from=1-1, to=2-1]
            \arrow[hook, from=1-2, to=2-2]
            \arrow[two heads, from=2-1, to=2-2]
        \end{tikzcd}
    \end{equation}
    where \( \pt \to \coll{K}{U} \) picks out the element \( \zcoll \), and \( \mapcoll \colon U \surj \coll{K}{U} \) is the order-preserving surjection defined by
    \begin{equation*}
        x \mapsto
        \begin{cases}
            \icoll x & \text{if } x \in U \setminus K\\
            \zcoll & \text{otherwise.}
        \end{cases}
    \end{equation*}
\end{dfn}

\begin{lem} \label{lem:collapsible_is_puhsout}
    Let \( U \) be a molecule, \( \beta \in \set{-, +} \) and \( K \subseteq U \) be a \( \beta \)\nbd collapsible subset.
    Then the commutative square (\ref{tik:square_collapse}) is a pushout in the category of posets.
\end{lem}
\begin{proof}
    Since the underlying square is a pushout of sets, for the statement to be true, we claim that it is enough that for every element \( u \in U \setminus K \) covering an element \( k \) of \( K \), then either \( \dim u = 1 \), or there exists \( v \in U \setminus K \) such that \( v < u \) and \( v \) covers an element of \( K \).
    Indeed, in that case, the elements covering \( \zcoll \) in \( \coll{K}{U} \) with the universal partial order will be exactly of the form \( \icoll{u} \) for \( u \in U \setminus K \) with \( \dim u = 1 \). 
    Therefore, suppose that \( u \in U \setminus K \) of dimension \( n \geq 2 \) covers an element \( k \) of \( K \).
    We will find \( v < u \) such that \( v \notin K \) but covers an element of \( K \).

    Recall that we can find a path \( \gamma \eqdef u_{-1} \eqdef \bot \to u_0 \to u_1 \to \ldots \to u_n \eqdef u \) in \( \augm{U} \) such that \( u_0 = \bd{0}{-\beta} u \).
    Then, none of the \( u_i \) belong to \( K \), since otherwise, \( u_0 \) would be in \( K \), hence so would be \( u \), since \( K \) is \( \beta \)\nbd collapsible.
    Now pick any path \( \gamma' \eqdef k_{-1} = \bot \to k_1 \to \ldots \to k_{n - 1} \eqdef k \) from \( \bot \) to \( k \) in \( \augm{U} \) which necessarily belong to \( K \), since it is closed.
    By Lemma \ref{lem:diamond_transitive}, there exist \( r \geq 1 \) and a sequence of diamonds 
    \begin{equation*}
        \left(D_i \eqdef \set{z_i < y_i, y'_i < x_i}\right)_{i = 1}^r
    \end{equation*}
    such that \( \gamma = D_r \cdot \ldots \cdot D_1 \cdot \gamma' \).
    Up to swapping \( y_i \) and \( y'_i \), we may assume that \( y'_i = y_{i + 1} \) for all \( i \in \set{1, \ldots, r - 1} \).
    Let \( \ell \eqdef \max \set{i \mid y_i \in K} \), which exists since \( y_1 \in K \).
    If \( x_{\ell} < u \), then we are done since \( x_{\ell} \) covers \( y_{\ell} \in K \).
    Else \( x_{\ell} = u \).
    Then, either \( \ell = r \), in which case \( y'_{\ell} = u_{n - 1} \notin K \), or \( \ell < r \), in which case \( y'_{\ell} = y_{\ell + 1} \notin K \).
    In any case, \( y'_{\ell} \) is not in \( K \).
    Since \( y_{\ell} \in K \) is of dimension \( \geq 1 \) and \( K \) is closed, \( z_{\ell} \in K \).
    Thus we found \( y'_{\ell} < u \) such that \( y'_\ell \) is not in \( K \) and covers \( z_{\ell} \in K \).
    This concludes the proof.
\end{proof}

\begin{lem} \label{lem:path_from_zero_bd_to_all_points}
    Let \( U \) be a molecule, \( \a \in \set{-, +} \) and \( x \in \gr{0}{U} \).
    Then there exists \( k \in \mathbb{N} \) and an inclusion \( \iota \colon k\dglobe{1} \incl U \) such that \( \bd{0}{\a} \iota = \bd{0}{\a} U \) and \( \bd{0}{-\a} \iota = x \).
\end{lem}
\begin{proof}
    We proceed by induction on the submolecules of \( V \) of \( U \).
    The base case where \( \dim V = 0 \) is clear.
    Suppose inductively that the statement is true of all proper submolecules \( V \) of \( U \).
    First suppose that \( U \) is an atom.
    If \( \dim U = 0 \), then the statement is again clear, otherwise \( \dim U > 0 \) and \( x \in \bd{}{\beta} U \) for some \( \beta \in \set{-, +} \).
    We conclude by inductive hypothesis on \( \bd{}{\beta} U \submol U \) and globularity.
    Now suppose \( U \) split into \( V \cp{k} W \), and assume that \( x \in V \), the case \( x \in W \) is symmetrical.
    Then either \( k > 0 \), hence \( \bd{0}{} V = \bd{0}{} U \) and we may conclude by inductive hypothesis on \( V \), or \( k = 0 \).
    If \( \a = - \), then \( \bd{0}{-} V = \bd{0}{-} U \), and we conclude by inductive hypothesis on \( V \) again.
    Else, \( \a = + \) and \( \bd{0}{+} U = \bd{0}{+} W \).
    Consider by inductive hypothesis an inclusion \( \iota \colon k\dglobe{1} \incl V \) from \( x \) to \( \bd{0}{+} V \).
    Then \( \bd{0}{-} W \) is isomorphic to \( k'\dglobe{1} \) for some \( k' > 0 \), hence \( \iota \cp{0} \bd{1}{-} \idd{W} \colon (k + k') \dglobe{1} \incl U \) is the desired inclusion.
\end{proof}

\begin{lem} \label{lem:collapsible_connected_with_all_closure_element}
    Let \( U \) be a molecule, \( \beta \in \set{-, +} \), \( K \subseteq U \) be a \( \beta \)\nbd collapsible subset, and \( V \submol U \).
    Then \( V \cap K \) is either empty or connected.
    In the latter case, \( V \cap K \subseteq V \) is \( \beta \)\nbd collapsible and contains \( \bd{0}{\beta} V \).
\end{lem}
\begin{proof}
    Assume that \( \beta = - \), the case \( \beta = + \) is dual.
    Suppose that \( V \cap K \) is non-empty.
    We show that \( x \eqdef \bd{0}{-} V \) is in \( V \cap K \).
    Let \( z \in V \cap K \), then \( y \eqdef \bd{0}{-} z \in V \cap K \).
    By Lemma \ref{lem:path_from_zero_bd_to_all_points}, there exist \( k \geq 0 \) and \( \iota \colon k\arr \incl V \) such that \( \bd{0}{} \iota = (x, y) \).
    Since \( y = \bd{0}{+} \iota \in K \) and \( K \) is collapsible, \( \cofaces{}{+} y \in K \).
    Since \( K \) is closed, \( \faces{}{-} \cofaces{}{+} y \in K \).
    Iterating this process \( k \) times, we find that \( x \in K \).
    Then, suppose that \( V \cap K = F \cup G \) for two closed subsets \( F, G \) with \( F \cap G = \varnothing \), and suppose without loss of generality that \( x \in F \).
    Let \( y \in \gr{0}{G} \). 
    As previously, we have an inclusion \( \iota \colon k\arr \incl F \cup G \) such that \( \bd{0}{}\iota = (x, y) \).
    Then \( \iota(k\arr) = \invrs{\iota}(F) \cup \invrs{\iota}(G) \) and \( \invrs{\iota}F \cap \invrs{\iota}G = \varnothing \).
    By \cite[Lemma 3.3.13]{hadzihasanovic2024combinatorics}, \( \invrs{\iota}F = \varnothing \) or \( \invrs{\iota}G = \varnothing \), a contradiction.
    This proves that \(  \gr{0}{G} = \varnothing \), thus that \( G = \varnothing \).
    This shows that \( V \cap K \) is connected.
    Finally, let \( x \in V \) such that \( \bd{0}{+} x \in K \).
    Since \( K \) is  \( \beta \)\nbd collapsible, \( x \in V \cap K \).
    Hence \( V \cap K \subseteq V \) is \( \beta \)\nbd collapsible.
    This concludes the proof.
\end{proof}

\begin{lem} \label{lem:collapsible_negbeta_boundary_collapse_all}
    Let \( U \) be a molecule, \( \beta \in \set{-, +} \), and \( K \subseteq U \) be a \( \beta \)\nbd collapsible subset such that \( \bd{0}{-\beta} U \subseteq K \).
    Then \( K = U \).
\end{lem}
\begin{proof}
    We proceed by induction on the layering dimension \( \ell \) of \( U \) (see \cite[Comment 4.2.13]{hadzihasanovic2024combinatorics}).
    If \( \ell = -1 \), then \( U \) is an atom, thus \( U = \clset{\top_U} \subseteq K \) by definition.
    Inductively, let \( \ell \geq 0 \). Then \( U \) admits an \( \ell \)\nbd layering
    \begin{equation*}
        U \eqdef \order{1}{U} \cp{\ell} \ldots \cp{\ell} \order{r}{U}
    \end{equation*}
    with \( r \geq 2 \) such that \( \order{i}{U} \) has layering dimension \( < \ell \) for all \( 1 \le i \le r \).
    Suppose that \( \beta = - \), the case \( \beta = + \) is dual.
    Then \( \bd{0}{+} U = \bd{0}{+} \order{r}{U} \in K \).
    By Lemma \ref{lem:collapsible_connected_with_all_closure_element}, \( \order{r}{U} \cap K \) is \( \beta \)\nbd collapsible, thus by inductive hypothesis, \( \order{r}{U} = \order{r}{U} \cap K \).
    In particular, \( \bd{0}{+} \order{r - 1}{U} \in K \).
    Iterating this process \( r \) times, we find that for all \( i \in \set{0, \ldots, r - 1} \), \( \order{r - i}{U} = \order{r - i}{U} \cap K \).
    Therefore, \( U = U \cap K \).
    This concludes the proof.
\end{proof}

\begin{lem} \label{lem:collapsible_mapcoll_preserve_boundaries}
    Let \( U \) be a molecule, \( \beta \in \set{-, +} \), \( K \subseteq U \) be a \( \beta \)\nbd collapsible subset, \( \a \in \set{-, +} \) and \( k \in \mathbb{N} \).
    Then 
    \begin{equation*}
        \mapcoll (\bd{k}{\a} U) = \bd{k}{\a} (\coll{K}{U}),
    \end{equation*}
    in particular \( \bd{0}{\beta} (\coll{K}{U}) = \set{\zcoll} \).
    Furthermore, \( \coll{K}{U} \) is globular, and if \( U \) is round, so is \( \coll{K}{U} \).
\end{lem}
\begin{proof}
    We assume that \( \beta = - \), the case \( \beta = + \) is dual, and that \( K \) is a proper subset of \( U \), otherwise the statement is trivial.
    Suppose that \( k > 0 \), then
    \begin{equation*}
        \bd{k}{\a} (\coll{K}{U}) = \clos \faces{k}{\a} (\coll{K}{U}) \cup \gr{< k}{\clset{\maxel (\coll{K}{U})}} = \mapcoll(\bd{k}{\a} U) \cup \set{\zcoll}.
    \end{equation*}
    By Lemma \ref{lem:collapsible_connected_with_all_closure_element}, \( \bd{0}{-} U \subseteq K \), and since \( k > 0 \), we have \( \set{\zcoll} = p(\bd{0}{-} U) \subseteq p(\bd{k}{\a} U) \).
    Thus, \(  \mapcoll (\bd{k}{\a} U) = \bd{k}{\a} (\coll{K}{U}) \).

    Next suppose \( k = 0 \) and \( \a = + \).
    Then since \( K \) is a proper subset, \( \set{x} = \bd{0}{+} U \) is not a subset of \( K \) by Lemma \ref{lem:collapsible_negbeta_boundary_collapse_all}. 
    Thus, by construction of \( \coll{K}{U} \), we have \( \mapcoll (\bd{0}{+} U) = \bd{0}{+} (\coll{K}{U}) \).
    The next thing to see is therefore that \( \set{\zcoll} = \bd{0}{-} (\coll{K}{U}) \).
    By construction, \( \cofaces{}{+} \zcoll = \varnothing \), thus we have one inclusion.
    Let \( \icoll{u} \in \gr{0}{(U \setminus K)} \).
    If \( u \in \bd{0}{-} U \) then \( u \in K \) by Lemma \ref{lem:collapsible_connected_with_all_closure_element}, a contradiction.
    Thus, there exists \( v \in \cofaces{}{+} u \), and since \( K \) is closed, \( v \notin K \).
    We find that \( \icoll{v} \in \cofaces{}{+} \icoll{u} \), meaning that \( u \notin \bd{0}{-} (\coll{K}{U}) \).
    This proves that \( \set{\zcoll} = \bd{0}{-} (\coll{K}{U}) \).

    By Lemma \ref{lem:collapsible_is_puhsout}, Lemma \ref{lem:collapsible_connected_with_all_closure_element}, and formal properties, for all closed subsets \( V \subseteq U \) such that \( V \) is a molecule, we have 
    \begin{equation*}
        p(V) \cong
        \begin{cases}
            V & \text{if } V \cap K = \varnothing, \\
            \coll{(V \cap K)}{V} & \text{else}.
        \end{cases}
    \end{equation*}
    Using this and the first part of the proof, one sees that \( \coll{K}{U} \) is globular.

    Last, suppose that \( U \) is round of dimension \( n \geq 0 \).
    Let \( k < n \), and let \( x \in (\bd{k}{-} (\coll{K}{U})) \cap (\bd{k}{+} (\coll{K}{U})) \).
    We must show that \( x \in \bd{k - 1}{}  (\coll{K}{U}) \).
    Suppose first that \( x = \zcoll \), then \( \set{x} = \bd{0}{-} (\coll{K}{U}) \).
    If \( k > 0 \) we are done by globularity.
    Else \( k = 0 \).
    In that case, \( \zcoll \in \bd{0}{+} (\coll{K}{U}) \), so \( \bd{0}{+} U \subseteq K \), hence \( \coll{K}{U} = \pt \) by Lemma \ref{lem:collapsible_negbeta_boundary_collapse_all}.
    Thus, \( n = 0 = k \), contradiction.
    Suppose now that \( x = \icoll{u} \) for some \( u \in U \setminus K \).
    By the first part of the proof, \( u \in \bd{k}{-} U \cap \bd{k}{+} U \), and since \( U \) is round, \( u \in \bd{k - 1}{} U \), so that by the first part of the proof again, \( \icoll{u} \in \bd{k - 1}{} (\coll{K}{U}) \).
    This concludes the proof.
\end{proof}

\begin{prop} \label{prop:collapsible_collapse_to_molecules}
    Let \( U \) be a molecule, \( \beta \in \set{-, +} \) and \( K \subseteq U \) be a \( \beta \)\nbd collapsible subset.
    Then \( \coll{K}{U} \) is a molecule and the canonical map \( \mapcoll \colon U \surj \coll{K}{U} \) is a map of molecules.
\end{prop}
\begin{proof}
    Recall that a map of poset \( L \to \pt \) is final when seen as a functor just when \( L \) is connected. 
    Let \( \mapcoll \colon U \surj \coll{K}{U} \) be the canonical map.
    Let \( V \submol U \) be a submolecule of \( U \). 
    If \( V \cap K = \varnothing \), then \( \restr{\mapcoll}{V} \colon V \surj \mapcoll(V) \) is the identity.
    Otherwise, by Lemma \ref{lem:collapsible_is_puhsout}, the diagram
    \begin{center}
        \begin{tikzcd}
            {V \cap K} & \pt \\
            V & {p(V)}
            \arrow[""{name=0, anchor=center, inner sep=0}, two heads, from=1-1, to=1-2]
            \arrow[from=1-1, to=2-1]
            \arrow[from=1-2, to=2-2]
            \arrow[two heads, from=2-1, to=2-2]
            \arrow["\lrcorner"{anchor=center, pos=0.125, rotate=180}, draw=none, from=2-2, to=0]
        \end{tikzcd}
    \end{center}
    is a pushout.
    By Lemma \ref{lem:collapsible_connected_with_all_closure_element}, \( V \cap K \) is connected, thus \( V \cap K \surj \pt \) is final.
    Since final functors are stable under pushouts, we obtain that \( \restr{\mapcoll}{V} \colon V \surj \mapcoll(V) \) is final.
    This proves that for all submolecules \( V \submol U \), \( \restr{\mapcoll}{V} \colon V \surj \mapcoll(V) \) is final.
    Now let \( x \in U \), \( k \geq 0 \) and \( \a \in \set{-, +} \).
    Using either the fact that \( \restr{p}{\bd{k}{\a} x} \) is the identity if \( \clset{x} \cap K = \varnothing \) and Lemma \ref{lem:collapsible_mapcoll_preserve_boundaries} otherwise, we find that \( \mapcoll(\bd{k}{\a} x) = \bd{k}{\a} \mapcoll(x) \).
    This proves that, in case \( \coll{K}{U} \) is a regular directed complex, \( \mapcoll \) is a map of regular directed complexes.

    We prove by induction on the submolecules \( V \) of \( U \) that \( \mapcoll(V) \) is a molecule.
    The base case where \( V \) is a point is clear.
    Suppose inductively that \( \mapcoll(V) \) is a molecule for all proper submolecules \( V \) of \( U \).
    Suppose first that \( U \) is an atom.
    By inductive hypothesis and Lemma \ref{lem:collapsible_mapcoll_preserve_boundaries}, for all \( \a \in \set{-, +} \), \( p(\bd{}{\a} U) \) is a round molecule.
    Thus, \( p(\bd{}{-} U) \celto p(\bd{}{+} U) \) is a well-defined atom, and one sees that, unless \( K = U \), it is by construction isomorphic to \( \coll{K}{U} \).
    If \( K = U \), then \( U = \pt \) is also an atom. 
    This shows that \( \coll{K}{U} \) is an atom.
    Now suppose that \( U \) is a molecule.
    Then by inductive hypothesis, \( \coll{K}{U} \) is a regular directed complex.
    We conclude by \cite[Proposition 6.2.33]{hadzihasanovic2024combinatorics} and the first part of the proof that it is in fact a molecule.
\end{proof}

\begin{lem} \label{lem:has_two_point_is_susp}
    Let \( U \) be a molecule such that \( \gr{0}{U} = \bd{0}{-} U \coprod \bd{0}{+} U \).
    Then \( U = \sus{U'} \) for some molecule \( U' \).
\end{lem}
\begin{proof}
    We proceed by induction on the submolecules \( V \) of \( U \).
    The base case where \( V \) is a point is vacuously true, since \( \gr{0}{V} \) has exactly one element.
    Suppose inductively that the statement is true of all proper submolecules of \( U \).
    Consider first the case where \( U \) is an atom.
    We distinguish two cases.
    Either \( U \) is \( \dglobe{1} \), in which case it is \( \sus{\pt} \) and we are done, or \( U \) is of the form \( V \celto W \) with \( \dim V = \dim W \geq 1 \).
    In the latter case,
    \begin{align} 
         \bd{0}{} V \subseteq \gr{0}{V} \subseteq \gr{0}{U} &= \bd{0}{} U = \bd{0}{} V, \text{ and} \nonumber\\
         \bd{0}{-} V \cap \bd{0}{+} V &= \varnothing \label{aln:grade_boundary_zero}.
    \end{align}
    Thus, \( \gr{0}{V} = \bd{0}{-} V \coprod \bd{0}{+} V\).
    Similarly, \( \gr{0}{W} = \bd{0}{-} W \coprod \bd{0}{+} W \).
    By inductive hypothesis on the submolecules, \( V = \sus{V'} \) and \( W = \sus{W'} \) for some molecules \( V' \) and \( W' \).
    By \cite[Proposition 7.3.16]{hadzihasanovic2024combinatorics}, \( U \) is of the form \( \sus{(V' \celto W')} \).
    Suppose next that \( U \) splits into \( V \cp{k} W \) with \( \dim V, \dim W > k \).
    If \( k = 0 \), then \( \bd{0}{+} V \subseteq \gr{0}{U} \), but \( \bd{0}{+} V \cap (\bd{0}{-} U \coprod \bd{0}{+} U) \) is empty, contradicting the assumption.
    Thus, \( k > 0 \) and (\ref{aln:grade_boundary_zero}) also holds of \( V \) and \( W \).
    We deduce by inductive hypothesis that \( V = \sus{V'} \) and \( W = \sus{W'} \).
    By \cite[Proposition 7.3.16]{hadzihasanovic2024combinatorics} again, \( U \) is of the form \( \sus{(V' \cp{k - 1} W')} \).
    This concludes the proof.
\end{proof}

\begin{prop} \label{prop:desuspension}
    Let \( U \) be a molecule, and \( K^-, K^+ \subseteq U \) be disjoint closed subsets of \( U \) such that \( \gr{0}{U} \subseteq K^- \cup K^+ \), \( K^- \cup K^+ \) is a proper subset of \( U \), and for all \( \a \in \set{-, +} \), \( K^\beta \subseteq U \) is \( \beta \)\nbd collapsible.
    Then there exists a molecule \( V \), together with a commutative diagram
    \begin{center}
        \begin{tikzcd}
            {K^- \coprod K^+} & {\pt \coprod \pt} \\
            U & {\sus{V}}
            \arrow[two heads, from=1-1, to=1-2]
            \arrow[hook', from=1-1, to=2-1]
            \arrow["{{(\bot^-, \bot^+)}}", hook', from=1-2, to=2-2]
            \arrow["q"', two heads, from=2-1, to=2-2]
        \end{tikzcd}
    \end{center}
    whose underlying diagram in \( \Pos \) is a pushout, and such that \( q \) is a map of molecules.
\end{prop}
\begin{proof}
    By Proposition \ref{prop:collapsible_collapse_to_molecules}, we have a map of molecules \( p \colon U \surj \coll{K^-}{U} \), and since \( K^+ \subseteq U \) is \( + \)\nbd collapsible, so is \( p(K^+) \subseteq \coll{K^-}{U} \).
    Applying Proposition \ref{prop:collapsible_collapse_to_molecules} again, we have a map of molecules \( q \colon U \surj W \), which is the pushout of the map \( K^- \coprod K^- \surj \pt \coprod \pt \) of regular directed complexes along the inclusion \( K^+ \coprod K^- \incl U \).
    By Lemma \ref{lem:collapsible_mapcoll_preserve_boundaries}, the inclusion \( \pt \coprod \pt \incl W \) has image \( \bd{0}{} W \).
    To conclude, we need to show that \( W = \sus{V} \) for some molecule \( V \).
    Since \( q \) is a surjective map of molecules, 
    \begin{equation*}
        \gr{0}{W} = q(\gr{0}{U}) \subseteq q(K^-) \cup q(K^+) = \set{\bot^-, \bot^+}
    \end{equation*}
    We conclude by Lemma \ref{lem:has_two_point_is_susp}.
\end{proof}

\begin{thm} \label{thm:suspension_of_stricter}
    Let \( C \) be a stricter \( \omega \)\nbd category.
    Then \( \sus{C} \) is a stricter \( \omega \)\nbd category.
\end{thm}
\begin{proof}
    We use the second characterisation in Lemma \ref{lem:stricter_iff_local_wrt_pasting}.
    Let \( U \) be a molecule and \( \set{\F_x \colon \molecin{\imel{U}{x}} \to \sus{C}}_{x \in U} \) be a \( U \)\nbd matching family in \( \sus{C} \).
    For \( \beta \in \set{-, +} \), let \( K^\beta \eqdef \set{x \in U \mid \pcell{\F_x} = \bot^\beta} \).
    Suppose that there exists \( \beta \in \set{-, +} \) such that \( K^\beta \) is empty and let \( x \in U \).
    If \( \pcell{\F_x} \in C \), we would have \( x \in K^\beta \).
    Thus, \( \pcell{\F_x} = \bot^{-\beta} \), and \( K^{-\beta} = U \).
    Therefore, the candidate amalgamation \( \F \colon \molecin{U} \to \sus{C} \) is well-defined since it factors as
    \begin{equation*}
        \molecin{U} \to \molecin{\pt} \to \sus{C},
    \end{equation*}
    where \( \molecin{\pt} \to \sus{C} \) classifies \( \bot^{-\beta} \).

    Now suppose that \( K^- \) and \( K^+ \) are non-empty.
    Let \( \beta \in \set{-, +} \), then we claim that \( K^\beta \subseteq U \) is \( \beta \)\nbd collapsible.
    Let \( x \in U \), and suppose that \( \bd{0}{-\beta} x \in K^\beta \).
    Then \( \bd{0}{-\beta} \pcell{F_x} = \bot^{\beta} \).
    If \( \pcell{\F_x} \neq \bot^\beta \), then by construction of \( \sus{C} \), \( \bd{0}{-\beta} \pcell{\F_x} = \bot^{-\beta} \), a contradiction.
    Thus, \( \pcell{\F_x} = \bot^\beta \), hence \( x \in K^\beta \).
    Since \( K^\beta \) is non-empty, it is \( \beta \)\nbd collapsible.
    Then  \( \gr{0}{U} \subseteq K^- \cup K^+ \), and \( K^- \cap K^+ = \varnothing \).
    By \cite[Lemma 3.3.13]{hadzihasanovic2024combinatorics}, \( U \) is connected, thus the inclusion \( K^- \cup K^+ \subseteq U \) is proper.
    Now Proposition \ref{prop:desuspension} applies, and we have in particular a molecule \( V \) and a final map of molecules \( q \colon U \surj \sus{V} \). 
    Then, one checks that the family of globular cells of \( C \)
    \begin{equation*}
        \set{\pcell{\F_x} \mid q(x) = \sus{y}}_{y \in V}
    \end{equation*} 
    defines, as per Remark \ref{rmk:data_matching family}, a \( V \)\nbd matching family \( \set{\G_y \colon \molecin{\imel{V}{y}} \to C}_{y \in V} \) in \( C \).
    Since \( C \) is stricter, it admits a well-defined amalgamation \( \G \colon \molecin{V} \to C \).
    Finally, the candidate amalgamation \( \F \colon \molecin{U} \to \sus{C} \) is well-defined, since it factors as \( \F = \sus{\G} \after \molecin{q} \).
    Thus, \( \sus{C} \) is local with respect to \( s_U \).
    This concludes the proof.
\end{proof}

\begin{dfn} [Bipointed stricter \( \omega \)\nbd category]
    A \emph{bipointed stricter \( \omega \)\nbd category} is given by a stricter \( \omega \)\nbd category \( C \) together with a pair of objects \( (a, b) \).
    We let \( \bpt\somegaCat \) be the category of bipointed stricter \( \omega \)\nbd categories and strict functors that respect the bipointing.
\end{dfn}

\begin{cor} \label{cor:adjunction_hom_suspension}
    There is an adjunction
    \begin{equation*}
        \sus{} \colon \somegaCat \leftrightarrows \bpt\somegaCat \cocolon \hom,
    \end{equation*}
    where \( \sus{C} \) is the suspension of \( C \) bipointed by \( (\bot^-, \bot^+) \), and \( \hom(C, a, b) \) is given by \( C(a, b) \).
    Furthermore, the functor
    \begin{equation*}
        \sus{} \colon \somegaCat \to \somegaCat
    \end{equation*}
    preserves connected colimits.
\end{cor}
\begin{proof}
    That the pair of functors is well-defined follows from Lemma \ref{lem:hom_of_stricter_is_stricter} and Theorem \ref{thm:suspension_of_stricter}.
    That this form an adjunction follows from standard arguments.
    Last, since \( \bpt\somegaCat \) is a coslice construction, \( \sus{} \colon \somegaCat \to \somegaCat \) preserves connected colimits by \cite[Proposition 3.3.8]{riehl2019context}.
\end{proof}

\section{Diagrammatic sets} \label{sec:diagrammatic}

\subsection{Diagrams in a diagrammatic set}

Recall that a diagrammatic set is a presheaf over the category \( \atom \).
We write \( \dgmSet \) for the category of diagrammatic sets.
By \cite[Lemma 2.5]{chanavat2024htpy}, the Yoneda embedding \( \atom \incl \dgmSet \) factors as
\begin{equation*}
    \atom \incl \rdcpxcart \incl \dgmSet,
\end{equation*} 
and in the sequel, we always identify a regular directed complex with its associated diagrammatic set.

\begin{dfn} [Diagram in a diagrammatic set]
    Let \( U \) be a regular directed complex and \( X \) a diagrammatic set.
    A \emph{diagram of shape \( U \) in \( X \)} is a morphism \( u \colon U \to X \).
    A diagram is called
    \begin{itemize}
        \item a \emph{pasting diagram} if \( U \) is a molecule,
        \item a \emph{round diagram} if \( U \) is a \emph{round} molecule, and
        \item a \emph{cell} if \( U \) is an atom.
    \end{itemize}
    We let \( \dim u \eqdef \dim U \), and call respectively \( \Pd X \), \( \Rd X \) and \( \cell X \) the sets of pasting diagrams, round diagrams, and cells in \( X \).
\end{dfn}

\begin{rmk}
    By the Yoneda Lemma, we identify a cell \( u \colon U \to X \) with its corresponding element \( u \in X(U) \).
    Furthermore, since isomorphisms of molecules are unique when they exists, we may safely identify isomorphic pasting diagrams in the slice over \( X \).
\end{rmk}

\begin{dfn} [Subdiagram]
    Let \( u \colon U \to X \) be a pasting diagram in a diagrammatic set \( X \).
    A \emph{subdiagram of \( u \)} is a pair of a pasting diagram \( v \colon V \to X \) and a submolecule inclusion \( \iota \colon V \incl U \) such that \( u \after \iota = v \).
    A subdiagram is \emph{rewritable} if the submolecule inclusion \( \iota \) is rewritable.
    We write \( \iota \colon v \submol u \) for the data of a subdiagram of \( u \), or simply \( v \submol u \) if \( \iota \) is irrelevant or evident from the context. 
\end{dfn}

\begin{dfn} [Composition structure of pasting diagrams]
    Let \( u \colon U \to X \) be a pasting diagram in a diagrammatic set \( X \).
    For \( n \in \mathbb{N} \) and \( \a \in \set{-, +} \), we write \( \bd{n}{\a} u \) for the pasting diagram \( \restr{u}{\bd{n}{\a}U} \colon \bd{n}{\a} U \to X \).
    We may omit the index \( n \) if \( n = \dim u - 1 \).
    This makes \( \Pd(X) \) a reflexive \( \omega \)\nbd graph.
    Now let \( u \colon U \to X \) and \( v \colon V \to X \) be pasting diagrams, such that \( \bd{k}{+} u = \bd{k}{-} v \).
    We let \( u \cp{k} v \colon U \cp{k} V \to X \) be the unique pasting diagram determined by the universal property of the pushout \( U \cp{k} V \).
    This makes \( \Pd(X) \) a composition structure.

    More generally, if \( \iota \colon \bd{k}{+} u \submol \bd{k}{-} v \), we write \( u \cpsub{\iota} v \) for the universal pasting diagram determined by the the pasting at the submolecule \( \iota \).
    We use the dual notation if \( \iota \colon \bd{k}{-} u \submol \bd{k}{+} v \). 

    We often omit the index \( k \) when it is equal to \( \min \set{\dim u, \dim v} - 1 \), and omit \( \iota \) when it is irrelevant or evident from the context.
\end{dfn}

\begin{rmk}
    \( \Rd(X) \) is a sub-reflexive \( \omega \)\nbd graph of \( \Pd(X) \), but is not itself a composition structure, since pastings of round diagrams are generally not round. 
\end{rmk}

\begin{lem} \label{lem:Pd_is_stricter}
    Let \( X \) be a diagrammatic set.
    Then the composition structure \( \Pd(X) \) is a stricter \( \omega \)\nbd category.
\end{lem}
\begin{proof}
    Let \( P \) be a regular directed complex and consider a matching family \( \set{\F_x \colon \molecin{\imel{P}{x}} \to \Pd(X)}_{x \in P} \).
    Then the candidate amalgamation sends \( w \colon W \to P \) to the canonical morphism of diagrammatic sets \( \colim_{x \in W} \imel{P}{w(x)} \to X \), which does not depend on the decomposition of \( w \).
\end{proof}

\noindent Since morphisms \( f \colon X \to Y \) of diagrammatic sets also induce strict functors \( \Pd(f) \colon \Pd(X) \to \Pd(Y) \), this defines a functor
\begin{equation*}
    \Pd \colon \dgmSet \to \somegaCat.
\end{equation*}

\begin{dfn} [Bipointed diagrammatic set]
    A \emph{bipointed diagrammatic set} is given by a diagrammatic set \( X \) together with a pair of \( 0 \)\nbd cells \( a, b \colon \pt \to X \).
    We let \( \bpt\dgmSet \) be the category of bipointed diagrammatic sets and morphisms that respect the bipointing.
\end{dfn}

\begin{dfn} [Suspension]
    The \emph{suspension of diagrammatic sets} is the functor
    \begin{equation*}
        \sus{-} \colon \dgmSet \to \bpt{\dgmSet}
    \end{equation*} 
    defined as the left Kan extension along the Yoneda embedding of the functor sending an atom \( U \) to \( (\sus{U}, \bot^-, \bot^+) \).
\end{dfn}

\begin{prop} \label{prop:suspension_of_dgmSet}
    The functor \( \sus{} \colon \dgmSet \to \dgmSet \) preserves connected colimits.    
\end{prop}
\begin{proof}
    The category \( \bpt{\dgmSet} \) is the coslice \( \slice{(\pt + \pt)}{\dgmSet} \).
    Since \( \sus{-} \) is left adjoint by construction, the result follows from \cite[Proposition 3.3.8]{riehl2019context}
\end{proof}

\begin{dfn} [Connected diagrammatic set]
    Let \( X \) be a diagrammatic set. 
    We say that \( X \) is \emph{connected} if its category of elements is connected.
\end{dfn}

\begin{rmk} \label{rmk:ez_connected}
    Since \( \atom \) is an Eilenberg-Zilber category \cite[Proposition 1.17]{chanavat2024htpy}, \( X \) is connected if and only if its poset of non-degenerate cells is connected.
    In particular by \cite[Lemma 3.3.13]{hadzihasanovic2024combinatorics}, all molecules are connected.
\end{rmk}

\subsection{Degenerate diagrams and equivalences}

\begin{dfn} [Degenerate pasting diagram]
    Let \( u \colon U \to X \) be a diagram.
    We say that \( u \) is \emph{degenerate} if there exists a pair of a diagram \( v \colon V \to X \) and a surjective cartesian map of regular directed complexes \( p \colon U \to V \) such that \( v \after p = u \), and \( \dim v < \dim v \). 
    We let
    \begin{align*}
        \Dgn X &\eqdef \set{u \in \Pd X \mid u \text{ is degenerate}} & \dgn X \eqdef \Dgn X \cap \cell x,\\
        \Nd X & \eqdef \set{u \in \Pd X \mid u \text{ is not degenerate}} & \nd X \eqdef \Nd X \cap \cell x.
    \end{align*}
\end{dfn}

\begin{dfn} [Reverse of a degenerate diagram]
    Let \( u \colon U \to X \) be a degenerate diagram, equal to \( v \after p \) for some diagram \( v \colon V \to X \) and surjective cartesian map \( p \colon U \to V \) with \( n \eqdef \dim u > \dim v \).
    The \emph{reverse of \( u \)} is the diagram \( \rev{u} \eqdef v \after \dual{n}{p} \colon \dual{n}{U} \to X \).
\end{dfn}

\begin{dfn}[Partial Gray cylinder]
	Let \( U \) be a regular directed complex and \( K \subseteq U \) a closed subset.
	The \emph{partial Gray cylinder on \( U \) relative to \( K \)} is the oriented graded poset \( \arr \pcyl K U \) whose
    \begin{itemize}
        \item underlying graded poset is obtained as the pushout
            \begin{center}
                \begin{tikzcd}
                    {I \times K} & K \\
                    {I \times P} & {(I \times P) \coprod_{I \times K} K}
                    \arrow[two heads, from=1-1, to=1-2]
                    \arrow[hook', from=1-1, to=2-1]
                    \arrow["{(-)}", hook', from=1-2, to=2-2]
                    \arrow["q", two heads, from=2-1, to=2-2]
                    \arrow["\lrcorner"{anchor=center, pos=0.125, rotate=180}, draw=none, from=2-2, to=1-1]
                \end{tikzcd}  
            \end{center}
            in the category of posets;
        \item orientation is specified, for all \( \a \in \set{+, -} \), by
        \begin{align*}
            \faces{}{\a}(x) & \eqdef \set{(y) \mid y \in \faces{}{\a}x}, \\
            \faces{}{\a}(i, x) & \eqdef \begin{cases}
                \set{(0^\a, x)} + \set{(1, y) \mid y \in \faces{}{-\a}x \setminus K} &
                \text{if \( i = 1 \),} \\
                \set{(i, y) \mid y \in \faces{}{\a}x \setminus K} + 
                \set{(y) \mid y \in \faces{}{\a}x \cap K} &
                \text{otherwise}.
            \end{cases}
        \end{align*}
    \end{itemize}
    This is equipped with a canonical projection map \( \tau_K \colon \arr \pcyl K P \surj P \).
\end{dfn}

\begin{dfn}[Inverted partial Gray cylinder]
	Let \( U \) be a molecule, \( n \eqdef \dim U \), and \( K \subseteq \bd{}{+}U \) a closed subset.
	The \emph{left-inverted partial Gray cylinder on \( U \) relative to \( K \)} is the oriented graded poset \( \lcyl{K} U \) whose
    \begin{itemize}
        \item underlying graded poset is the same as \( \arr \pcyl{K} U \);
        \item orientation is as in \( \arr \pcyl K U \), except for all \( x \in \gr{n}{U} \) and \( \a \in \set{+, -} \)
    \begin{align*}
        \faces{}{-}(1, x) &\eqdef \set{(0^-, x), (0^+, x)} + \set{(1, y) \mid y \in \faces{}{+}x \setminus K}, \\
        \faces{}{+}(1, x) &\eqdef \set{(1, y) \mid y \in \faces{}{-}x}, \\
        \faces{}{\a}(0^+, x) &\eqdef \set{(0^+, y) \mid y \in \faces{}{-\a}x \setminus K} + 
            \set{(y) \mid y \in \faces{}{-\a}x \cap K}.
    \end{align*}
\end{itemize}
	Dually, if \( K \subseteq \bd{}{-}U \), the \emph{right-inverted partial Gray cylinder on \( U \) relative to \( K \)} is the oriented graded poset \( \rcyl{K}{U} \) whose
    \begin{itemize}
        \item underlying graded poset is the same as \( \arr \pcyl{K} U \);
        \item orientation is as in \( \arr \pcyl K U \), except for all \( x \in \gr{n}{U} \) and \( \a \in \set{+, -} \)
        \begin{align*}
            \faces{}{-}(1, x) &\eqdef \set{(1, y) \mid y \in \faces{}{+}x}, \\
            \faces{}{+}(1, x) &\eqdef \set{(0^-, x), (0^+, x)} + \set{(1, y) \mid y \in \faces{}{-}x \setminus K}, \\
            \faces{}{\a}(0^-, x) &\eqdef \set{(0^-, y) \mid y \in \faces{}{-\a}x \setminus K} + 
                \set{(y) \mid y \in \faces{}{-\a}x \cap K}.
        \end{align*}
    \end{itemize}
\end{dfn}

\begin{rmk} \label{rmk:inverted_cylinder_well_def}
	By \cite[Lemma 1.20, Lemma 1.26]{chanavat2024equivalences}, partial Gray cylinders and inverted partial Gray cylinders respect the classes of molecules, round molecules and atoms.
	Moreover, for all molecules \( U \) and closed subsets \( K \subseteq U \),
	\begin{itemize}
		\item \( \tau_K \colon \arr \pcyl K U \surj U \) is a cartesian map of molecules,
		\item if \( p \colon U \to V \) is a cartesian map of molecules with \( \dim V < \dim U \), then \( p \after \tau_K \colon \lcyl{K}{U} \to V \) and \( p \after \tau_K \colon \rcyl{K}{U} \to V \) are cartesian maps of molecules.
	\end{itemize}
\end{rmk}

\begin{dfn} [Higher invertor shapes] \label{dfn:higher_invertor_shape}
    Let \( U \) be a round molecule.
    The family of \emph{higher invertor shapes on \( U \)} is the family of molecules \( \hcyl \s U \) indexed by strings \( \s \in \set{L, R}^* \), defined inductively on the length of \( \s \) by
    \begin{align*}
        \hcyl{\langle\rangle} U & \eqdef U, \\
        \hcyl{L\s} U &\eqdef \lcyl{\bd{}{+} \hcyl{\s} U} (\hcyl{\s} U), \\
        \hcyl{R\s} U &\eqdef \dual{\dim U + |\s|}{\left(\rcyl{\bd{}{-} \hcyl{\s} U} (\hcyl{\s} U)\right)}.
    \end{align*}
	These are equipped with cartesian maps \( \tau_s \colon \hcyl{\s} U \to U \) of their underlying posets, with the property that for all cartesian maps of molecules \( p \colon U \to V \) such that \( \dim V < \dim U \), the composite \( p \after \tau_{\s} \) is a cartesian map of molecules.
\end{dfn}

\begin{lem} \label{lem:subdivision_of_unitors}
    Let \( s \colon U \sd V \) be a subdivision of molecules with formal dual \( c \colon V \to U \), and \( K \subseteq U \) be a closed subset of \( U \).
    Then the formal dual \( \arr \pcyl K s \colon \arr \pcyl K U \sd \arr \pcyl{s(K)} V \) of the order-preserving function \( \arr \pcyl{K} c \colon \arr \pcyl{s(K)} V \to \arr \pcyl{K} U \) is a subdivision of partial Gray cylinders.
\end{lem}
\begin{proof}
    This is a special case of \cite[Proposition 1.34]{chanavat2025semistrict}.
\end{proof}

\begin{lem} \label{lem:subdivision_of_invertors}
    Let \( s \colon U \sd V \) be a subdivision of atoms with formal dual \( c \colon V \to U \).
    Then
    \begin{enumerate}
        \item for all closed subsets \( K \subseteq \bd{}{+} U \), the formal dual \( \lcyl{K} s \colon \lcyl{K} U \sd \lcyl{s(K)} V \) of the order-preserving function \( \lcyl{K} c \colon \lcyl{s(K)} V \to \lcyl{K} U \) is a subdivision of left-inverted partial Gray cylinders, and dually
        \item for all closed subsets \( K \subseteq \bd{}{-} U \), the formal dual \( \rcyl{K} s \colon \rcyl{K} U \sd \rcyl{s(K)} V \) of the order-preserving function \( \rcyl{K} \colon \rcyl{s(K)} V \to \rcyl{K} U \) is a subdivision of right-inverted partial Gray cylinders.
    \end{enumerate}
\end{lem}
\begin{proof}
    We prove the first point, the second is entirely dual.
    Let \( n \eqdef \dim U \).
    By construction, \( \bd{}{-} \lcyl{K} U \) and \( \bd{}{+} \lcyl{K} U \) are, respectively, of the form
    \begin{equation*}
        U \cpsub{} (\arr \pcyl K \bd{}{+} U) \subcp{} \dual{n}{U} \text{ and } \arr \pcyl{K \cap \bd{}{-} U} \bd{}{-} U,
    \end{equation*}
    and similarly for \( \bd{}{-} \lcyl{K} V \) and \( \bd{}{+} \lcyl{K} V \).
    Since \( s(-) \) preserves pastings at submolecules and boundaries, we have by Lemma \ref{lem:subdivision_of_unitors} a clear subdivision of regular directed complexes 
    \begin{equation*}
        \restr{\lcyl{K} s}{\bd{}{} \lcyl{K} U} \colon \bd{}{} \lcyl{K} U \to \bd{}{} \lcyl{s(K)} V.
    \end{equation*}
    Since \( U \) and \( V \) are atoms, this extends to a subdivision \( \lcyl{K} s \colon \lcyl{K} U \to \lcyl{s(K)} V \).
\end{proof}

\begin{rmk}
    Given a subdivision \( s \colon U \sd V \) of atoms, and a string \( \s \in \set{L, R}^* \), we have, by inductively applying Lemma \ref{lem:subdivision_of_invertors}, a subdivision \( \hcyl{\s} s \colon \hcyl{\s} U \sd \hcyl{\s} V \).
\end{rmk}

\begin{dfn} [Unit]
    Let \( u \colon U \to X \) be a pasting diagram.
    The \emph{unit on \( u \)} is the degenerate pasting diagram \( \un u \colon u \celto u \) defined by \( u \after \tau_{\bd{}{} U} \colon \arr \pcyl{\bd{}{}U} U \to X \).
\end{dfn}

\begin{dfn} [Equivalence] 
    Let \( e \colon u \celto v \) be a round diagram in a diagrammatic set \( X \).
    We say that \( e \) is an \emph{equivalence} if there exist a round diagram \( e^* \colon v \celto u \), together with round diagrams \( h \colon e \cp{} e^* \celto \un(u) \) and \( h' \colon e^* \cp{} e \celto \un(v) \) such that \( h \) and \( h' \) are equivalences.
    We let
    \begin{equation*}
        \Eqv X \eqdef \set{e \in \Rd(X) \mid e \text{ is an equivalence}},\quad \eqv X = \Eqv X \cap \cell X.
    \end{equation*}
\end{dfn}

\begin{prop} \label{prop:main_equivalence}
    Let \( X \) be a diagrammatic set.
    Then
    \begin{enumerate}
        \item every degenerate round diagram is an equivalence;
        \item any two weak inverses of an equivalence are equivalent to each other;
        \item any round diagram equivalent to an equivalence is an equivalence.
    \end{enumerate}
    Furthermore, any morphism \( f \colon X \to Y \) of diagrammatic sets sends equivalences to equivalences.
\end{prop}
\begin{proof}
    See \cite[Theorem 2.13, Proposition 2.19, Corollary 2.29, Proposition 2.31]{chanavat2024equivalences}.
\end{proof}

\subsection{Model structure for diagrammatic \texorpdfstring{$(\infty, n)$}{(∞, n)}-categories}

Recall that a \emph{marked diagrammatic set} is a diagrammatic set \( X \) together with a set \( A \subset \gr{> 0}{\cell X} \) called the \emph{marked cell}, containing all the degenerate cells. 
A morphism of marked diagrammatic sets is a morphism of the underlying diagrammatic sets sending marked cells to marked cells.
We write \( \mdgmSet \) for the category of marked diagrammatic sets and their morphisms. 
Furthermore, if \( P \) is a regular directed complex and \( A \subseteq \gr{> 0}{P} \) is a subset of elements of \( P \) of dimension \( > 0 \), we let \( (P, A) \) be the marked diagrammatic set \( (P, \dgn P \cup \set{\mapel{a} \mid a \in A}) \).

\begin{dfn}
    The functor \( \fun{U} \colon \mdgmSet \to \dgmSet \) forgetting the marking of a marked diagrammatic set has a left adjoint \( \minmark{(-)} \) defined by \( \minmark{X} \eqdef (X, \dgn X) \).
    Given a diagrammatic set, we also let \( \natmark{X} \eqdef (X, \eqv X) \).
    By Proposition \ref{prop:main_equivalence}, this is well-defined and extends to a functor \( \natmark{(-)} \colon \dgmSet \to \mdgmSet \). 
\end{dfn}

\begin{dfn} [Marking]
   Let \( j \colon (X, A) \to (Y, B) \) be a morphism of marked diagrammatic sets. 
   We say that \( j \) is a \emph{marking} if \( \fun{U}j \) is an isomorphism.
\end{dfn}

\begin{dfn} \label{dfn:localisation}
    We recall the definition of \emph{localisation} adapted from \cite[Section 2.4]{chanavat2024model}.
    A \emph{cellular extension} of a diagrammatic set \( X \) is a pushout diagram
    \begin{center}
        \begin{tikzcd}
            {\coprod_{e \in \cls{S}} \bd{}{}U_e} &&& {\coprod_{u \in \cls{S}} U_e} \\
            X &&& {X_\cls{S}}
            \arrow["{(\bd{}{}e)_{e \in \cls{S}}}", from=1-1, to=2-1]
            \arrow["{(e)_{e \in \cls{S}}}", from=1-4, to=2-4]
            \arrow[hook, from=2-1, to=2-4]
            \arrow["{\coprod_{e \in \cls{S}}\bd{U_e}{}}", hook, from=1-1, to=1-4]
            \arrow["\lrcorner"{anchor=center, pos=0.125, rotate=180}, draw=none, from=2-4, to=1-1]
        \end{tikzcd}
    \end{center}
    in \( \dgmSet \) such that for each \( e \in \cls{S} \), \( U_e \) is an atom.
    We say that \( X_{\cls{S}} \) is the result of \emph{attaching the cells \( \set{e \colon e^- \celto e^+}_{e \in \cls{S}} \) to X}.

    Let \( (X, A) \) be a marked diagrammatic set.
    We define \( \preloc{X}{A} \) to be the diagrammatic set obtain, for each cell \( a \colon u \celto v \) in \( A \cap \nd X \), by
    \begin{enumerate}
        \item attaching cells \( a^L, a^R  \colon v \celto u \), then
        \item attaching cells \( \hinv{L}(a) \colon a \cp{} a^L \celto \un(u) \) and \( \hinv{R} \colon a^R \cp{} a \celto \un(v) \).
    \end{enumerate} 
    Then, let \( \order{0}{X} \eqdef X \) and \( \order{0}{A} \eqdef A \).
    Inductively on \( n > 0 \), define
    \begin{equation*}
        \order{n}{X} \eqdef \preloc{\order{n - 1}{X}}{\order{n - 1}{A}}, \quad\quad \order{n}{A} \eqdef \set{\hinv{R}(a), \hinv{L}(a) \mid a \in \order{n - 1}{A}}.
    \end{equation*}
    We then have a sequence of inclusions
    \begin{equation*}
        \order{0}{X} \incl \order{1}{X} \incl \ldots \incl \order{n}{X} \incl \ldots,
    \end{equation*}
    whose transfinite composition is \( \loc{X}{A} \), the \emph{localisation of \( X \) at \( A \)}, which comes equipped with a canonical inclusion \( X \incl \loc{X}{A} \).
    We may describe the non-degenerate cells of \( \loc{X}{A} \) as being either in the image of \( \nd X \setminus A \), or of the form \( \hinv{\s}a, (\hinv{\s}a)^L \) or \( (\hinv{\s}a)^R \) for a cell \( a \in (\nd X) \cap A \) and a string \( \s \in \set{L, R}^* \). 
    For \( a \in (\nd X) \cap A \), \( \hinv{\s}a \) is defined by letting \( \hinv{\langle\rangle} a \eqdef a \), and inductively on \( \s \in \set{L, R}^* \), by
    \begin{align*}
        \hinv{L\s}a \eqdef \hinv{L}(\hinv{\s}a) &\colon \hinv{\s}a \cp{} (\hinv{\s}a)^L \celto \un(\bd{}{-}\hinv{\s}a),\\
        \hinv{R\s}a \eqdef \hinv{R}(\hinv{\s}a) &\colon (\hinv{\s}a)^R \cp{} \hinv{\s}a \celto \un(\bd{}{+}\hinv{\s}a).
    \end{align*}
    By induction, if \( a \) is of shape \( U \), we have \( \hinv{\s}(a) \colon \hcyl{\s} U \to X \).
    
    Finally, the construction of the localisation can be extended to a colimit preserving functor
    \begin{equation*}
        \Loc \colon \mdgmSet \to \dgmSet.
    \end{equation*}
    By definition, each cell in \( A \) becomes an equivalence in \( \loc{X}{A} \).
\end{dfn}

\begin{comm} \label{comm:diff_of_localisation}
    This definition of localisation differs slightly from the one given in \cite[2.38]{chanavat2024model}. 
    In the latter, the attached cell \( \hinv{R} \) has type \( \un(v) \celto a^R \cp{} a \), whereas in this version, we define \( \hinv{R} \colon a^R \cp{} a \celto \un(v) \), which explain the appearance of a dual in Definition \ref{dfn:higher_invertor_shape}.
    We made this choice to match the construction of the coherent walking equivalence from \cite{hadzihasanovic2024model} in (\ref{dfn:strict_walking_equivalence}), so that Lemma \ref{lem:swE_is_iso_to_molecin_loc_globe} holds.
    The reader can be assured that this difference of definition is inessential and, in the sequel, we will always point out which results of \cite{chanavat2024model} would need a small adjustment to apply in the present case. 
\end{comm}

\begin{dfn} [Walking equivalence]
    Let \( U \) be an atom.
    The \emph{walking equivalence of shape \( U \)} is the diagrammatic set \( \selfloc{U} \eqdef \Loc (U, \set{\top_U}) \).
    We also write \( V \simeq W \) for the walking equivalence of shape \( V \celto W \), and let \( \rglobe{n + 1} \eqdef (\dglobe{n} \simeq \dglobe{n}) \) for all \( n \in \mathbb{N} \). 
\end{dfn}

\begin{dfn} [Weak composites]
    Let \( X \) be a diagrammatic set.
    We say that \( X \) has \emph{weak composites} if for each round diagram \( u \colon U \to X \), there exists a cell \( \compos{u} \colon \compos{U} \to X \) parallel to \( u \) such that \( u \simeq \compos{u} \).
    In that case, \( \compos{u} \) is called a weak composite of \( u \).
\end{dfn}

\begin{dfn} [\( (\infty, n) \)\nbd category] \label{dfn:infty_n_cat}
    Let \( n \in \mathbb{N} \cup \set{\omega} \), and \( X \) be a diagrammatic set.
    We say that \( X \) is an \( (\infty, n) \)\nbd category if:
    \begin{enumerate}
        \item \( X \) has weak composites, and
        \item all cells of dimension \( > n \) are equivalences.
    \end{enumerate}
    A morphisms of diagrammatic sets is called a \emph{functor} when its domain and codomain are \( (\infty, n) \)\nbd categories. 
\end{dfn}

\begin{rmk}
    In the case \( n = \omega \), the second condition is void.
\end{rmk}

\begin{comm}
    In \cite{chanavat2024model}, an \( (\infty, \omega) \)\nbd category was called an \( (\infty, \infty) \)\nbd category.
    In this article, we chose to update the notation in order to be more uniform with that of strict \( n \)\nbd category, where \( n \) ranges in \( \mathbb{N} \cup \set{\omega} \), and because overall, following \cite{loubaton2023theory}, we think it is a better notation. 
\end{comm}

\begin{dfn} [Marked horn]
    Let \( (U, A) \) be a marked atom with greatest element \( \top \in A \) such that \( k \eqdef \dim U - 1 \geq 0 \), \( \a \in \set{-, +} \), \( x \in \maxel{\bd{}{\a} U} \), and call \( \Lambda^x_U \eqdef U \setminus \set{\top_U, x} \).
    We say that the inclusion of marked regular directed complexes
    \begin{equation*}
        \lambda^x_U \colon (\Lambda^x_U, \Lambda^x_U \cap A) \incl (U, A) 
    \end{equation*}
    is a \emph{marked horn of \( U \)} if there exists molecules \( (\order{i}{L}, \order{i}{R})_{i = 1}^k \) such that
    \begin{enumerate}
        \item \( \bd{}{\a} U = \order{k}{L} \cp{k - 1} (\ldots \cp{1} \order{1}{L} \cp{0} \clset{x} \cp{0} \order{1}{R} \cp{1} \ldots) \cp{k - 1} \order{k}{R} \);
        \item \( \dim \order{i}{L}, \dim \order{i}{R} \le i \) for each \( 1 \le i \le k \);
        \item \( \gr{i}{\order{i}{L}} \cup \gr{i}{\order{i}{R}} \subseteq A \);
        \item \( x \in A \) if and only if \( \faces{}{-\a} U \subseteq A \).
    \end{enumerate}
    We let \( \Jhorn \) be the set of marked horns.
\end{dfn}

\begin{comm}
    Let \( \lambda^x_U \colon (\Lambda^x_U, \Lambda^x_U \cap A) \to (U, A) \) be a marked horn, and \( W \) be an \( (\infty, \omega) \)\nbd category.
    We recall from \cite[Comment 3.13]{chanavat2024model} that a morphism \( s \colon \Lambda^x_U \to W \) classifying in \( W \) an equation \( \fun{E}x \qeq v \) in the indeterminate \( x \), where \( v \eqdef \restr{s}{\bd{}{+}U} \), and \( \fun{E} \) is a context in the sense of \cite[3.1]{chanavat2024equivalences}.
    Now if \( s \) defines furthermore a morphisms of marked diagrammatic sets \( s \colon (\Lambda^x_U, \Lambda^x_U \cap A) \to \natmark{W} \), then this implies that the context \( \fun{E} \) is weakly invertible.
    By \cite[Lemma 5.10]{chanavat2024equivalences}, this equation has a solution \( u \), witnessed by an equivalence \( h \colon \fun{E}u \celto v \), which, up to passing to weak composites, is exactly the data of an extension of \( s \) along \( \lambda^x_U \).
\end{comm}

\begin{prop} \label{prop:model_structre_on_marked_dgm_set}
    For each \( n \in \mathbb{N} \cup \set{\omega} \) there exists a model structure on marked diagrammatic sets, called the \emph{coinductive \( (\infty, n) \)\nbd model structure}, where
    \begin{enumerate}
        \item cofibrations are the monomorphisms;
        \item fibrant objects are of the form \( \natmark{X} \), for \( X \) an \( (\infty, n) \)\nbd category.
    \end{enumerate}
\end{prop}
\begin{proof}
    See \cite[3.24, Theorem 4.9]{chanavat2024model}.
\end{proof}

\begin{dfn}
    Let \( U \) be a round molecule.
    The \emph{walking weak composite of \( U \)} is the inclusion of diagrammatic sets
    \begin{equation*}
        c_U \colon U \incl U \simeq \compos{U}.
    \end{equation*}
    We let \( \Jcomp \) be the set of walking weak composites.
\end{dfn}

\begin{rmk}
    Each walking weak composite is the localisation of a marked horn.
\end{rmk}

\begin{dfn}
    Let \( n \in \mathbb{N} \cup \set{\omega} \).
    We let \( \Jn{n} \eqdef \set{U \incl \selfloc{U} \mid U \text{ atom}, \dim U > n } \).
\end{dfn}

\begin{rmk} \label{rmk:infty_n_cat_iff_rlp_Jcomp_Jn}
    By a simple variation of \cite[Proposition 3.8, Proposition 3.9]{chanavat2024model} to account for Comment \ref{comm:diff_of_localisation}, a diagrammatic set \( X \) is an \( (\infty, n) \)\nbd category if and only if it has the right lifting property against \( \Jcomp \cup \Jn{n} \).
\end{rmk}

\begin{lem} \label{lem:isofib_left_right_lift}
    Let \( f \colon X \to Y \) be a functor of \( (\infty, \omega) \)\nbd categories with the right lifting property against \( \Jcomp \), \( u \) be a round diagram in \( X \) and \( v \) a cell in \( Y \) parallel to \( f(u) \).
    Then
    \begin{enumerate}
        \item for all cells \( h \colon f(u) \celto v \) such that \( h \) is an equivalence, there exists an equivalence \( z \colon u \celto v' \) such that \( f(z) = h \);
        \item for all cells \( h \colon v \celto f(u) \) such that \( h \) is an equivalence, there exists an equivalence \( z \colon v' \celto u \) such that \( f(z) = h \).  
    \end{enumerate}
\end{lem}
\begin{proof}
    Since \( f \) has the right lifting against \( \Jcomp \), the first part follows from a small variation of the proof of \cite[Proposition 3.7]{chanavat2024model} to account for Comment \ref{comm:diff_of_localisation}.
    We consider the second point.
    Suppose that the cell \( h \colon v \celto f(u) \) is an equivalence.
    Then \( h \) has a weak inverse \( h^* \colon f(u) \celto v \), which is itself an equivalence, and can be taken to be a cell since \( Y \) has weak composites.
    Thus by the first part of the proof, \( h^* = f(z^*) \), for an equivalence \( z^* \colon u \celto v' \).
    Then, \( z^* \) has a weak inverse \( z' \colon v' \celto u \). 
    Since morphisms of diagrammatic set preserve weak inverses, \( f(z') \) is a weak inverse of \( h^* \).
    By Proposition \ref{prop:main_equivalence}, there exists an equivalence \( k \colon f(z') \celto h \), which we may take to be a cell, since \( Y \) has weak composites.
    By the first part of the proof, \( k = f(k') \) for some equivalence \( k' \) of type \( z' \celto z \), where \( z \colon v' \celto u \) is an equivalence by Proposition \ref{prop:main_equivalence}.
    Since \( h = f(z) \), this concludes the proof.
\end{proof}

\begin{lem} \label{lem:isofib_rlp_marked_horn}
    Let \( f \colon X \to Y \) be a functor of \( (\infty, \omega) \) categories with the right lifting property against \( \Jcomp \),
    Then \( \natmark{f} \) has the right lifting property against \( \Jhorn \).
\end{lem}
\begin{proof}
    Let \( \lambda^x_U \colon (\Lambda^x_U, \Lambda^x_U \cap A) \to (U, A) \) be a marked horn and assume that \( x \in \bd{}{-} U \), the case \( x \in \bd{}{+} U \) is dual.
    Consider a lifting problem
    \begin{center}
        \begin{tikzcd}
            {(\Lambda^x_U, \Lambda^x_U \cap A)} & \natmark{X} \\
            {(U, A)} & {\natmark{Y}.}
            \arrow["s", from=1-1, to=1-2]
            \arrow[from=1-1, to=2-1]
            \arrow["\natmark{f}", from=1-2, to=2-2]
            \arrow["h"', from=2-1, to=2-2]
        \end{tikzcd}
    \end{center}
    By \cite[Theorem 4.9]{chanavat2024model}, both \( \natmark{X} \) and \( \natmark{Y} \) have the right lifting property against \( \Jhorn \).
    Then \( s \) classifies in \( X \) an equation \( \fun{E}x \qeq v \) in the unknown \( x \), where \( \fun{E} \) is a weakly invertible context. 
    Let us choose a solution \( u \) witnessed by an equivalence \( z \colon \fun{E}u \celto v \).
    Notice that, by assumption, \( h \colon (f\fun{E})w \celto f(v) \) is an equivalence in \( Y \).
    In \( Y \), we may form the equation 
    \begin{equation*}
        y \cpsub{f(u)} f(z) \qeq h,
    \end{equation*}
    in the unknown \( y \), which has a solution \( e \colon w \celto f(u) \) witnessed by an equivalence \( k \colon e \cpsub{} f(z) \celto h \).
    Since \( Y \) has weak composites, we may assume that \( e \) and \( w \) are cells.
    By Lemma \ref{lem:isofib_left_right_lift}, \( e = f(e') \) for some equivalence \( e' \colon w' \celto u \).
    Thus \( k \) has type \( f(e' \cpsub{} z) \celto h \).
    By Lemma \ref{lem:isofib_left_right_lift} again, \( k = f(k') \) for some equivalence \( k' \colon e' \cpsub{} z \celto h' \).
    In particular, \( h' \) is an equivalence of type \( \fun{E}w' \celto v \) and \( f(h') = h \).
    If \( \faces{}{+} U \subseteq A \), then \( v \) is an equivalence, hence by \cite[Theorem 5.22]{chanavat2024equivalences}, \( w' \) is an equivalence.
    This shows that \( h' \colon U \to X \) extends to a morphism \( (U, A) \to \natmark{X} \) solving the lifting problem.
    This concludes the proof.
\end{proof}

\begin{lem} \label{lem:marked_localisation_acyclic}
    Let \( n \in \mathbb{N} \cup \set{\omega} \) and \( U \) be an atom of dimension \( \geq 1 \).
    Then the marking \( \minmark{\selfloc{U}} \incl (\selfloc{U}, \dgn \selfloc{U} \cup \set{\iota \colon U \incl \selfloc{U}}) \) is an acyclic cofibration in the coinductive \( (\infty, n) \)\nbd model structure on marked diagrammatic sets.
\end{lem}
\begin{proof}
    Notice that a marking is an acyclic cofibration if and only if it has the left lifting property against all fibrant objects.
    Let \( X \) be an \( (\infty, n) \)\nbd category and consider a morphism \( u \colon \minmark{\selfloc{U}} \to \natmark{X} \).
    By construction, \( u \after \iota \colon U \to X \) is an equivalence in \( X \), thus is marked in \( \natmark{X} \).
    This shows that the lifting problem has a solution and concludes the proof.
\end{proof}

\noindent Recall that in a model structure, a \emph{pseudo-generating set of acyclic cofibrations} is a set of \( J \) of acyclic cofibrations such that a morphism with fibrant codomain is a fibration if and only if it has the right lifting property against \( J \).

\begin{thm} \label{thm:n_model_structure_on_dgm_set}
    Let \( n \in \mathbb{N} \cup \set{\omega} \).
    Then there exists a model structure on diagrammatic sets, called the \emph{\( (\infty, n) \)\nbd model structure}, such that
    \begin{enumerate}
        \item the set \( \set{\bd{}{} U \incl U \mid U \text{ atom}} \) is a generating set of cofibrations;
        \item fibrant objects are the \( (\infty, n) \)\nbd categories;
        \item \( \Jcomp \cup \Jn{n} \) is a pseudo-generating set of acyclic cofibrations.
    \end{enumerate}
    Furthermore, the adjunction \( \minmark{(-)} \dashv \fun{U} \) is a Quillen equivalence with the coinductive \( (\infty, n) \)\nbd model structure on marked diagrammatic sets.
\end{thm}
\begin{proof}
    The existence as well as the characterisation of fibrant objects and the Quillen equivalence are given by \cite[3.27, Theorem 4.21, Theorem 4.23]{chanavat2024model}.
    Using the two-out-of-three, a slight variation of \cite[Lemma 4.16]{chanavat2024model} to account for Comment \ref{comm:diff_of_localisation}, and Lemma \ref{lem:marked_localisation_acyclic}, one sees that the set \( \minmark{(\Jcomp \cup \Jn{n})} \) is a set of acyclic cofibrations.
    Since left Quillen equivalences reflect weak equivalences between cofibrant objects, \( \Jcomp \cup \Jn{n} \) is a set of acyclic cofibrations.
    By Remark \ref{rmk:infty_n_cat_iff_rlp_Jcomp_Jn}, \( \Jcomp \cup \Jn{n} \) detects fibrant objects.
    Let \( f \colon X \to Y \) be a functor of \( (\infty, n) \)\nbd categories with the right lifting property against \( \Jcomp \cup \Jn{n} \).
    It remains to prove that \( f \) is a fibration.
    Since \( \fun{U}\natmark{f} = f \), it is enough to prove that \( \natmark{f} \) is a fibration. 
    By Lemma \ref{lem:isofib_rlp_marked_horn}, \( \natmark{f} \) has the right lifting property against \( \Jhorn \).
    By \cite[Theorem 4.22]{chanavat2025gray}, a pseudo-generating set of acyclic cofibrations for the coinductive \( (\infty, n) \)\nbd model structure is given by \( \Jhorn \cup J_{\mathsf{mark}} \), where \( J_{\mathsf{mark}} \) is a set of markings.
    But a morphism of marked diagrammatic sets has the right lifting property against a set of markings \( J_{\mathsf{mark}} \) if and only if its domain has the right lifting property against \( J_{\mathsf{mark}} \).
    Since the domain of \( \natmark{f} \) is the fibrant object \( \natmark{X} \), we conclude.
\end{proof}

\begin{comm}
    In \cite{chanavat2024model}, a pseudo-generating set of acyclic cofibrations for the \( (\infty, n) \)\nbd model structure was given by a closure of \( \Jcomp \cup \Jn{n} \) under certain Gray products.
    We showed in Theorem \ref{thm:n_model_structure_on_dgm_set} that this closure is unnecessary.
\end{comm}

\section{Homotopy theory of stricter \texorpdfstring{$\omega$}{ω}-categories} \label{sec:model}

\subsection{Folk model structure on stricter \texorpdfstring{$\omega$}{ω}-categories}

\noindent Recall from \cite{lafont2010folk} the existence, for all \( n \in \set{\mathbb{N}} \cup \set{\omega} \), of the \emph{folk model structure} on strict \( n \)\nbd categories.

\begin{dfn} [Stricter \( \omega \)\nbd category of cylinders]
    Let \( C \) be a stricter \( \omega \)\nbd category.
    The \emph{stricter \( \omega \)\nbd category of cylinders} is the stricter \( \omega \)\nbd category 
    \begin{equation*}
       \Gamma(C) \eqdef \homlax(\globe{1}, C). 
    \end{equation*}
\end{dfn}

\begin{rmk} \label{rmk:strict_stricter_same_cylinders}
    By Proposition \ref{prop:reflection_to_stricter_monoidal}, the stricter category \( \Gamma(C) \) coincides with the strict \( \omega \)\nbd category of cylinders \cite[Remark 20.2.9]{ara2025polygraphs}, which happen to be stricter, since \( C \) is.
\end{rmk}

\begin{dfn}
    We write \( \Icof \) for the the set
    \begin{equation*}
        \Icof \eqdef \set{\molecin{\bd{}{}U} \incl \molecin{U} \mid U \text{ atom}}
    \end{equation*}
\end{dfn}

\begin{thm} \label{thm:folk_model_structure_on_stricter_n}
    Let \( n \in \mathbb{N} \cup \set{\omega} \).
    There is a cofibrantely generated model structure on the category \( \snCat{n} \), called the \emph{folk model structure}, where:
    \begin{enumerate}
        \item every \( n \)\nbd category is fibrant;
        \item \( \trunc{n}\Icof \) is a set of generating cofibrations.
    \end{enumerate}
    Furthermore, this model structure is right transferred from the folk model structure on \( \nCat{n} \) along the adjunction 
    \begin{center}
        \begin{tikzcd}
            \snCat{n} & \nCat{n},
            \arrow[""{name=0, anchor=center, inner sep=0}, "\iota"', curve={height=12pt}, hook, from=1-1, to=1-2]
            \arrow[""{name=1, anchor=center, inner sep=0}, "\rcs"', curve={height=12pt}, from=1-2, to=1-1]
            \arrow["\dashv"{anchor=center, rotate=-90}, draw=none, from=1, to=0]
        \end{tikzcd}
    \end{center}
    which is in particular a Quillen pair, and an equivalence if \( n \le 3 \).
\end{thm}
\begin{proof}
    This is a direct application of \cite[Proposition 21.3.2]{ara2025polygraphs} with Remark \ref{rmk:strict_stricter_same_cylinders} and Lemma \ref{lem:reflection_of_polygraphs_are_stricter_polygraphs}.
    That the set \( \trunc{n}\Icof \) is a generating set of cofibrations follows from Corollary \ref{cor:pushout_principal_cell}.
    If \( n \le 3 \), then by Theorem \ref{thm:strict_le_3_are_stricter}, the full subcategory inclusion \( \iota \) is the identity, hence \( \nCat{n} \) and \( \snCat{n} \) have the same cofibrations and fibrant objects, so the two folk model structures coincide. 
\end{proof}

\noindent Therefore, we have the following commutative square of left Quillen functors
\begin{center}
    \begin{tikzcd}
        \omegaCat & \somegaCat \\
        {\nCat{n}} & {\snCat{n}.}
        \arrow["\rcs", from=1-1, to=1-2]
        \arrow["{\trunc{n}}"', from=1-1, to=2-1]
        \arrow["{\trunc{n}}", from=1-2, to=2-2]
        \arrow["\rcs"', from=2-1, to=2-2]
    \end{tikzcd}
\end{center}

\begin{rmk}\label{rmk:also_right_transferred_n_folk}
    By construction, the folk model structure on \( \nCat{n} \) is right transferred along the full subcategory inclusion \( \nCat{n} \incl \omegaCat \).
    Thus, the same applies for the folk model structure on \( \snCat{n} \) and the full subcategory inclusion \( \snCat{n} \incl \somegaCat \).
\end{rmk}

\subsection{Coherent walking equivalence}

We let \( \molecin{-} \colon \dgmSet \to \somegaCat \) be the left Kan extension along the Yoneda embedding of the functor \( \atom \to \somegaCat \) defined by \( U \mapsto \molecin{U} \).
Notice that by Corollary \ref{cor:regular_directed_complex_colimit_of_itself}, there is no ambiguity when one writes \( \molecin{P} \) for a regular directed complex \( P \).

\begin{lem} \label{lem:molecin_preserves_cofibration}
    Let \( f \colon X \incl Y \) be a monomorphism of diagrammatic sets.
    Then \( \molecin{f} \) is a relative stricter polygraph.
\end{lem}
\begin{proof}
    Let \( I \eqdef \set{\bd{}{} U \incl U \mid U \text{ atom}} \).
    By \cite[Remark 2.9]{chanavat2024htpy}, \( f \) can be constructed as a pushout of transfinite composition of elements of \( I \).
    Since \( \molecin{-} \) is left adjoint, the same holds of \( \molecin{f} \), that is, \( \molecin{f} \) is a relative stricter polygraph.
\end{proof}

\begin{cor} \label{cor:molecin_polygraph_with_basis}
    Let \( X \) be a diagrammatic set. 
    Then \( \molecin{X} \) is a stricter polygraph with basis \( \cls{S} = \coprod_{k \geq 0} \cls{S}_k \), where
    \begin{equation*}
        \cls{S}_k \eqdef \set{\molecin{u} \colon \molecin{U} \to \molecin{X} \mid u \colon U \to X \in \gr{k}{\nd X}}.
    \end{equation*}
\end{cor}
\begin{proof}
    By Lemma \ref{lem:molecin_preserves_cofibration} and Lemma \ref{lem:stricter_polygraph_basis}.
\end{proof}

\begin{lem} \label{lem:molecin_monoidal}
    The functor \( \molecin{-} \colon \dgmSet \to \somegaCat \) is monoidal with respect to the Gray product of diagrammatic sets and stricter \( \omega \)\nbd categories.
\end{lem}
\begin{proof}
    The result is true on atoms by definition of the Gray product of stricter \( \omega \)\nbd categories.
    We conclude by universal property of Day convolution.
\end{proof}

\begin{dfn} \label{dfn:strict_walking_equivalence}
    We recall from \cite{hadzihasanovic2024model} the construction of the \emph{coherent walking equivalence \( \wE \)}, which is a strict \( \omega \)\nbd category equipped with an inclusion \( \globe{0} \incl \wE \).
    We define \( \wE \) as a transfinite composition of inclusions \( \order{n}{\wE} \incl \order{n + 1}{\wE} \) such that
    \begin{itemize}
        \item \( \order{n}{\wE} \) is a strict \( n \)\nbd category, and
        \item for \( n \geq 1 \), \( \order{n}{\wE} \) is equipped with strict functors \( \iota_n \colon \order{n - 1}{\wE} \to \order{n}{\wE} \),
    as well as \( \fun{L}_k, \fun{R}_k \colon \sus{\order{n - 1}{\wE}} \to \wE \).
    \end{itemize}
    Then, \( \order{0}{\wE} \) is the set with two elements \( \set{x, y} \) and \( \order{1}{\wE} \) is the free category on three generators \( \set{a \colon x \to y, a^L \colon y \to x, a^R \colon y \to x} \), equipped with the evident inclusion \( \iota_1 \colon \set{x, y} \incl \order{1}{\wE} \), and
    \begin{equation*}
        \fun{L}_1 \colon \sus{x} \mapsto a \comp{0} a^L, \sus{y} \mapsto x, \quad\text{ and }\quad \fun{R}_1 \colon \sus{x} \mapsto a^R \comp{0} a, \sus{y} \mapsto y.
    \end{equation*}

    Let \( k > 1 \), and suppose that \( (\order{k - 1}{\wE}, \iota_{k - 1}, \fun{L}_{k - 1}, \fun{R}_{k - 1}) \) have been defined.
    Then \( (\order{k}{\wE}, \iota_{k}, \fun{L}_{k}, \fun{R}_{k}) \) is defined by the pushout
    \begin{center}
        \begin{tikzcd}[column sep=huge]
            {\sus{\order{k - 2}{\wE}} \coprod \sus{\order{k - 2}{\wE}}} & {\order{k - 1}{\wE}} \\
            {\sus{\order{k - 1}{\wE}} \coprod \sus{\order{k - 1}{\wE}}} & {\order{k}{\wE}}
            \arrow[""{name=0, anchor=center, inner sep=0}, "{(\fun{L}_{k-1},\fun{R}_{k-1})}", from=1-1, to=1-2]
            \arrow["{\sus{\iota_{k-1}} \coprod \sus{\iota_{k-1}}}"', from=1-1, to=2-1]
            \arrow["{\iota_k}", from=1-2, to=2-2]
            \arrow["{(\fun{L}_k,\fun{R}_k)}"', from=2-1, to=2-2]
            \arrow["\lrcorner"{anchor=center, pos=0.125, rotate=180}, draw=none, from=2-2, to=0]
        \end{tikzcd}
    \end{center}
    in \( \omegaCat \).
    The inclusion \( \globe{0} \incl \wE \) then classifies \( x \in \order{0}{\wE} \).
    We denote by \( \fun{L}_\infty, \fun{R}_\infty \colon \sus{\wE} \to \wE \) the transfinite composition of \( (\fun{L}_k)_k \) and \( (\fun{R}_k)_k \) respectively.
\end{dfn}

\begin{dfn} [Stricter coherent walking equivalence]
    The \emph{stricter coherent walking equivalence} is the stricter \( \omega \)\nbd category \( \swE \eqdef \rcs \wE \), which comes equipped with the inclusion \( j \colon \globe{0} \incl \swE \).
\end{dfn}

\begin{lem} \label{lem:inclusion_into_stricter_is_equivalence}
    The inclusion \( j \colon \globe{0} \incl \swE \) is an acyclic cofibration in the folk model structure on stricter \( \omega \)\nbd categories.
\end{lem}
\begin{proof}
    Let \( i \colon \dglobe{0} \incl \wE \) be the inclusion of strict \( \omega \)\nbd categories such that \( j = \rcs i \).
    By \cite[Remark 1.29, Theorem 1.33]{hadzihasanovic2024model} and the two-out-of-three, \( i \) a cofibration and a weak equivalence.
    By Theorem \ref{thm:folk_model_structure_on_stricter_n}, \( j \) is an acyclic cofibration. 
\end{proof}

\begin{lem} \label{lem:swE_is_iso_to_molecin_loc_globe}
    The stricter \( \omega \)\nbd categories \( \swE \) and \( \molecin{\rglobe{1}} \) are isomorphic.
\end{lem}
\begin{proof}
    By \cite[Remark 1.29]{hadzihasanovic2024model}, \( \swE \) is the stricter polygraph whose \( k \)\nbd globular cells are
    generated by the set \( \cls{S}_k \), which can be describe inductively as
    \begin{equation*}
        \cls{S}_0 = \set{x, y}, \cls{S}_1 = \set{a, a^L, a^R},
    \end{equation*} 
    and for \( k > 1 \),
    \begin{equation*}
        \cls{S}_k = \set{\fun{L}_\infty(\sus{u}), \fun{R}_\infty(\sus{u}) \mid u \in \cls{S}_{k - 1}}.
    \end{equation*}
    We canonically identify elements of \( \cls{S}_k \) with
    \begin{equation*}
        \cls{S}_k \cong \set{(a, \s), (a^L, \s), (a^R, \s) \mid \s \in \set{L, R}^*},
    \end{equation*}
    by letting \( (a, \langle\rangle) \eqdef a, (a^L, \langle\rangle) \eqdef a^L \), and \( (a^R, \langle\rangle) \eqdef a^R \), then inductively on \( \s \in \set{L, R}^* \), we let 
    \begin{align*}
        (a, L\s)   & \eqdef \fun{L}_\infty(\sus{(a, \s)}),  \\
        (a^L, L\s) & \eqdef \fun{L}_\infty(\sus{(a^L, \s)}), \\
        (a^R, L\s) & \eqdef \fun{L}_\infty(\sus{(a^R, \s)}),
    \end{align*}
    and similarly using \( \fun{R}_\infty \) for the case \( R\s \).    
    By Corollary \ref{cor:molecin_polygraph_with_basis}, \( \molecin{\rglobe{1}} \) is the stricter polygraph whose basis is given, in the notation of (\ref{dfn:localisation}), by \( \coprod_{k \geq 0} \cls{T}_k \), where, letting \( b \colon \dglobe{1} \to \dglobe{1} \) be the identity, we have
    \begin{equation*}
        \cls{T}_0 = \set{0^-, 0^+}, \cls{T}_1 = \set{b, b^L, b^R},
    \end{equation*}
    and for \( k > 1 \),
    \begin{equation*}
        \cls{T}_k = \set{\hinv{\s}b, (\hinv{\s}b)^L, (\hinv{\s}b)^R \mid \s \in \set{L, R}^{k - 1}}.
    \end{equation*}
    We have a family of bijections \( \set{\phi_k \colon \cls{T}_k \cong \cls{S}_k}_{k \geq 0} \) given by letting \( \phi_0 \) map \( (0^-, 0^+) \) to \( (x, y) \) and, for \( k > 0 \), by letting \( \phi_k \) be defined by
    \begin{equation*}
        \hinv{\s}b \mapsto (b, \s),\quad (\hinv{\s}b)^L \mapsto (b^L, \s),\quad (\hinv{\s}b)^R \mapsto (b^R, \s).
    \end{equation*}
    One checks that \( \phi \colon \molecin{\rglobe{1}} \to \swE \) induces a strict functor, which is therefore an isomorphism of stricter polygraphs.
    This concludes the proof.
\end{proof}

\begin{lem} \label{lem:suspension_commute_selfloc}
    Let \( U \) be an atom of dimension \( \geq 1 \).
    Then \( \sus{\selfloc{U}} \) and \( \selfloc{\sus{U}} \) are isomorphic. 
\end{lem}
\begin{proof}
    By Remark \ref{rmk:ez_connected}, \( U \) is connected.
    By induction on the construction of the localisation, so is \( \selfloc{U} \).
    Therefore, by Proposition \ref{prop:suspension_of_dgmSet}, 
    \begin{equation*}
        \sus{\left(\colim_{v \colon V \to \selfloc{U}} V\right)} \cong \colim_{v \colon V \to \selfloc{U}} \sus{V}
    \end{equation*}
    Then, for all \( \s \in \set{L, R}^* \), an inspection shows that \( \sus{\hcyl{\s}U} \) and \( \hcyl{\s} \sus{U} \) are isomorphic.
    By induction on the construction of the localisation, \( \selfloc{\sus{U}} \) and \( \colim\limits_{v \colon V \to \selfloc{U}} \sus{V} \) are isomorphic.    
    This concludes the proof.
\end{proof}

\begin{lem} \label{lem:sus_commute_molec_connected}
    Let \( X \) be a connected diagrammatic set.
    Then \( \molecin{\sus{X}} \) and \( \sus{\molecin{X}} \) are isomorphic.
\end{lem}
\begin{proof}
    This is the case when \( X \) is an atom. 
    We conclude by Corollary \ref{cor:adjunction_hom_suspension} and Proposition \ref{prop:suspension_of_dgmSet}.
\end{proof}

\begin{prop} \label{prop:walking_eq_of_dim_n}
    Let \( n \in \mathbb{N} \) and call \( i_n \colon \dglobe{n + 1} \incl \rglobe{n + 1} \) the canonical inclusion.
    Then 
    \begin{equation*}
        \molecin{(\bd{}{-} i_n)} \colon \molecin{\dglobe{n}} \to \molecin{\rglobe{n + 1}} 
    \end{equation*}
    is an acyclic cofibration in the folk model structure on stricter \( \omega \)\nbd categories.
\end{prop}
\begin{proof}
    By Lemma \ref{lem:inclusion_into_stricter_is_equivalence} and Lemma \ref{lem:swE_is_iso_to_molecin_loc_globe}, \( \molecin{(\bd{}{-} i_0)} \) is an acyclic cofibration.
    By a straightforward variation of \cite[Proposition 2.8]{hadzihasanovic2024model}, the functor \( \sus{} \colon \somegaCat \to \somegaCat \) preserves acyclic cofibrations, thus \( \sus{}\cdots\sus{\molecin{(\bd{}{-} i_0)}} \), where \( \sus{} \) is applied \( n \) times, is an acyclic cofibration.
    Using Lemma \ref{lem:suspension_commute_selfloc} and Lemma \ref{lem:sus_commute_molec_connected}, one sees that the latter is in fact isomorphic to \( \molecin{(\bd{}{-} i_n)} \).
    This concludes the proof.
\end{proof}

\subsection{Right transfer from the diagrammatic model structures}

\begin{lem} \label{lem:pushout_with_localisation}
    Let \( s \colon U \sd V \) be a subdivision between atoms of dimension \( \geq 1 \).
    Then there is a strict functor \( \tilde{s} \colon \molecin{\selfloc{U}} \to \molecin{\selfloc V} \) fitting in a pushout diagram
    \begin{center}
        \begin{tikzcd}
            {\molecin{U}} & {\molecin{\selfloc{U}}} \\
            {\molecin{V}} & {\molecin{\selfloc V},}
            \arrow[""{name=0, anchor=center, inner sep=0}, hook, from=1-1, to=1-2]
            \arrow["{\molecin{s}}"', from=1-1, to=2-1]
            \arrow["{\tilde s}", from=1-2, to=2-2]
            \arrow[hook, from=2-1, to=2-2]
            \arrow["\lrcorner"{anchor=center, pos=0.125, rotate=180}, draw=none, from=2-2, to=0]
        \end{tikzcd}
    \end{center}
    in \( \somegaCat \).
\end{lem}
\begin{proof}
    Let \( n \geq 1 \) be the dimension of \( U \), \( u \colon U \to U \) and \( v \colon V \to V \) be the identity.
    For each \( k \geq 0 \), let \( \order{k}{U} \) and \( \order{k}{V} \) be the marked diagrammatic sets corresponding respectively to the \( k \)\nbd th step of the localisation as in (\ref{dfn:localisation}).
    Note that we have
    \begin{align*}
        \nd \gr{n + k}{\order{k}{U}} &= \set{\hinv{\s}u \colon \hcyl{\s} U \to \order{k - 1}{U}, \hinv{\s}^Lu, \hinv{\s}^Ru \colon \dual{n + k}{\hcyl{\s} U} \to \order{k - 1}{U}}_{\s \in \set{L, R}^k},\\
        \nd \gr{n + k}{\order{k}{V}} &= \set{\hinv{\s}v \colon \hcyl{\s} V \to \order{k - 1}{V}, \hinv{\s}^Lv, \hinv{\s}^Rv \colon \dual{n + k}{\hcyl{\s} V} \to \order{k - 1}{V}}_{\s \in \set{L, R}^k}.
    \end{align*}
    We show by induction on \( k \geq 0 \) that the mapping \( s_k \colon \molecin{\order{k}{U}} \to \molecin{\order{k}{V}} \) defined, for all \( \s \in \set{L, R}^k \), by
    \begin{equation*}
        \hinv{\s}u \mapsto \hinv{\s}v,\quad \hinv{\s}^Lu, \mapsto \hinv{\s}^Lv,\quad \hinv{\s}^Ru, \mapsto \hinv{\s}^Rv,
    \end{equation*}
    is a well-defined strict functor fitting in a pushout square
    \begin{equation} \label{tik:inductive_square_subdivision_localisation}
        \begin{tikzcd}
            {\molecin U} & {\molecin{\order k U}} \\
            {\molecin V} & {\molecin {\order k V}.}
            \arrow[""{name=0, anchor=center, inner sep=0}, from=1-1, to=1-2]
            \arrow["{s_0}"', from=1-1, to=2-1]
            \arrow["{s_k}", from=1-2, to=2-2]
            \arrow[from=2-1, to=2-2]
            \arrow["\lrcorner"{anchor=center, pos=0.125, rotate=180}, draw=none, from=2-2, to=0]
        \end{tikzcd}
    \end{equation}
    The base case is \( s_0 \eqdef \molecin{s} \), which is indeed a strict functor making the previous square trivially a pushout.
    Inductively, let \( k > 0 \) and suppose that the construction is well-defined up to \( k - 1 \).
    To make the following diagrams more readable, we omit writing \( \molecin{(-)} \).
    Then \( {\molecin{\order k U}} \) and \( {\molecin{\order k V}} \) are produced from \( {\molecin{\order {k - 1} U}} \) and \( {\molecin{\order {k - 1} V}} \) by first attaching left and right inverses to the elements of the form \( \hinv{\s}u \) and \( \hinv{\s}v \) for \( \s \in \set{L, R}^{k - 1} \), giving the following pushouts in \( \somegaCat \)
    \begin{center}
        \begin{tikzcd}
            {\coprod\limits_{\s \in \set{L, R}^{k - 1}} \bd{}{}\dual{n + k - 1}{\hcyl{\s}U}} & {\coprod\limits_{\s \in \set{L, R}^{k - 1}} \dual{n + k - 1}{\hcyl{\s}U}} \\
            {\order{k - 1} U} & {U'}
            \arrow[""{name=0, anchor=center, inner sep=0}, from=1-1, to=1-2]
            \arrow[from=1-1, to=2-1]
            \arrow[from=1-2, to=2-2]
            \arrow[from=2-1, to=2-2]
            \arrow["\lrcorner"{anchor=center, pos=0.125, rotate=180}, draw=none, from=2-2, to=0]
        \end{tikzcd}
    \end{center}
    and 
    \begin{center}
        \begin{tikzcd}
            {\coprod\limits_{\s \in \set{L, R}^{k - 1}} \bd{}{}\dual{n + k - 1}{\hcyl{\s}V}} & {\coprod\limits_{\s \in \set{L, R}^{k - 1}} \dual{n + k - 1}{\hcyl{\s}V}} \\
            {\order{k - 1} V} & {V'}
            \arrow[""{name=0, anchor=center, inner sep=0}, from=1-1, to=1-2]
            \arrow[from=1-1, to=2-1]
            \arrow[from=1-2, to=2-2]
            \arrow[from=2-1, to=2-2]
            \arrow["\lrcorner"{anchor=center, pos=0.125, rotate=180}, draw=none, from=2-2, to=0]
        \end{tikzcd}
    \end{center}
    By Lemma \ref{lem:subdivision_of_invertors} and Corollary \ref{cor:pushout_principal_cell}, we have, in \( \somegaCat \), the pushout square
    \begin{center}
        \begin{tikzcd}
            {\coprod\limits_{\s \in \set{L, R}^{k - 1}} \bd{}{}\dual{n + k - 1}{\hcyl{\s}U}} & {\coprod\limits_{\s \in \set{L, R}^{k - 1}} \dual{n + k - 1}{\hcyl{\s}U}} \\
            {\coprod\limits_{\s \in \set{L, R}^{k - 1}} \bd{}{}\dual{n + k - 1}{\hcyl{\s}V}} & {\coprod\limits_{\s \in \set{L, R}^{k - 1}} \dual{n + k - 1}{\hcyl{\s}V}}
            \arrow[""{name=0, anchor=center, inner sep=0}, from=1-1, to=1-2]
            \arrow[from=1-1, to=2-1]
            \arrow["{\coprod \dual{n + k - 1}{\hcyl{\s}c}}", from=1-2, to=2-2]
            \arrow[from=2-1, to=2-2]
            \arrow["\lrcorner"{anchor=center, pos=0.125, rotate=180}, draw=none, from=2-2, to=0]
        \end{tikzcd}
    \end{center}
    Combined with the inductive pushout square (\ref{tik:inductive_square_subdivision_localisation}), the pasting law for pushouts gives a strict functor \( s' \colon \molecin{U'} \to \molecin{V'} \) being the pushout of \( s_k \) along \( \molecin{\order{k - 1}{U}} \to \molecin{U'} \).
    Finally, \( \molecin{\order k U} \) and \(  {\molecin{\order k V}} \) are given by the following cellular extensions of \( U' \) and \( V' \):
    \begin{center}
        \begin{tikzcd}[column sep=small]
            {\coprod\limits_{\s \in \set{L, R}^k} \bd{}{}\hcyl{\s}U} & {\coprod\limits_{\s \in \set{L, R}^k} \hcyl{\s}U} & {\coprod\limits_{\s \in \set{L, R}^k} \bd{}{}\hcyl{\s}V} & {\coprod\limits_{\s \in \set{L, R}^k} \hcyl{\s}V} \\
            {U'} & {\order k U} & {V'} & {\order k V}.
            \arrow[""{name=0, anchor=center, inner sep=0}, from=1-1, to=1-2]
            \arrow[from=1-1, to=2-1]
            \arrow[from=1-2, to=2-2]
            \arrow[""{name=1, anchor=center, inner sep=0}, from=1-3, to=1-4]
            \arrow[from=1-3, to=2-3]
            \arrow[from=1-4, to=2-4]
            \arrow[from=2-1, to=2-2]
            \arrow[from=2-3, to=2-4]
            \arrow["\lrcorner"{anchor=center, pos=0.125, rotate=180}, draw=none, from=2-2, to=0]
            \arrow["\lrcorner"{anchor=center, pos=0.125, rotate=180}, draw=none, from=2-4, to=1]
        \end{tikzcd}
    \end{center}
    By a similar argument, the pushout of \( s' \colon  \molecin{U'} \to \molecin{V'} \) along the strict functor \( \molecin{U'} \to \molecin{\order{k}{U}} \) is \( s_k \colon \molecin{\order{k}{U}} \to \molecin{\order{k}{V}} \).
    The pasting law for pushouts concludes the inductive step.
    By transfinite composition \( \tilde{s} \colon \molecin{\selfloc{U}} \to \molecin{\selfloc{V}} \) is the pushout of \( \molecin{s} \) along \( \molecin{U} \to \molecin{\selfloc U} \).
    This concludes the proof.
\end{proof}

\begin{prop} \label{prop:molecin_send_Jcomp_to_acof}
    Let \( U \) be a round molecule.
    Then
    \begin{equation*}
        \molecin{c_U} \colon \molecin{U} \to \molecin{(U \simeq \compos{U})}
    \end{equation*}
    is an acyclic cofibration in the folk model structure for stricter \( \omega \)\nbd categories.
\end{prop}
\begin{proof}
    By Lemma \ref{lem:molecin_preserves_cofibration}, \( \molecin{c_U} \) is a cofibration.
    Let \( n \eqdef \dim U \) and \( s \colon \dglobe{n} \sd U \) be the unique subdivision.
    By Lemma \ref{lem:pushout_with_localisation}, we have the two commutative squares
    \begin{center}
        \begin{tikzcd}
            {\molecin{\dglobe n}} & {\molecin{\dglobe {n + 1}}} & {\molecin{\rglobe {n + 1}}} \\
            {\molecin U} & {\molecin{(U \celto \compos U)}} & {\molecin{(U \simeq \compos U)}}
            \arrow[""{name=0, anchor=center, inner sep=0}, from=1-1, to=1-2]
            \arrow["{{\molecin{s}}}"', from=1-1, to=2-1]
            \arrow[""{name=1, anchor=center, inner sep=0}, from=1-2, to=1-3]
            \arrow[from=1-2, to=2-2]
            \arrow[from=1-3, to=2-3]
            \arrow[from=2-1, to=2-2]
            \arrow[from=2-2, to=2-3]
            \arrow["\lrcorner"{anchor=center, pos=0.125, rotate=180}, draw=none, from=2-3, to=1]
        \end{tikzcd}
    \end{center}
    where the right square is a pushout.
    By Lemma \ref{lem:generalised_substitution}, the left square is also a pushout.
    By the pasting law for pushouts, the outer square is a pushout.
    Since composite of the top horizontal strict functors is an acyclic cofibration by Proposition \ref{prop:walking_eq_of_dim_n}, we conclude that the composites of the bottom horizontal strict functors, which is \( \molecin{c_u} \), is an acyclic cofibration.
    This concludes the proof.
\end{proof}

\begin{dfn} [\( n \)\nbd truncated diagrammatic nerve]
    Let \( n \in \mathbb{N} \cup \set{\omega} \). 
    The \( n \)\nbd truncated diagrammatic nerve is the right adjoint \( \N{n} \) to the functor \( \trunc{n} \after \molecin{-} \).
    If \( n = \omega \), \( \trunc{\omega} \) is the identity, and we call \( \N{} \eqdef \N{\omega} \) the \emph{diagrammatic nerve}.
\end{dfn}

\begin{rmk}
    For all stricter \( n \)\nbd categories \( C \), \( \N{n} C \) is equal to \( \N{} C \), where in the latter expression, the stricter \( n \)\nbd category \( C \) is seen as a stricter \( \omega \)\nbd category.
\end{rmk}

\begin{prop} \label{prop:quillen_folk_dgm_infty}
    The folk model structure on stricter \( \omega \)\nbd categories is right transferred from the \( (\infty, \omega) \)\nbd model structure on \( \dgmSet \) along the diagrammatic nerve \( \N{} \colon \somegaCat \to \dgmSet \).
    In particular, the adjunction 
    \begin{equation*}
        \molecin{-} \colon \dgmSet \leftrightarrows \somegaCat \cocolon \N{}
    \end{equation*}
    is a Quillen adjunction.
\end{prop}
\begin{proof}
    Recall from Theorem \ref{thm:n_model_structure_on_dgm_set} that \( I \eqdef \set{\bd{}{} U \incl U \mid U \text{ atom}} \) is a generating set of cofibrations for the \( (\infty, \omega) \)\nbd model structure on diagrammatic sets, while \( \Jcomp \) is a pseudo-generating set of acyclic cofibration. 
    By Lemma \ref{lem:molecin_preserves_cofibration}, Proposition \ref{prop:molecin_send_Jcomp_to_acof}, \( \molecin{-} \) preserves cofibrations and a pseudo-generating set of acyclic cofibrations.
    By \cite[E.2.14]{joyal2008theory}, \( \molecin{-} \) is left Quillen. 
    Since all objects are fibrant in the folk model structure, Ken Brown's Lemma implies that \( \N{} \) preserves all weak equivalences, and in particular acyclic cofibrations. 
    Thus, \( \N{} \) takes transfinite composition of pushouts of elements of \( \molecin{\Jcomp} \) to weak equivalences.
    By \cite[Theorem 11.3.2]{hirschhorn2003model}, this is enough to show that there exists a model structure on \( \somegaCat \), let us call it the transferred model structure, cofibrantely generated by \( \molecin{I} \) and \( \molecin{\Jcomp} \).
    Since \( \molecin{I} \) is also a set of generating cofibrations for the folk model structure by Theorem \ref{thm:folk_model_structure_on_stricter_n}, the transferred and the folk model structure have the same cofibrations.
    Since all stricter \( \omega \)\nbd categories are fibrant in the folk model structure, they have in particular the right lifting property against \( \molecin{\Jcomp} \).
    Thus all stricter \( \omega \)\nbd categories are fibrant in the transferred model structure.
    By \cite[Proposition E.1.10]{joyal2008theory}, the transferred and the folk model structure coincide.
    This concludes the proof.
\end{proof}

\begin{thm}\label{thm:quillen_folk_dgm_n}
    Let \( n \in \mathbb{N} \cup \set{\omega} \).
    Then the folk model structure on stricter \( n \)\nbd categories is right transferred from the \( (\infty, n) \)\nbd model structure on \( \dgmSet \) along the \( n \)\nbd truncated diagrammatic nerve \( \N{n} \colon \snCat{n} \to \dgmSet \).
    In particular, the adjunction 
    \begin{equation*}
        \molecin{-} \colon \dgmSet \leftrightarrows \snCat{n} \cocolon \N{n}
    \end{equation*}
    is a Quillen adjunction.
\end{thm}
\begin{proof}
    The case \( n = \omega \) is Proposition \ref{prop:quillen_folk_dgm_infty}.
    Let \( n \in \mathbb{N} \).
    By Remark \ref{rmk:also_right_transferred_n_folk} the folk model structure on stricter \( n \)\nbd categories is right transferred along the right adjoint inclusion \( \snCat{n} \incl \somegaCat \), and the functor \( \trunc{n} \colon \somegaCat \to \snCat{n} \) is left Quillen.
    Let \( U \) be an atom of dimension \( k > n \), and \( i \colon U \incl \selfloc{U} \) be the canonical inclusion.
    A direct inspection and Corollary \ref{cor:molecin_polygraph_with_basis} show that \( \trunc{n}\molecin{i} \) is an isomorphism.
    Thus \( \trunc{n} \after \molecin{(\Jcomp \cup \Jn{n})} \) is a set of acyclic cofibrations.
    By \cite[E.2.14]{joyal2008theory}, \( \trunc{n} \after \molecin{-} \) is left Quillen. 
    Since ``being right transferred'' is compatible with composition of right Quillen functors, we conclude by Proposition \ref{prop:quillen_folk_dgm_infty}.
\end{proof}

\noindent By Day convolution, diagrammatic sets support a biclosed monoidal structure \( - \gray - \) which restricts to the Gray product on atoms.

\begin{prop} \label{prop:Gray_monoidal}
    Let \( n \in \mathbb{N} \cup \set{\omega} \).
    Then the folk model structure on stricter \( n \)\nbd categories is monoidal with respect to the Gray product.
    Furthermore, the functors 
    \begin{equation*}
         \molecin{-} \colon \dgmSet \to \snCat{n} \quad\text{ and }\quad \rcs \colon \nCat{n} \to \snCat{n}
    \end{equation*}
    are strong monoidal left Quillen functors for the Gray products, when \( \dgmSet \) is equipped with the \( (\infty, n) \)\nbd model structure, and \( \nCat{n} \) with the folk model structure.
\end{prop}
\begin{proof}
    The functor \( \rcs \) is strong monoidal by Proposition \ref{prop:reflection_to_stricter_monoidal}, and \( \molecin{-} \) is monoidal by definition and universal property of Day convolution.
    Since the folk model structure is right transferred along the right adjoint of \( \rcs \) by Theorem \ref{thm:folk_model_structure_on_stricter_n} and also along the right adjoint of \( \molecin{-} \) by Theorem \ref{thm:quillen_folk_dgm_n}, we conclude with a formal argument using either \cite[Theorem 5.6]{ara2020monoidal} or \cite[Theorem 5.10]{chanavat2025gray}.
\end{proof}

\bibliographystyle{alpha}
\small\bibliography{main.bib}

\end{document}